\documentclass[a4paper,11pt]{amsart} 
\usepackage{amsmath,amssymb,amsfonts,amsthm}
\usepackage[abbrev,nobysame]{amsrefs}
\usepackage{mathtools,cases}
\usepackage[colorlinks,linkcolor=blue,anchorcolor=blue,citecolor=blue]{hyperref}
\usepackage{enumerate,enumitem}
\usepackage{graphicx}
\numberwithin{equation}{section}

\newcommand{\p}{\partial}
\newcommand{\vphi}{\varphi}
\newcommand{\om}{\omega}
\newcommand{\tri}{\triangle}
\newcommand{\eps}{\epsilon}
\newcommand{\thmref}[1]{Theorem~\ref{#1}}

\newcommand{\lemref}[1]{Lemma~\ref{#1}}

\newcommand{\KN}{\mathbin{\bigcirc\mspace{-15mu}\wedge\mspace{3mu}}}
\newcommand{\Om}{\Omega}
\newcommand{\Na}{\nabla}

\newtheorem{theorem}{Theorem}[section]
\newtheorem{thm}[theorem]{Theorem}
\newtheorem{rem}[theorem]{Remark}
\newtheorem{defn}[theorem]{Definition}
\newtheorem{lemma}[theorem]{Lemma}
\newtheorem{lem}[theorem]{Lemma}

\newtheorem{proposition}[theorem]{Proposition}
\newtheorem{prop}[theorem]{Proposition}

\newtheorem{ques}[theorem]{Question}

\def\Vol{\mathop{\rm Vol}\nolimits}

\def\tr{\mathop{\rm tr}\nolimits}

\def\cM{{\mathcal M}}

\def\a{\alpha}

\def\curl{\mathop{\rm curl}\nolimits}
\def\div{\mathop{\rm div}\nolimits}

\def\sech{\mathop{\rm sech}\nolimits}

\title[Compactness for complete $G_2$-solitons]{On the structure of complete $G_2$-solitons}

\author{Haozhao Li}
\address{Institute of Geometry and Physics and  School of Mathematical Sciences, University of Science and Technology of China, No. 96 Jinzhai Road, Hefei, Anhui Province, 230026, China}
\email{hzli@ustc.edu.cn}

\author{Yuanqing Ma}
\address{Yunnan Key Laboratory of Modern Analytical Mathematics and Applications, School of Mathematics,
      Yunnan Normal University, Kunming, Yunnan 650500, P.R. China}
\email{240097@ynnu.edu.cn}

\author{Kai Zheng}
\address{University of Chinese Academy of Sciences, Beijing 100190, P.R. China}
\email{KaiZheng@amss.ac.cn}

\begin{document}
\begin{abstract}
In this work, we establish compactness and regularity results for complete gradient Laplacian solitons of closed $G_2$-structures. Under a lower scalar-curvature bound and a distance-dependent bound on the gradient of the soliton potential, we prove pointed measured Gromov-Hausdorff compactness. 
With a uniform lower bound for the localised Perelman entropy, the Gromov-Hausdorff convergence improves to pointed $C^{1,\alpha}$ convergence. 

Our principal result shows that this \(C^{1,\alpha}\) convergence
upgrades to smooth convergence on the regular set. More precisely, after
passing to a subsequence, the metrics, the defining positive
\(3\)-forms, and the soliton potentials converge smoothly on every
compact subset of the regular set, and the limiting data define a
gradient Laplacian soliton.

The proof develops a local entropy method adapted to \(G_2\)-solitons.
Since the available \(C^{1,\alpha}\) control does not directly close an
elliptic bootstrap for the \(G_2\)-soliton, and no suitable
pseudolocality theorem is available in this setting, we instead use the
localised Perelman's functionals. These yield an entropy
\(\varepsilon\)-regularity theorem and a gap theorem for scalar-flat
solitons. On the regular set, pointed \(C^{1,\alpha}\) convergence gives
an almost-Euclidean local isoperimetric inequality, which in turn verifies
the required small-entropy condition automatically. The resulting
curvature bounds are then combined with \(G_2\)-specific differential
identities and quantitative interior estimates to control the soliton data.

Finally, at the critical exponent in dimension seven, we show that a uniform weighted \(L^{\frac{7}{2}}\)
-curvature bound then yields pointed $C^\infty$ compactness.
\end{abstract}

    \maketitle

    \tableofcontents

\section{Introduction}

$G_2$-Laplacian solitons are self-similar solutions of the Laplacian flow of closed $G_2$-structures and play the role of singularity models for the flow. 
In this paper we develop a compactness and regularity theory for complete gradient $G_2$-solitons. Our goal is to understand how sequences of such solitons can degenerate and what singular spaces are needed to compactify their moduli space. The main results show that, under natural scalar-curvature and potential-growth assumptions, sequences admit pointed measured Gromov-Hausdorff limits; under a non-collapsing entropy assumption the convergence improves to $C^{1,\alpha}$
; and under additional entropy or $L^{\frac{7}{2}}$-curvature hypotheses the convergence is smooth on the regular set, or smooth everywhere.

A $G_2$-structure on a 7-dimensional manifold $M$ is determined by a positive 3-form $\vphi$. More precisely, the positivity of $\vphi$ is defined via the positivity of the quadratic form
\begin{align*}
G_\vphi(X,Y)= \frac{1}{6}i_X\vphi\wedge i_Y\vphi\wedge\vphi, \quad \forall X,Y\in TM.
\end{align*}
This quadratic form naturally induces a Riemannian metric $g_\vphi$ through
\begin{align}\label{G2 metric}
g_\vphi \cdot dv_{g_{\vphi}}=G_{\vphi},
\end{align}
where $dv_{g_{\vphi}}$ is the volume form corresponding to $g_\vphi$.
We recall the relevant background knowledge of $G_2$-geometry in Section \ref{G2 soliton}.

\begin{defn}
A $G_2$-Laplacian soliton, or simply a $G_2$-soliton,
\[(M,\varphi,g,X),\]
consists of a closed $G_2$-structure $\varphi$ together with a vector field $X$ and a constant $\lambda$ satisfying
\begin{align}\label{$G_2$-soliton}
\Delta_\varphi \varphi = \lambda \varphi + L_X \varphi \quad \text{on } M.
\end{align}
We say the $G_2$-soliton is expanding, steady, or shrinking according to whether
\begin{align*}
\lambda > 0,\quad \lambda = 0,\quad \text{or} \quad \lambda < 0,
\end{align*}
respectively.

If the vector field $X$ is the gradient of a function, that is, if
\[X = \nabla f\]
for some smooth function $f$, then the $G_2$-soliton is called a gradient soliton and is denoted by
\[(M,\varphi,g,f).\]

A $G_2$-soliton is called torsion-free if the underlying $G_2$-structure is torsion-free. A $G_2$-soliton is said to be complete if the associated Riemannian metric \eqref{G2 metric} is complete.
\end{defn}

So far, no compact, non-torsion-free $G_2$-solitons have been found. More precisely, compact shrinking solitons do not exist, and steady solitons are all torsion-free; see Lin \cite{MR3004019}*{Corollary 1} (note that the sign of $\lambda$ is reversed in their convention). Moreover, $G_2$-solitons with $X=0$ are also torsion-free, as shown in Lotay and Wei \cite{MR3613456}*{Proposition 9.2}. Thus, the only remaining possibility is an expanding soliton with a nonzero vector field $X$, but such a soliton cannot occur on homogeneous manifolds by Podest\`a and Raffero \cite{MR3927848}*{Corollary 2.2}.

There are, however, numerous constructions of complete, non-compact $G_2$-solitons \cite{MR2282011,MR3772582,MR3653239,MR3819121,MR4395079,MR4134249,MR4165690,MR3693936}. In \cite{MR4698533}*{Theorem 1.1}, Ng derives a necessary condition for the existence of gradient $G_2$-solitons on homogeneous spaces. More recently, complete gradient Laplacian solitons have been constructed in \cite{MR4555993,MR4349461,MR4728483,arXiv:2112.09095,arXiv:2501.05437}. The uniqueness of complete gradient shrinkers with conical ends is proved by Haskins, Khan, and Payne \cite{MR4884543}*{Theorem 1.1}.

\bigskip
\begin{defn}
Let $\mathcal M$ be the set of complete gradient $G_2$-solitons \[(M,\varphi,g,f,p)\] each endowed with a chosen basepoint $p$.
\end{defn}
We normalise $f$ such that $f(p)=0$ at the basepoint. 

A natural problem is to determine which singular models are required in order to compactify the moduli space $\mathcal M$.
Our aim is to analyse the potential singular models by employing singularity analysis methods from Riemannian convergence theory.

\bigskip
We use \(Rm\), \(Ric\), and \(S\) to denote the Riemann curvature tensor,
Ricci curvature, and scalar curvature of \(g\), respectively.  We write
\(d(x,p)\) for the distance from \(x\) to the basepoint \(p\), and
\(dv_f=\exp(-f)dv_g\) for the weighted measure.

We now introduce the following subclass of $\mathcal M$, where $\Lambda$ controls the negative scalar curvature, $F$ controls the potential growth and \(\underline\nu\) is a non-collapsing parameter. The class \(\mathcal M(\Lambda,F)\) is intentionally broad: it imposes only weak control on the potential function \(f\).  This weak control is one of
the main analytic difficulties in the compactness theory. Nevertheless, we will show that the compactness theorems continue to apply for such general functions.
\begin{defn}\label{3 assumptions}
We define the moduli space
\[\mathcal M(\Lambda,F)\]
to be the subset of $\mathcal M$ consisting of complete gradient $G_2$-solitons with uniformly bounded $\lambda$ that satisfy the following properties:
\begin{enumerate}[label=(\alph*)]
\item the scalar curvature is uniformly bounded from below,
    \label{Scalar lower bound}
    \begin{align*}
    S \geq -\Lambda(d(x,p)),
    \end{align*}

\item the smooth potential function $f$ 
satisfies
    \label{f condition}
    \begin{align*}
     |\nabla f| \leq F(d(x,p)) ,
    \end{align*}
\end{enumerate}
where $\Lambda$ and $F$ are smooth, positive, non-decreasing functions on
$\mathbb R_+ := [0,\infty).$

We then define
\[\mathcal M(\Lambda,F,\underline{\nu})\]
to be the subset of $\mathcal M(\Lambda,F)$ consisting of gradient $G_2$-solitons that also satisfy the following non-collapsing requirement:
\begin{enumerate}[resume, label=(\alph*)]
\item the local version of Perelman's $\nu$-entropy (see Section \ref{Local Perelman entropy}) admits a uniform lower bound
\label{noncollapsing}
\begin{align*}
\nu\left( B_{g}(x,r), g, r^2 \right) \geq \underline{\nu}, \quad \forall x \in M,\quad r \in (0,1).
\end{align*}
\end{enumerate}
Throughout this article, the soliton constant $\lambda$ appearing in various estimates is assumed to be uniformly bounded, i.e.
$|\lambda|\leq L_\lambda$ for some fixed constant $L_\lambda$. In compactness arguments/results, $\lambda$ is taken to be fixed, and we therefore suppress it in the notation for moduli spaces. 
\end{defn}
The space $\mathcal M(\Lambda,F)$ contains at least the following explicit examples: 
a shrinking $G_2$-soliton constructed by Fowdar \cite{MR4349461}, 
a cohomogeneity-one, $Sp(2)$-invariant shrinking $G_2$-soliton and a cohomogeneity-one, $SU(3)$-invariant steady $G_2$-soliton obtained by Haskins and Nordstr\"om \cite{arXiv:2112.09095}, 
a homogeneous steady $G_2$-soliton due to Fino and Raffero \cite{MR4728483}
and inhomogeneous complete steady gradient solitons discovered by Ball \cite{MR4555993}; see Section \ref{examples}. These examples exhibit different possible growth behaviours of the potential function, which naturally leads one to look for corresponding compactness theorems for a general function $F$.

\bigskip

The starting point is the estimate
\begin{align*}
\frac{-\lambda + S}{3} g
\leq Ric_f
\leq \frac{-\lambda - S}{3} g
,
\end{align*}
see \lemref{Ricci bound}; a complete proof is given in \cite{MR4732956}*{Theorem 2.15}.
It converts scalar curvature and potential-growth assumptions into Bakry-\'{E}mery comparison geometry. This gives volume doubling and weak compactness. The remaining difficulty is to upgrade metric convergence to smooth convergence of the $G_2$-structures.

The proofs combine general compactness theory for weighted manifolds with
several new \(G_2\)-specific analytic ingredients.  The principal new
ingredients are the smooth convergence on the regular set under local Riemann curvature bounds
(\thmref{G2 smooth convergence}),
the entropy epsilon-regularity theorem
(Theorem~\ref{Epsilon regularity}), 
the gap theorem
(Theorem~\ref{Gap theorem}), 
the critical $L^{\frac{7}{2}}$-energy
epsilon-regularity theorem (\thmref{Laplace4}), and the Shi-type higher
order estimates (\thmref{Shi}).  These results provide the mechanism for
upgrading weak or \(C^{1,\alpha}\) convergence to smooth convergence of the
metrics, the \(G_2\)-structures, and the potential functions.

The main compactness results are organised in increasing order of
regularity.  First, Theorem~\ref{thm:weak convergence} gives pointed
measured Gromov-Hausdorff compactness for the broad class
\(\mathcal M(\Lambda,F)\).  Second, after imposing the entropy
non-collapsing condition, Theorem~\ref{thm:C1alpha-compactness} upgrades
the convergence to pointed \(C^{1,\alpha}\) convergence and produces a
\(C^{1,\alpha}\) singular limit space.  Third,
Theorem~\ref{thm:smooth-convergence-regular-set} shows that the convergence is smooth on the regular set and the limit is a
conifold \(G_2\)-soliton.  Finally, Theorem~\ref{thm:smooth-compactness}
proves that a uniform $L^{\frac{7}{2}}$-bound for the Riemann curvature prevents
singularities from forming and yields pointed \(C^\infty\) compactness.

\bigskip


\begin{thm}[Weak compactness]\label{thm:weak convergence}
Let $\left( M_i, \vphi_i, g_i, f_i,p_i\right)$ be a sequence in \[\mathcal M(\Lambda,F).\]
Then, after possibly passing to a subsequence, the associated sequence of metric measure spaces
converges in the pointed measured Gromov-Hausdorff sense to a metric measure space,
see Definition \ref{pmGH convergence}.
Furthermore, the limit space has a weak regular-singular decomposition $M=\mathcal R \cup \mathcal S$,
see  \thmref{weak regular-singular decomposition}.
\end{thm}
The resulting limit space may be highly singular and may even exhibit collapse. Nonetheless, S. Huang, Y. Li and B. Wang \cite{MR4234100}*{Theorem 1.5, Proposition 4.8} establish that the regular part enjoys specific convexity properties and that the tangent cone depends H\"older continuously in the Gromov-Hausdorff topology, extending the work of Colding and Naber \cite{MR2950772} on manifolds with a lower Ricci curvature bound.
We summarise these consequences in Section \ref{Pointed measured Gromov-Hausdorff convergence}.

Under the additional non-collapsing assumption \ref{noncollapsing}, the Gromov-Hausdorff convergence can be strengthened to pointed $C^{1,\alpha}$ convergence for sequences in $\mathcal M(\Lambda,F,\underline{\nu})$. This follows from a general compactness theorem of H.Z. Li, Y. Li and B. Wang \cite{MR4220743}*{Theorem 10.2}, which extends Cheeger-Naber's result \cite{MR3418535} to certain metric measure spaces by exploiting estimates in harmonic coordinates \cite{MR1074481}, Anderson's gap lemma \cite{MR1074481}, Colding's volume continuity theorem \cite{MR1454700 }, and the characterisation of tangent cones  by Cheeger and Colding \cite{MR1405949,MR1484888}, namely that every tangent cone is a metric cone. We summarise these arguments in Section \ref{C1alpha compactness}.

\begin{thm}[Pointed $C^{1,\alpha}$ compactness]\label{thm:C1alpha-compactness}
Let $\left( M_i, \vphi_i, g_i, f_i,p_i\right)$ be a sequence in
\[\mathcal M(\Lambda,F,\underline{\nu}).\]
Then the pointed Gromov-Hausdorff convergence obtained in \thmref{thm:weak convergence} can in fact be upgraded to pointed $C^{1,\alpha}$ convergence; see Definition \ref{c1alpha convergence}, i.e.
after passing to a subsequence,
\[
(M_i,g_i,f_i,p_i)
\xrightarrow{pC^{1,\alpha}}
(M_\infty,g_\infty,f_\infty,p_\infty).
\]
Furthermore, the limiting space
is a $C^{1,\alpha}$ singular space and the regular-singular decomposition $$M_\infty=\mathcal R \cup \mathcal S$$ satisfies the properties stated
in Definition \ref{regular-singular decomposition}.
\end{thm}

\bigskip

Theorem~\ref{thm:C1alpha-compactness} gives a \(C^{1,\alpha}\) singular
limit space.  The next problem is to determine when this weak regularity can
be upgraded to smooth convergence of the full \(G_2\)-soliton data on the
regular set.  In particular, one must control not only the metrics, but also
the \(G_2\)-structures and the potential functions.  This leads to the
following question.

\begin{ques}\label{Regularity main question}
Is it possible to prove the smooth convergence of the $G_2$-structures $\vphi_i$ on the regular part?
\end{ques}
To address Question \ref{Regularity main question}, we encounter a number of obstacles. First, as discussed in Remark \ref{elliptic difficulty}, a straightforward bootstrap argument cannot be applied to the elliptic system (see Section \ref{G2 soliton}) satisfied by a $G_2$-soliton. Second, for the complete and potentially non-compact $G_2$-flow, there is currently no counterpart of a pseudolocality theorem generalising Perelman's pseudolocality result for the Ricci flow.

To overcome these difficulties, we introduce a new approach that uses the localized version of Perelman's entropy developed by B. Wang in \cite{MR3855081}.The new strategy establishes conifold compactness \thmref{thm:conifold-compactness main} assuming a small entropy condition, and then confirms this small entropy hypothesis in \lemref{lem:automatic-local-entropy}. The argument relies on several steps. We begin by extracting curvature bounds via the entropy $\varepsilon$-regularity theorem \thmref{Epsilon regularity}. Next, we derive higher-order quantitative interior estimates for the metrics, the $G_2$-structures, and the potential functions, as stated in Proposition \ref{lift}. With these tools in place, we prove smooth convergence of the $G_2$-structures in \thmref{G2 smooth convergence} using an exhaustion-diagonal procedure. 

The following theorem provides an answer to Question \ref{Regularity main question}.

\begin{defn}[Conifold $G_2$-soliton]\label{Conifold}
A length space $\left( M, d,g, f,p\right)$ is called a conifold $G_2$-soliton if it is a $C^{\infty}$ singular space in the sense of Definition \ref{regular-singular decomposition} and there exists a $G_2$-structure $\vphi$ on the regular part $\mathcal R$ such that $(\mathcal R, \vphi, g , f)$ is
a $G_2$-soliton.
\end{defn}

\begin{theorem}[Conifold compactness]
\label{thm:smooth-convergence-regular-set}
Let
\[
(M_i,\vphi_i,g_i,f_i,p_i)
\in\mathcal M(\Lambda,F,\underline\nu)
\]
be a sequence of complete gradient closed \(G_2\)-solitons with fixed
soliton constant \(\lambda\). Then the pointed $C^{1,\alpha}$ convergence in \thmref{thm:C1alpha-compactness} can be upgraded to pointed $C^{\infty}$ convergence, see Definition \ref{Cinfty convergence} and the limit is a conifold \(G_2\)-soliton. 

Precisely speaking, after passing to a further subsequence of the pointed $C^{1,\alpha}$ convergent subsequence, there exists a smooth
$G_2$-structure \(\vphi_\infty\) on \(\mathcal R\) such that
\[
(M_i,\vphi_i,g_i,f_i,p_i)
\xrightarrow{pC^\infty}
(M_\infty,\vphi_\infty,g_\infty,f_\infty,p_\infty)\text{ on }\mathcal R. 
\]
Moreover,
\[
g_{\vphi_\infty}=g_\infty,
\quad
d\vphi_\infty=0,
\quad
\Delta_{\vphi_\infty}\vphi_\infty
=
\lambda\vphi_\infty
+
L_{\nabla^{g_\infty}f_\infty}\vphi_\infty
\quad\text{on }\mathcal R.
\]
\end{theorem}

\begin{rem}
The term ``conifold'' was introduced in Chen and Wang \cite{MR3739253}*{Definition 2.1}, where it was proved that the moduli space of non-collapsed Calabi-Yau conifolds is compact with respect to the pointed $C^{\infty}$ topology.
\end{rem}

\bigskip
In Section~\ref{Smooth convergence on uniform curvature bounds},
a small critical \(L^{\frac{7}{2}}\)-curvature energy estimate, together with the
Shi-type higher-order estimates, is used to rule out singular points in the
limit.
Finally, we establish a compactness theorem for complete, non-compact gradient $G_2$-solitons within the subclass $\mathcal M(\Lambda,F,\underline{\nu})$ under the Riemann curvature $L^{\frac{7}{2}}$ bound.
\begin{thm}[Pointed $C^{\infty}$ compactness]\label{thm:smooth-compactness}
Let $\left( M_i, \vphi_i, g_i, f_i,p_i\right)$ be a sequence in
\[\mathcal M(\Lambda,F,\underline{\nu})\]
with a uniform $L^{\frac{7}{2}}$-bound on the Riemann curvature:
\begin{equation}
\int_{B_{g_i}(p_i,R)}\left|Rm_{g_i} \right|_{g_i}^{\frac{7}{2}}dv_{f_i}\leq E(R)<\infty .
\end{equation}

Then a subsequence of $\left( M_i, \vphi_i, g_i, f_i, p_i\right)$ converges, in the pointed $C^\infty$ sense, to a smooth $G_2$-soliton.
\end{thm}

\begin{rem}
The epsilon-regularity theorems under small-energy hypotheses and the non-collapsed orbifold compactness theorem have been studied for a variety of classes of critical points; see, for example,
Anderson \cite{MR999661} and Bando, Kasue and Nakajima \cite{MR1001844} for Einstein metrics,
Tian and Viaclovsky \cite{MR2138071,MR2166311} for critical metrics including Bach flat metrics and metrics with harmonic curvature in dimension 4,
Chen and Weber \cite{MR2737786} for extremal K\"ahler metrics, Weber \cite{MR2755484} for compact Ricci soliton,
 and the references contained therein.

For complete, non-compact gradient shrinkers, Haslhofer and Buzano (M\"uller) \cite{MR2846384} prove a compactness theorem  under the assumptions of a uniform lower bound on the entropy and a uniform upper bound on the energy.

The proof of \thmref{thm:smooth-compactness} combines these ideas with the special structure of closed $G_2$-solitons to rule out the formation of singular points in the limit.
\end{rem}

\begin{rem}
The use of Perelman's entropy to establish epsilon-regularity has been developed in several contexts. For recent progress in this direction, we refer the reader to Hein and Naber \cite{MR3245102}, Lee, Naber and Neumayer \cite{MR4584264}, B. Wang \cite{arXiv:2010.09981} and J. Wang and Y. Wang \cite{MR4846770}.
\end{rem}

\bigskip



\section{\texorpdfstring{$G_2$}{G_2}-soliton}\label{G2 soliton}
In this section, we begin with a brief review of $G_2$-solitons: one approach is through the exterior differential system associated with the torsion form, and the other is through curvature equations. We then derive several fundamental formulas and estimates for $G_2$-solitons that will be used in subsequent parts of the paper. Throughout this note, we assume that the $G_2$-structure $\varphi$ is a closed 3-form.

The Hodge star operator $\ast_\vphi$ and the codifferential \[\delta=(-1)^k\ast_\vphi d\ast_\vphi\] acting on $k$-forms are then defined with respect to the $G_2$-metric $g_\vphi$ in \eqref{G2 metric}. The 4-form associated to a $G_2$-structure $\vphi$ is given by
\[\psi = \ast_\vphi \vphi.\]

Let $\Omega^k$ denote the space of all smooth $k$-forms on $M$.
For a given $G_2$-structure, the spaces $\Omega^2$ and $\Omega^3$ split as direct sums of irreducible $G_2$-representations:
\begin{align*}
\Omega^2 &= \Omega^2_7 \oplus \Omega^2_{14},\quad
\Omega^3 = \Omega^3_1 \oplus \Omega^3_7 \oplus \Omega^3_{27}.
\end{align*}
See Bryant \cite{MR2282011}*{Section 2.6}.
The subspaces are characterised as follows
\begin{equation}\label{eq:de1}
\left\{
\begin{aligned}
    \Om^2_7& =\left\{ u\lrcorner \vphi \mid u\in \Gamma(TM) \right\} = \left\{ \ast_\vphi(\a \wedge \psi) \mid \a\in \Omega^1 \right\}   \\
    &=\left\{ \beta \in \Lambda^2  \mid\ast_\vphi(\vphi \wedge \beta )=2\beta \right\}   ;\\
   \Om^2_{14}&=\left\{ \beta \in \Om^2 \mid \ast_\vphi(\psi \wedge \beta )=0 \right\} \\
   &=\left\{ \beta \in \Om^2 \mid \ast_\vphi(\vphi \wedge \beta )=-\beta \right\};
\end{aligned}
\right.
\end{equation}
and
\begin{equation}\label{eq:de2}
\left\{
\begin{aligned}
\Om^3_1&=\left\{ f\vphi \mid f\in C^{\infty}(M) \right \};\\
\Om^3_7 &= \left\{ u\lrcorner \psi \mid u\in \Gamma(TM) \right\}=\left\{ \ast_\vphi(\a \wedge \vphi) \mid \a\in    \Om^1 \right\}\\
&=\left\{ \eta \in    \Om^3 \mid \ast_\vphi[\vphi \wedge \ast_\vphi(\vphi \wedge \eta)] =4\eta \right\} ; \\
\Om^3_{27}&=\left\{ \eta \in    \Om^3 \mid \eta \wedge \vphi=0, \eta \wedge \psi=0 \right\}.
\end{aligned}
\right.
\end{equation}

If $\vphi$ is closed, then the torsion form
\[
\tau = \delta\vphi
\]
is a 2-form lying in $\Omega_{14}^2$ and satisfies the following system of exterior differential equations:
\begin{align}\label{tau}
d\psi = \tau \wedge \vphi = -\ast_\vphi \tau, \quad \tau \wedge \psi = 0.
\end{align}

A closed $G_2$-structure $\vphi$ is said to be \emph{torsion-free} if it is coclosed, i.e. \[\tau=0.\] On a compact manifold, this condition is equivalent to the nonlinear harmonicity equation
\begin{align*}
\tri_\vphi\vphi=0.
\end{align*}
The $G_2$-Laplacian flow is the gradient flow of Hitchin's volume functional with respect to a natural gradient-type metric on the space of closed $G_2$-structures and is given by
\begin{align*}
\p_t\vphi=\tri_\vphi\vphi,
\end{align*}
see Bryant \cite{MR2282011}*{Section 6.2 Remark 16}.

We begin by fixing some standard notation.
We use
\begin{align*}
Rm,\quad Ric,\quad S
\end{align*}
to denote the Riemann curvature tensor, the Ricci curvature, and the scalar curvature of the Riemannian metric $g$, respectively. Given any $k$-form
\[\eta=\frac{1}{k!}\eta_{i_1\dots i_k}\,dx^{i_1\cdots i_k},\]
we consider two norms: the squared tensor norm
\[
|\eta|^2=g_\varphi^{i_1j_1}\cdots g_\varphi^{i_k j_k}\,\eta_{i_1\dots i_k}\eta_{j_1\dots j_k},
\]
and the associated squared form norm
\[
\|\eta\|^2=g(\eta,\eta)=\frac{1}{k!}|\eta|^2.
\]
We denote the Hodge Laplacian by $\tri_\vphi$, the Laplace-Beltrami operator by $\tri$, and the weighted $f$-Laplacian by $\tri_f$; see the conventions stated at the beginning of Section \ref{Formulas for scalar curvature and the potential function}.

\subsection{Torsions and curvatures}
\begin{defn}\label{ABS definition}
We introduce the symmetric tensors
\begin{align*}
A_{ij}:=-\frac{1}{2}{\tau_i}^{k}\tau_{kj},\quad B_{ij}:={\vphi_i}^{pq} \nabla_p \tau_{qj} .
\end{align*}
More precisely, the trace of the tensor $A$ equals the squared norm of the torsion form $\tau$:
\[\tr A = \|\tau\|^2.\]
\end{defn}

For a closed $G_2$-structure $\vphi$, its Hodge Laplacian \[\tri_\vphi\vphi=d\tau\] admits a decomposition into irreducible $G_2$-modules, see \cite{MR3613456}*{Section 2.2}.
\begin{lem}\label{d tau decomposition}
Suppose $\vphi$ is a closed $G_2$-structure. Then
\begin{align*}
\tri_\vphi\vphi=d\tau=\frac{1}{7}\|\tau\|^2\vphi+i_\vphi h_0,\quad
h_0=\frac{1}{2}B_{ij}-\frac{3}{14}\|\tau\|^2g_{ij}+\frac{1}{2}A_{ij}.
\end{align*}
\end{lem}
\begin{proof}
Since $\tau\in \Om_{14}^2$, we employ the characterisation in \eqref{eq:de1} and invoke \cite{MR4732956}*{Proposition 2.19} to obtain $\pi_7^3d\tau=0$ together with the explicit formula for $\pi_1^3d\tau$.
To confirm that
\[\pi_1^3d\tau=\frac{1}{7}\|\tau\|^2\vphi,\]
we compute
\begin{align*}
g(\vphi,d\tau)dv
&=d\tau \wedge \ast_\vphi \vphi
=d(\tau \wedge \ast_\vphi \vphi)-\tau \wedge d(\ast_\vphi \vphi)
=\tau\wedge\ast_\vphi\tau
=\|\tau\|^2 dv.
\end{align*}
The $\Om_{27}^3$-component is given by $i_\vphi h_0$, where $h_0$ is a trace-free symmetric $2$-tensor introduced in \lemref{d tau decomposition}; see also \cite{MR3613456}*{Equation (2.20)}. Combining these components yields the identities claimed in the lemma.
\end{proof}

We next recall the expression for $d(\tau^3)$ from \cite{MR2282011}*{Corollary 3} and for $L_X\vphi$ from \cite{MR3613456}*{Section 9}.
\begin{lem}\label{d tau soliton}
Let $(M,\vphi,g,X)$ be a $G_2$-soliton. Then
\begin{align*}
&L_X\vphi=di_X\vphi
=(\tfrac{1}{7}\|\tau\|^2-\lambda)\vphi+i_\vphi h_0,\\
&\frac{1}{3}d(\tau^3)
=-\frac{1}{7}\|\tau\|^4 dv+\tau^2\wedge \pi_{27}^3(d i_X\vphi).
\end{align*}
\end{lem}
\begin{proof}
By combining the soliton equation \eqref{$G_2$-soliton} with the decomposition \lemref{d tau decomposition}, we obtain
\[
\lambda\vphi+L_X\vphi
=\tri_\vphi\vphi
=\frac{1}{7}\|\tau\|^2\vphi+i_\vphi h_0.
\]
Using the Cartan formula,
\[
L_X\vphi=di_X\vphi,
\]
we then obtain the first identity. For the second, observe that
\[
\frac{1}{3}d(\tau^3)=\tau^2\wedge d\tau.
\]
Using \lemref{d tau decomposition} again and substituting \eqref{tau}, we find
\[
\tau^2\wedge\vphi
=-\tau\wedge\ast_\vphi\tau
=-\|\tau\|^2 dv.
\]
Finally, we substitute the decomposition of $di_X\vphi$ to obtain the stated formula.
\end{proof}


\begin{defn}
The \textit{$j_\vphi$-operator} \cite{MR2282011}*{Equation (2.18)} is the map from 3-forms $\gamma$ to the space $S^2$ of symmetric 2-tensors defined by
\begin{align*}
    j_\vphi(\gamma)(X,Y)=\ast_\vphi\big(i_X\vphi\wedge i_Y\vphi\wedge\gamma\big).
\end{align*}
\end{defn}
The Ricci tensor and scalar curvature of the Riemannian metric $g$ induced by a closed $G_2$-structure $\vphi$ are given by \cite{MR2282011}*{Equation (4.36)(4.37)}.
\begin{lem}\label{scalar curvature general}
\begin{align*}
    Ric(g)&=\frac{1}{4}\|\tau\|^2g-\frac{1}{4}j_\vphi\Big[d\tau-\frac{1}{2}\ast_\vphi(\tau\wedge\tau)\Big],\quad
    S(g)=-\frac{1}{2}\|\tau\|^2\leq 0.
\end{align*}
\end{lem}

\begin{defn}
The torsion tensor $T$ is the skew-symmetric tensor
\begin{align*}
    T_{ij}=-T_{ji}=-\frac{1}{2}\tau_{ij}.
\end{align*}
\end{defn}

Using this, we can rewrite Definition \ref{ABS definition} as
\begin{lem}\label{ABS}
\begin{align*}
    A_{ij}=-2{T_i}^{k}T_{kj},\quad
    B_{ij}=-2{\vphi_i}^{pq} \nabla_p T_{qj}, \quad
    \tr A=-2S,\quad
    S=-|T|^2.
\end{align*}
\end{lem}
\begin{defn}\label{Ricci curvature soliton gradient}
We set
\begin{align*}
    \lambda_S:=-\frac{\lambda}{3}+\frac{S}{3},\quad
    \mathcal K:=\frac{S}{3} g+A=-\frac{|T|^2}{3} g-2{T_i}^{k}T_{kj}.
\end{align*}
\end{defn}
Using \lemref{ABS} and $\nabla^pT_{pj}=0$, see \eqref{T contraction 3}, we have
\begin{lem}\label{trK}
\begin{align*}
    \tr \mathcal K=\frac{1}{3}S,\quad
\div A
=  - 2\,\div({T_i}^{k}T_{kj})
=  - 2T^{ak}\nabla_a T_{kj}.
\end{align*}
\end{lem}

As shown in \cite{MR3613456}*{Equation (9.12)(9.14)}, we have:

\begin{lem}
If $(M,\vphi,g,X)$ is a $G_2$-soliton, then
\begin{align}
&Ric_X:=Ric+\frac{1}{2}L_Xg= -\frac{\lambda}{3} g+\mathcal K, \label{Ricci curvature soliton}
\\
&2S=-7\lambda-3\,\mathrm{div}\,X,
\label{scalar curvature soliton}
\\
&\delta(i_X\vphi)=0.
\end{align}
\end{lem}

The nonexistence of $G_2$-solitons on a closed manifold $M$ has been established in \cite{MR3004019}*{Corollary 1} and further developed in \cite{MR3613456}*{Proposition 9.1, Proposition 9.2, Proposition 9.5, Remark 9.6}.
We include the proof here for the reader's convenience.
\begin{lem}\label{nonexistence soliton}
Let $M$ be a closed manifold, i.e.\ a compact manifold without boundary. Then:
\begin{itemize}
    \item shrinking $G_2$-solitons do not exist;
    \item any steady $G_2$-soliton is necessarily torsion-free;
    \item any expanding $G_2$-soliton with $X=0$ must be torsion-free.
\end{itemize}
Consequently, the only compact $G_2$-solitons that can possess nonzero torsion are expanding $G_2$-solitons with $X\neq0$.
\end{lem}
\begin{proof}
Integrating the scalar curvature equation \eqref{scalar curvature soliton} over a closed $M$ yields the formula for the averaged scalar curvature of a $G_2$-soliton:
\begin{align}\label{averaged scalar curvature}
\int_M S dv= -\frac{7\lambda}{2} \Vol(M).
\end{align}
Using the non-positivity of the scalar curvature from \lemref{scalar curvature general}, we obtain the constraint on the soliton constant $\lambda$:
\[
\lambda\geq 0.
\]

If $\lambda=0$, then the averaged scalar curvature must vanish, which forces $\tau=0$, i.e. the structure is torsion-free. When $\lambda>0$ and $X=0$, \lemref{d tau soliton} gives
    \begin{align*}
    \lambda=\frac{1}{7}\|\tau\|^2 \quad\text{and}\quad \frac{1}{3}d(\tau^3)=-\frac{1}{7}\|\tau\|^4 dv.
\end{align*}
Integrating these over the closed manifold $M$ again implies that $\tau$ must vanish.  Thus, we reach the final conclusion.
 \end{proof}
\begin{rem}
In \cite{potential function}, we will derive a constraint on the scalar curvature of the only possible compact, expanding, gradient $G_2$-soliton.
\end{rem}
 \begin{rem}
No explicit examples of compact non-torsion-free expanding $G_2$-solitons with $X\neq0$ are currently known.
 \end{rem}

\subsubsection{Control $Ric_X$ in terms of the scalar curvature}
    In \cite{MR4732956}*{Theorem 2.15}, we study the eigenvalues of $A$ and obtain bounds for $Ric_X$.
    \begin{lem}\label{Ricci bound}
    We have \[
    0\leq A\leq -\frac{2}{3}Sg.
    \]
        Furthermore, if $(M,\vphi,g,X)$ is a $G_2$-soliton, then
        \begin{align*}
            \frac{-\lambda+S}{3}g\leq Ric_X\leq \frac{-\lambda-S}{3}g.
        \end{align*}
    \end{lem}
On a compact $G_2$-manifold, the holonomy group is exactly $G_2$ if and only if the fundamental group is finite \cite{MR1787733}*{Proposition 10.2.2}. We now state the analogous result for complete non-compact $G_2$-solitons.
\begin{rem}
If 
\begin{align}\label{fundamental group}
c=\inf_M(S-\lambda)/3>0,
\end{align}
then
\lemref{Ricci bound} implies
\[Ric_X\geq cg.\] That verifies the assumptions in \cite[Theorem 1.1]{MR2373611}. Thus a complete non-compact $G_2$-soliton has finite fundamental group.

In the compact case, since $\lambda\geq 0$ by \lemref{nonexistence soliton}, the assumption \eqref{fundamental group} fails.
\end{rem}
\begin{rem}
In Section \ref{examples}, the manifold $M_2$ fulfils the condition \eqref{fundamental group}, whereas $M_1$ and $M_3$ do not.
\end{rem}

\subsection{Scalar curvature and the potential function}

\subsubsection{Formulas for scalar curvature and the potential function}
\label{Formulas for scalar curvature and the potential function}
We set \[X=\nabla f\] and refer to $f$ as the potential function.

From now on, for every tensor field, we use the rough Laplacian:
\[
    \tri 
    :=
    \operatorname{tr}_g\nabla^2 
    =
    \sum_{i=1}^7
    \left(
        \nabla_{e_i}\nabla_{e_i}
        -
        \nabla_{\nabla_{e_i}e_i}
    \right),
\]
where \(\{e_i\}_{i=1}^7\) is a local orthonormal frame. In normal
coordinates at a point,
\[
    \tri =g^{ij}\nabla_i\nabla_j.
\]
In particular, for a smooth function \(u\),
\[
    \tri u=\operatorname{div}(\nabla u).
\]

With the standard formal-adjoint convention, the Bochner
Laplacian is related to the rough Laplacian by
\[
    \nabla^\ast\nabla =-\tri .
\]
Thus, for compactly supported smooth functions \(u,v\),
\[
    \int_M v\,\tri u\,dv_g
    =
    -\int_M\langle\nabla u,\nabla v\rangle\,dv_g.
\]

The Hodge Laplacian on differential forms is denoted by
\[
    \Delta_\vphi:=d\delta+\delta d.
\]
The Weitzenb\"ock
formula takes the form
\[
    \Delta_\vphi
    =
    -\tri+\mathcal R_g(\cdot),
\]
where \(\mathcal R_g\) is the curvature action on differential forms.
In particular, on functions,
\[
    \Delta_\vphi u=-\tri u.
\]

\begin{prop}\label{lem:Sc df}
Let $(M,\vphi,g,\nabla f)$ be a gradient $G_2$-soliton. Then
\begin{align}
\label{soliton Ricci gradient potential}
&\tri \varphi = Rm*\varphi - \lambda\varphi - di_{\nabla f}\varphi,\\
&\label{soliton Ricci gradient}
Ric_f := Ric + \nabla^2 f = -\frac{\lambda}{3} g + \mathcal K,\quad i_{\nabla f}T=0,\\
&\label{soliton scalar}
-3\tri f - 7\lambda = 2S \leq 0,\quad \nabla S = -\frac{3}{2}\nabla\tri f,\\
&\label{div K}
\div \mathcal K = \frac{1}{3}\nabla S +\div A= -\frac{1}{6}\nabla S + Ric(\nabla f), \\
&\label{nabla S gradient tau}
\nabla S = 2Ric(\nabla f) -2\div A,\\
&\label{gradient gradient f}
\nabla [S + |\nabla f|^2 + \tfrac{2\lambda}{3} f]
  = \tfrac{2}{3}S\nabla f -2\div A,\\
&\label{Laplace gradient f}
\frac{1}{2}\tri|\nabla f|^2
  = |\nabla^2 f|^2 - \frac{2}{3}\langle \nabla S,\nabla f\rangle
    + Ric(\nabla f,\nabla f).
\end{align}
Define the weighted $f$-Laplace operator by
\[
\tri_f(\cdot) := \tri(\cdot) - (\nabla f,\nabla\cdot).
\]
Then we have
\begin{align}
\tri_f f &= -\frac{7\lambda}{3} - \frac{2S}{3} - |\nabla f|^2,\label{lem:Sc f}\\
\frac{1}{2}\tri_f|\nabla f|^2
&= |\nabla^2 f|^2 + Ric_f(\nabla f,\nabla f)\label{lem:Sc df f}\\
&\quad - \frac{2}{3}\langle \nabla S,\nabla f\rangle
      - \langle \nabla|\nabla f|^2,\nabla f\rangle.\notag
\end{align}
\end{prop}

\begin{proof}
By applying the Weitzenb\"ock formula, which relates the Hodge Laplacian to the Bochner Laplacian, the soliton equation \eqref{$G_2$-soliton} can be rewritten as 
\[\tri \varphi = Rm*\varphi - \lambda\varphi - \mathcal{L}_{\nabla f}\varphi.
\] Here $Rm*\varphi$ denotes the Weitzenb\"ock curvature action $\mathcal R_g(\cdot)$. Then \eqref{soliton Ricci gradient potential} follows from using closeness of $\vphi$.

The second relation in \eqref{soliton Ricci gradient} follows from
\[
0 = \delta i_X\vphi = \curl(X) + 2i_XT
\quad\text{and}\quad
\curl(\nabla f)=0,
\]
see \cite{MR4884543}*{Page 5050}.
The identities in \eqref{soliton scalar} are obtained by taking the trace of \eqref{soliton Ricci gradient},
\[S+\tri f=\frac{-7}{3}\lambda+\tr\mathcal K,\]
 inserting the expression from \lemref{trK}.

For the first equality in \eqref{div K}, we compute the divergence of $\mathcal K$ using Definition~\ref{Ricci curvature soliton gradient}:
\begin{align*}
\div \mathcal K
&= \frac{1}{3}\nabla S +\div A.
\end{align*}
To obtain the second equality in \eqref{div K}, we compute $\nabla S$ using the soliton equation.
The second contracted Bianchi identity gives
\[
\nabla S = 2\,\div\,Ric.
\]
Using the soliton equation and the Bochner formula for the gradient of a function, we obtain
\begin{align*}
\nabla S
&= -2\,\div(\nabla^2 f) + 2\,\div \mathcal K\\
&= -2[\nabla\tri f + Ric(\nabla f)] + 2\,\div \mathcal K.
\end{align*}
Using the scalar curvature identity \eqref{soliton scalar}, we obtain
\begin{align*}
\nabla S
= \frac{4}{3}\nabla S - 2Ric(\nabla f) + 2\,\div \mathcal K,
\end{align*}
which yields the second identity for $div \mathcal K$ in \eqref{div K}.
Equating the two expressions for $div \mathcal K$ in \eqref{div K} gives \eqref{nabla S gradient tau}.
To obtain \eqref{gradient gradient f}, we use the soliton equation \eqref{soliton Ricci gradient} directly:
\begin{align*}
\nabla|\nabla f|^2
  = 2\langle \nabla f, \nabla^2 f\rangle
  = 2\mathcal K(\nabla f)
    - \frac{2\lambda}{3}g(\nabla f)
    - 2Ric(\nabla f).
\end{align*}
Now, adding the second identity in \eqref{div K} to this identity, we obtain
\begin{align}\label{S nabla f f K}
\nabla [S + 3|\nabla f|^2+2\lambda  f]
&= 6\mathcal K(\nabla f) - 6\,\div \mathcal K .
\end{align}
From \eqref{soliton Ricci gradient} we also have \[T_{kj}f^j=0.\]
Inserting the expression for $\mathcal K$ from Definition~\ref{Ricci curvature soliton gradient}, we get
\begin{align*}
\mathcal K(\nabla f)
= \frac{S}{3} g_{ij}f^j - 2{T_i}^{k}T_{kj}f^j
= \frac{S}{3} f_i.
\end{align*}
Substituting this formula together with the first identity in \eqref{div K} into the preceding relation \eqref{S nabla f f K}, we arrive at
\begin{align*}
\nabla [S + 3|\nabla f|^2+2\lambda f]
= 2S\nabla f
   - 6\Big[\frac{1}{3}\nabla S +\div A\Big].
\end{align*}
A straightforward rearrangement of the terms gives \eqref{gradient gradient f}.

The formula \eqref{Laplace gradient f} follows from substituting \eqref{soliton scalar} into the Bochner identity
\[
\frac{1}{2}\tri|\nabla f|^2
  = |\nabla^2 f|^2
    + \langle \nabla\tri f,\nabla f\rangle
    + Ric(\nabla f,\nabla f).
\]
Finally, combining \[\tri_f f = \tri f - |\nabla f|^2\] with the $f$-Bochner identity
\[
\frac{1}{2}\tri_f|\nabla f|^2
  = |\nabla^2 f|^2
    + \langle \nabla\tri_f f,\nabla f\rangle
    + Ric_f(\nabla f,\nabla f),
\]
we deduce \eqref{lem:Sc f} and \eqref{lem:Sc df f}.
\end{proof}

\begin{rem}\label{bounds on S anf f remark}
Since the right-hand side of \eqref{S nabla f f K} (see also \eqref{gradient gradient f}) is no longer identically zero, we cannot conclude that the terms following the gradient on the left-hand side are constant. This loss is significant, as the corresponding identity plays an important role in obtaining gradient estimates for $f$ in the Ricci soliton setting.
\end{rem}
\subsection{Associated flow and higher derivatives}
\label{Associated flow and higher derivatives}
We recall that the $G_2$-Laplacian flow of the 3-form is given by
\begin{align}\label{Laplacian flow form}
\partial_t\vphi=\Delta_\vphi\vphi.
\end{align}
The induced evolution of the metric is the Ricci flow with an additional nonlinear term $\mathcal K$ from Definition~\ref{Ricci curvature soliton gradient},
\begin{align}\label{Laplacian flow metric}
\partial_t g=-2Ric+2\mathcal K.
\end{align}

We now recall the self-similar solution to the $G_2$-Laplacian flow that arises from the soliton equation \eqref{$G_2$-soliton}.
\begin{defn}
We introduce the scaling factor
\begin{align*}
s(t)=\frac{2\lambda t+3}{3}>0,\quad \text{i.e.} \quad \lambda t>-\frac{3}{2}.
\end{align*}
\end{defn}

For the shrinking, steady, and expanding cases, the time parameter $t$ ranges over
\begin{align}\label{time range}
\mathcal I:=(-\infty,-\frac{3}{2\lambda}),\quad (-\infty, +\infty),\quad (-\frac{3}{2\lambda},+\infty),\end{align}
respectively.

\begin{defn}\label{soliton diffeomorphisms}
We define a time-dependent vector field by
\begin{align*}
\tilde X(t)=s(t)^{-1}X,
\end{align*}
which generates a one-parameter family of diffeomorphisms $\sigma(t)$ via
\begin{align*}
\partial_t \sigma(t)=\tilde X(\sigma(t)),\quad \sigma(0)=id.
\end{align*}
\end{defn}
\begin{rem}
If the ordinary differential equation $u_t=F(u)$ with positive initial data admits a global solution, then $\tilde X$ is complete, and hence the diffeomorphisms $\sigma(t)$ are defined on the entire time interval specified above. In general, the flow $\sigma(t)$ only exists on its maximal existence interval.
\end{rem}

\begin{defn}[Associated $G_2$-Laplacian flow]\label{associated flow}
The associated $G_2$-Laplacian flow of a $G_2$-soliton is the triple consisting of the rescaled $G_2$-structure, Riemannian metric, and potential function
\begin{align}\label{associated flow vphi}
\vphi(t)=s(t)^{\frac{3}{2}}\sigma(t)^\ast\vphi,\quad g(t)= s(t)\sigma(t)^\ast g,\quad f(t)=\sigma(t)^\ast f.
\end{align}
The soliton $$(M,\vphi,g,f)$$ is recovered as the time slice at $t=0$ of the associated $G_2$-Laplacian flow.
\end{defn}

\begin{lem}\label{A K scaling invariant}
The tensors $A$ and $\mathcal K$ are invariant under this rescaling, namely
\begin{align*}
A(t)=\sigma(t)^\ast  A, \quad \mathcal K(t)=\sigma(t)^\ast  \mathcal K.
\end{align*}
\end{lem}
\begin{proof}
For brevity, write $s:=s(t)$ and $\sigma:=\sigma(t)$. The curvatures scale as
\begin{align}\label{curvatures scale}
Rm_{g(t)}=s \sigma^\ast Rm_g,\quad
Ric_{g(t)}=\sigma^\ast Ric_g,\quad
S_{g(t)}=s^{-1}  \sigma^\ast S_g.
\end{align}
Here, $Rm$ is viewed as a $(0,4)$-tensor.
For a $k$-form $\alpha$, using the expressions for the Hodge star and the codifferential, we obtain
\begin{align*}
\ast_{g(t)} (\sigma^\ast \alpha)= s^{\frac{7}{2}-k} \sigma^\ast (\ast_g \alpha),\quad \delta_{g(t)}(\sigma^\ast \alpha)=s^{-1} \sigma^\ast (\delta_g\alpha),\quad
\Delta_{\vphi(t)}\sigma^\ast\alpha=s^{-1} \sigma^\ast (\Delta_\vphi\alpha).
\end{align*}
Taking $k=3$ and using $\psi=\ast\vphi$ and $-2T=\delta\vphi$, we deduce
\begin{align*}
\psi(t)=s^{2}\sigma^\ast  \psi,\quad
T(t)=s^{\frac{1}{2}}\sigma^\ast  T,\quad
\Delta_{\vphi(t)}\vphi(t)=s^{\frac{1}{2}} \sigma^\ast (\Delta_\vphi\vphi),
\end{align*}
see also \cite{MR4884543}*{Lemma 2.1}.
This yields the claim.
\end{proof}

\begin{lem}\label{scaled f}
The families $g(t)$ and $\vphi(t)$ solve \eqref{Laplacian flow metric} and \eqref{Laplacian flow form}, respectively. Moreover,
\begin{align*}
\partial_tf(t)=|\nabla_{g(t)}  f(t)|_{g(t)}^2\geq 0.
\end{align*}
\end{lem}
\begin{proof}
Using \lemref{A K scaling invariant}, we compute
\begin{align*}
\partial_t g(t)
&=\sigma(t)^\ast\!\left[\frac{2\lambda}{3} g + L_{X}g\right]
=\sigma(t)^\ast[-2Ric(g) +2\mathcal K]
=-2Ric(g(t)) +2\mathcal K(t).
\end{align*}
As shown in \cite{MR3613456}*{Page 224}, the rescaled $G_2$-form $\vphi(t)$ from Definition \ref{associated flow}
satisfies the flow equation \eqref{Laplacian flow form}. The evolution of $f(t)$ follows immediately from the definition of $\sigma(t)$ in Definition \ref{soliton diffeomorphisms}.
\end{proof}

\begin{lem}
For all $x\in M$,
\begin{align*}
& f(x,t)\leq  f(x,0),\quad t\in (-\infty, 0];\quad
& f(x,t)\geq  f(x,0),\quad t\in [0,+\infty).
\end{align*}
\end{lem}
\begin{proof}
This follows by integrating the evolution equation for $f$ given in \lemref{scaled f}.
\end{proof}

Analogous identities to those in Proposition \ref{lem:Sc df} hold along the associated flow.
\begin{lem}
Along the associated flow, we have
\begin{align}
&\label{soliton Ricci gradient t}
Ric_{g(t)}+\nabla_{g(t)}^2 f= -\frac{\lambda}{3s(t)} g+\mathcal K_{g(t)},\\
&\label{soliton scalar t}
-\triangle f-\frac{7\lambda}{3s(t)}=\frac{2}{3}S\leq 0,\\
&\label{gradient gradient f t}
\nabla \bigl[S+|\nabla f|^2+\frac{2\lambda}{3s(t)} f\bigr]=\frac{2}{3}S\nabla f+ 4T^{ak}\nabla_aT_{kj}.
\end{align}
\end{lem}
We recall the Shi-type estimates of the $G_2$-Laplacian flow by Lotay and Wei \cite{MR3613456}*{Theorem 4.2, Theorem 4.4, Remark 4.5}.
\begin{lem}\label{Higher derivatives flow}
Let $\vphi(t)$ be a $G_2$-Laplacian flow defined on a relatively compact ball $B_{2R}(p)$. Assume that $|Rm(t)|_{g(t)}\leq K$ on $B_{2R}(p) \times [0,\frac{1}{K}]$. Then for every $k\geq 0$ we have
\[
|\nabla^k Rm(t)|_{g(t)}\leq C(k,R,K)\, t^{-\frac{k}{2}}K \quad\text{on } B_R(p) \times \Big[0,\frac{1}{K}\Big].
\]
\end{lem}


\begin{rem}
Let $(M,\vphi,g,\nabla f)$ be a gradient $G_2$-soliton. 
Suppose that $|Rm|_{g}\leq K$ on $B_{2R}(p)$. 
In order to obtain the higher-order estimates, one could apply \lemref{Higher derivatives flow} to the associated flow from \lemref{scaled f} on a fixed finite time interval $I$ away from the initial time and restricted to the unit ball. 

From the rescaled metric \eqref{associated flow vphi}, together with the corresponding curvature scaling relation \eqref{curvatures scale}, we have
\[
|Rm(t)|_{g(t)}=s^{-1}\,|Rm|_{g}\circ \sigma(t)\leq s^{-1} K.
\] An ingredient is to control the family of diffeomorphisms from Definition \ref{soliton diffeomorphisms} by using the bound on the vector field $\nabla f$. That ensures the flow trajectories starting from a smaller initial
ball will stay in a fixed ball for a short time interval.
Therefore,
\begin{align*}
&|\nabla^k Rm|_{g}\circ \sigma(t)
=s^{\frac{k+2}{2}}|\nabla^k Rm(t)|_{g(t)}\\
&\leq s^{\frac{k+2}{2}} \, C(k,K)\, t^{-\frac{k}{2}} \, s^{-1}
= s^{\frac{k}{2}} \, C(k,K)\, t^{-\frac{k}{2}}
\leq C(k,K).
\end{align*}
Finally, we scale the unit ball back to $B_{R}(p)$.
Then, for every $k\geq 0$,
\[
|\nabla^k Rm|_{g}R^{k+2}\leq C(k, K) \quad\text{on } B_R(p).
\]
\end{rem}
We will present a detailed elliptic estimate and sharpen these higher-order estimates by providing more explicit constants. The proof is quite involved, so we postpone it to Section \ref{Shi-type estimates}.

\subsection{Riemann curvature controls the \texorpdfstring{$G_2$}{G_2}-structures}
In this section, we derive useful equations for various curvatures and obtain estimates of the $G_2$-structures by using the Riemannian structures.
\subsubsection{Laplacian of the Riemann curvature tensor, the Ricci curvature tensor, and the scalar curvature}
The following lemma is obtained by a straightforward computation. We will use the notion $S_{ik}=\nabla_i \nabla_k S$.
\begin{lem}\label{lem:tau Rm 1}
 Let $(M,\vphi,g,\nabla f)$ be a gradient $G_2$-soliton. Then
the Riemann curvature tensor satisfies
\begin{equation}\label{Rm}
  \begin{aligned}
&\tri R_{ijkl}= \Big(R_{qlij,k}f_q-R_{qkij,l}f_q\Big)\\
&\quad-R_{qikp}R_{qjlp}-R_{qjkp}R_{iqlp}+R_{qpkl}R_{ijpq}
-R_{qilp}R_{qjpk}-R_{qjlp}R_{iqpk}\\
&\quad+2\lambda_S R_{ijkl}+\Big(R_{qlij}A_{qk}-R_{qkij}A_{ql}\Big)\\
 &\quad+\tfrac{1}{3}\Big(S_{ik}g_{lj}-S_{jk}g_{li}-S_{il}g_{kj}+S_{jl}g_{ki}\Big)\\
 &\quad+\Big(A_{lj,ik}-A_{li, jk}-A_{kj,il}+A_{ki, jl}\Big).
  \end{aligned}
\end{equation}
For the Ricci tensor we have
\begin{equation}\label{Ric}
  \begin{aligned}
  \tri R_{ik}&= R_{ik,q}f_q+2R_{qikp}R_{qp}
+2\lambda_S R_{ik}+\Big(R_{qi}A_{qk}-R_{qkij}A_{qj}\Big)\\
&\quad+\frac{1}{3}\Big(-S_{ik}+\tri Sg_{ki}\Big)+\Big(-A_{ji, jk}-A_{kj,ij}+\tri A_{ki}\Big).
  \end{aligned}
\end{equation}
The scalar curvature fulfills
\begin{align}
\tri S&= S_{p}f_p-2|Ric|^2
-\frac{2\lambda}{3}S+\frac{2}{3}S^2+2<Ric,A>-2
\nabla^j\nabla_i {A^i}_j\label{S}.
\end{align}
Furthermore, employing the contraction notation $\ast$ with respect to the metric $g$, these relations can be compactly written as
\begin{align}
\tri_f Rm=&\,Rm\ast Rm-\frac{2}{3}\lambda Rm+ Rm\ast T\ast T + \sum_{p+q=2}\nabla^{p} T\ast \nabla^qT ,\label{Rm brief}\\
\tri_f Ric=&\,Rm\ast Ric-\frac{2}{3}\lambda Ric + Rm\ast T\ast T + \sum_{p+q=2}\nabla^{p} T\ast \nabla^qT,\label{Ric brief}\\
\tri_f S=&\, -2|Ric|^2 -\frac{2\lambda}{3}S+\frac{2}{3}S^2+2<Ric,A>-2
\div\div A. \label{S brief}
\end{align}
\end{lem}
\begin{proof}
We first derive the identity \eqref{Rm} for $Rm$.
In normal coordinates, applying the second Bianchi identity
\[
R_{ijkl,p}+R_{ijlp,k}+R_{ijpk,l}=0
\]
to the term $R_{ijkl,p}$ gives
\begin{align*}
\tri R_{ijkl}
=R_{ijkl,pp}
&=-\bigl(R_{ijlp,kp}+R_{ijpk,lp}\bigr).
\end{align*}
Using the commutation formula for covariant derivatives,
\[
T_{i_1\cdots i_s,kp}-T_{i_1\cdots i_s,pk}
=\sum_{1\leq t\leq s}R_{qi_tkp}T_{i_1\cdots i_{t-1} q i_{t+1}\cdots i_s},
\]
we obtain
\begin{align*}
\tri R_{ijkl}
&=-\bigl(R_{ijlp,pk}+R_{ijpk,pl}\bigr)+Rm\ast Rm.
\end{align*}
Here
\begin{align*}
Rm\ast Rm
&=-R_{qikp}R_{qjlp}-R_{qjkp}R_{iqlp}-R_{qlkp}R_{ijqp}-R_{qpkp}R_{ijlq}\\
&\quad -R_{qilp}R_{qjpk}-R_{qjlp}R_{iqpk}-R_{qplp}R_{ijqk}-R_{qklp}R_{ijpq}.
\end{align*}
By applying the first Bianchi identity to the third and eighth terms, we may rewrite this as
\begin{align*}
Rm\ast Rm
&=-R_{qikp}R_{qjlp}-R_{qjkp}R_{iqlp}
+R_{qpkl}R_{ijpq}-R_{qk}R_{ijlq}\\
&\quad -R_{qilp}R_{qjpk}-R_{qjlp}R_{iqpk}-R_{ql}R_{ijqk}.
\end{align*}

For the first term in $\tri Rm$, the first Bianchi identity yields
\begin{align*}
R_{ijlp,pk}=R_{lpij,pk}
=-R_{lpjp,ik}-R_{lppi,jk}
=-R_{lj,ik}+R_{li,jk}.
\end{align*}
We now insert the soliton equation \eqref{soliton Ricci gradient} and use the commutation formula to simplify the expression $R_{lj,i}-R_{li,j}$:
\begin{align*}
R_{lj,i}-R_{li,j}
&=(-f_{lj}+\lambda_S g_{lj}+A_{lj})_i-(-f_{li}+\lambda_S g_{li}+A_{li})_j\\
&=R_{qlij}f_q+(\lambda_Sg_{lj})_i-(\lambda_Sg_{li})_j+A_{lj,i}-A_{li,j}.
\end{align*}
Substituting the soliton equation again to replace $f_{qk}$ in these terms, we get
\begin{align*}
-R_{ijlp,pk}
&=(R_{lj,i}-R_{li,j})_k\\
&=[R_{qlij}f_q+(\lambda_Sg_{lj})_i-(\lambda_Sg_{li})_j+A_{lj,i}-A_{li,j}]_k\\
&=R_{qlij,k}f_q+R_{qlij}f_{qk}
+(\lambda_Sg_{lj})_{ik}-(\lambda_Sg_{li})_{jk}+A_{lj,ik}-A_{li,jk}\\
&=R_{qlij,k}f_q-R_{qlij}R_{qk}+\lambda_S R_{klij}+R_{qlij}A_{qk}\\
&\quad+\frac{S_{ik}}{3}g_{lj}-\frac{S_{jk}}{3}g_{li}+A_{lj,ik}-A_{li,jk}.
\end{align*}
Similarly, for the second term we find
\begin{align*}
R_{ijpk,pl}
&=(R_{kj,i}-R_{ki,j})_l\\
&=R_{qkij,l}f_q-R_{qkij}R_{ql}+\lambda_S R_{lkij}+R_{qkij}A_{ql}\\
&\quad+\frac{S_{il}}{3}g_{kj}-\frac{S_{jl}}{3}g_{ki}+A_{kj,il}-A_{ki,jl}.
\end{align*}
Putting all of these identities together, we obtain the desired equation for $Rm$.

To derive equation \eqref{Ric} for $Ric$, we now contract \eqref{Rm} over the indices $j,l$, and apply the first Bianchi identity to the second term in the first line, which gives
\begin{align*}
    R_{qjij,k}=R_{qi,k},\quad R_{qkij,j}=R_{qjik,j}=
    -R_{qk,i}+R_{qi,k}.
\end{align*}
To complete the derivation of \eqref{Ric}, it only remains to contract the five terms appearing in the second line. We claim that the sum of the three middle terms vanishes:
\begin{align*}
R_{qjkp}R_{iqjp}+R_{qpkj}R_{ijpq}+R_{qijp}R_{qjpk}=0,
\end{align*}
which again follows from the first Bianchi identity. Indeed, define
\[
II_{ki}=R_{qjkp}R_{iqjp}=-R_{qpjk}R_{qpji}.
\]
Then
\begin{align*}
II_{ki}
&=R_{qpkj}R_{ijpq}\\
&=(R_{qkjp}+R_{qjpk})(R_{ipqj}+R_{iqjp})\\
&=R_{qkjp}R_{ipqj}+R_{qkjp}R_{iqjp}+R_{qjpk}R_{ipqj}+R_{qjpk}R_{iqjp}\\
&=-R_{pjkq}R_{ipjq}+2II_{ki}-R_{qjkp}R_{iqjp}.
\end{align*}
This shows the claimed identity. Using in addition
\[
\sum_i A_{ii}=-2S,
\]
we finally deduce the desired equation for $Ric$:
\begin{equation*}
  \begin{aligned}
  \tri R_{ik}
  &= R_{ik,q}f_q+2R_{qikp}R_{qp}+2\lambda_S R_{ik}
  +(R_{qi}A_{qk}-R_{qkij}A_{qj})\\
  &\quad+\frac{1}{3}\Big(7S_{ik}-2S_{ik}+\tri Sg_{ki}\Big)\\
  &\quad+(-2S_{ik}-A_{ji, jk}-A_{kj,ij}+\tri A_{ki}).
  \end{aligned}
\end{equation*}

Finally, we derive the scalar curvature equation \eqref{S} by contracting \eqref{Ric} once more and using
    \[
    A_{ij,ji}=A_{ij,ij}-R_{qj}A_{qj}+R_{qi}A_{qi}=A_{ij,ij}.
    \]
We could also derive the equation for the scalar curvature directly by taking the divergence of \eqref{nabla S gradient tau}, and using
\begin{align*}
g^{ij}\nabla_i(2R_{jp}f_p )
&=2g^{ij}(R_{jp,i}f_p +R_{jp}f_{pi} )\\
&=S_pf_p +2g^{ij} R_{jp}(-R_{pi} +\lambda_Sg_{pi}+A_{pi})\\
&=S_pf_p -2|Ric|^2+2\lambda_S S+2R_{ip}A_{pi}.
\end{align*}
Putting everything together, we obtain \eqref{S}, as claimed.
\end{proof}

\subsubsection{Laplacian of the torsion tensor}
Thanks to the presence of the $G_2$-structure, we obtain additional identities involving the torsion tensor $T$.
Expressed in terms of the Levi-Civita connection $\nabla$ associated with the metric $g$, the torsion tensor satisfies
\begin{align}
&T_{ij}=\frac{1}{24}\nabla_{i}\vphi_{pqr}\,{\psi_j}^{pqr}, \label{T and vphi 1}\\
& \nabla^i\vphi_{jkl}=T^{ip}\psi_{pjkl}.\label{T and vphi 2}
\end{align}
The dual $4$-form $\psi$ obeys
\begin{align}
&\psi_{ijkl}=\varphi_{pij}\varphi_{qkl} g^{pq}-g_{ik}g_{jl}+g_{il}g_{jk},\label{psi}\\
& \nabla_p\psi_{ijkl}=
-T_{pi}\vphi_{jkl}
+T_{pj}\vphi_{kli}
-T_{pk}\vphi_{lij}
+T_{pl}\vphi_{ijk}   .\label{nabla psi}
\end{align}

From the first and second Bianchi identities we deduce the contraction relations
\begin{align}
&R_{ipqr}\vphi^{pqr}=0, \label{Riem contraction 1}\\
&R_{ijpq,r}\vphi^{pqr}=0. \label{Riem contraction 2}
\end{align}

When $\vphi$ is closed, the torsion tensor $T$ satisfies the following contraction properties:
\begin{align}
&T_{pq}\vphi^{pqk}=0, \label{T contraction 1}\\
& T_{pq}\psi^{pqij}=-2T^{ij}, \label{T contraction 2}\\
& \Na^p T_{pj}=0,\label{T contraction 3}
\end{align}
its covariant derivative is given by
\begin{align} \label{eq:nabla_T}
        4\Na_k T_{ij}=& R_{jkpq}{\vphi_i}^{pq}-R_{ikpq}{\vphi_j}^{pq}+R_{jipq}{\vphi_k}^{pq}  \\
        &-2T_{jp}T_{kq}{\vphi_i}^{pq}+2T_{ip}T_{kq}{\vphi_j}^{pq}+2T_{ip}T_{jq}{\vphi_k}^{pq} \notag,
    \end{align}
and the Ricci tensor can be expressed as
\begin{align} \label{Ric S T}
R_{ij}={\vphi_i}^{pq} \nabla_p T_{qj}-{T_i}^{k}T_{kj}.
\end{align}
\begin{rem}
The above formulas involving the torsion tensor can be found in \cite{MR2282011,MR2559631} and \cite{MR3613456}*{Section 2}.
\end{rem}

\begin{lem}\label{tri vphi}
If $\vphi$ is closed, then
    \begin{align*}
    4\tri T_{ij}
&= 2R_{jp,q}{\vphi_i}^{pq}+R_{jkpq}{T^{kr}} \varphi_{ari}{\varphi}^{apq} \\
&-2R_{ip,q}{\vphi_j}^{pq}-R_{ikpq}{T^{kr}} \varphi_{arj}{\varphi}^{apq} \\
&+4R_{jkir}{T^{kr}}+2R_{jipq}T^{pq}\\
&-2\Na^kT_{jp}T_{kq}{\vphi_i}^{pq}-T_{jp}{A_{q}}^r{\psi_{ri}}^{pq}\\
&+2\Na^kT_{ip}T_{kq}{\vphi_j}^{pq}+T_{ip}{A_{q}}^r{\psi_{rj}}^{pq}\\
&+2\Na^kT_{ip}T_{jq}{\vphi_k}^{pq}
+2T_{ip}\Na^kT_{jq}{\vphi_k}^{pq}
+2T_{ip}{A^p}_{j}.
    \end{align*}
\end{lem}
\begin{proof}
Using \eqref{eq:nabla_T}, we compute each term in local coordinates:
    \begin{align*}
    4\tri T_{ij} &=4\Na^k\Na_kT_{ij} \\
        &= \Na^k(R_{jkpq}{\vphi_i}^{pq}-R_{ikpq}{\vphi_j}^{pq}+R_{jipq}{\vphi_k}^{pq}) \\
        &\quad+2\Na^k(-T_{jp}T_{kq}{\vphi_i}^{pq}+T_{ip}T_{kq}{\vphi_j}^{pq}+T_{ip}T_{jq}{\vphi_k}^{pq}).
    \end{align*}

    \textbf{1st term}
For the first term we use \eqref{T and vphi 2}:
    \begin{align*}
I:=\Na^k(R_{jkpq}{\vphi_i}^{pq})
=\Na^kR_{jkpq}{\vphi_i}^{pq}+R_{jkpq}{T^{kr}}{\psi_{ri}}^{pq}.
    \end{align*}
By the second Bianchi identity
\[
R_{jkpq,k}=-R_{jkqk,p}-R_{jkkp,q}=-R_{jq,p}+R_{jp,q},
\]
we obtain
        \begin{align*}
I_1=\Na^kR_{jkpq}{\vphi_i}^{pq}
=(-R_{jq,p}+R_{jp,q}){\vphi_i}^{pq}
=2R_{jp,q}{\vphi_i}^{pq}.
    \end{align*}
Substituting \eqref{psi} into the second part of $I$ gives
        \begin{align*}
I_2&=R_{jkpq}{T^{kr}}{\psi_{ri}}^{pq}\\
&=R_{jkpq}{T^{kr}} \big[\varphi_{ari}{\varphi_{b}}^{pq} g^{ab}-{g_{r}}^p{g_{i}}^q+{g_{r}}^q{g_{i}}^p\big]\\
&=R_{jkpq}{T^{kr}} \varphi_{ari}{\varphi}^{apq}
+2R_{jkir}{T^{kr}}.
    \end{align*}
Altogether this yields
    \begin{align*}
I=2R_{jp,q}{\vphi_i}^{pq}+R_{jkpq}{T^{kr}} \varphi_{ari}{\varphi}^{apq}
+2R_{jkir}{T^{kr}}.
    \end{align*}

\textbf{2nd term}
In the same way, the second term is
    \begin{align*}
II&= -\Na^k(R_{ikpq}{\vphi_j}^{pq})\\
&=-2R_{ip,q}{\vphi_j}^{pq}-R_{ikpq}{T^{kr}} \varphi_{arj}{\varphi}^{apq}
-2R_{ikjr}{T^{kr}}.
    \end{align*}
Exchanging $k,r$ we get
    \begin{align*}
-R_{ikjr}{T^{kr}}=-R_{irjk}{T^{rk}}=R_{jkir}{T^{kr}},
    \end{align*}
which matches the last term in $I$.

\textbf{3rd term}
For the third term, we employ \eqref{T and vphi 2}, \eqref{T contraction 2} and \eqref{Riem contraction 2}:
\begin{align*}
\Na^k(R_{jipq}{\vphi_k}^{pq})
&=\Na^kR_{jipq}{\vphi_k}^{pq} +R_{jipq}\Na^k{\vphi_k}^{pq} \\
&=R_{jipq}\Na^k{\vphi_k}^{pq}
=R_{jipq} T^{kr}{\psi_{rk}}^{pq}\\
&=2R_{jipq}T^{pq}.
\end{align*}

\textbf{4th term}
For the fourth term we use \eqref{T and vphi 2} and \eqref{T contraction 3}:
\begin{align*}
IV&:=-2\Na^k(T_{jp}T_{kq}{\vphi_i}^{pq})
=-2\Na^kT_{jp}T_{kq}{\vphi_i}^{pq}-2T_{jp}T_{kq}\Na^k{\vphi_i}^{pq}.
\end{align*}
From \eqref{T and vphi 2} we know
\begin{align*}
\Na^k{\vphi_i}^{pq}
=T^{kr}{\psi_{ri}}^{pq}.
\end{align*}
Hence the second part of $IV$ simplifies to
\begin{align*}
IV_2&:=-2T_{jp}T_{kq}\Na^k{\vphi_i}^{pq}
=-2T_{jp}T_{kq}T^{kr}{\psi_{ri}}^{pq}
=-T_{jp}{A_{q}}^r{\psi_{ri}}^{pq}.
\end{align*}
Here we use $T_{kq}T^{kr}=\frac{1}{2}{A_q}^r$.
Using \eqref{psi} and $\tr A=-2S$ again,
\begin{align*}
IV_2
&=-T_{jp}{A_{q}}^r(\varphi_{ari}{\varphi_{b}}^{pq} g^{ab}-{g_{r}}^p{g_{i}}^q+{g_{r}}^q{g_{i}}^p)\\
&=-T_{jp}{A_{q}}^r \varphi_{ari}\varphi^{apq}+T_{jp}{A_{i}}^p+2ST_{ji}.
\end{align*}
Thus the fourth term is
\begin{align*}
IV
=-2\Na^kT_{jp}T_{kq}{\vphi_i}^{pq}-T_{jp}{A_{q}}^r{\psi_{ri}}^{pq}.
\end{align*}

\textbf{5th term}
Analogously, the fifth term becomes
\begin{align*}
2\Na^k(T_{ip}T_{kq}{\vphi_j}^{pq})
=2\Na^kT_{ip}T_{kq}{\vphi_j}^{pq}+T_{ip}{A_{q}}^r{\psi_{rj}}^{pq}.
\end{align*}

\textbf{6th term}
From \eqref{T and vphi 2} and \eqref{T contraction 2} we obtain
        \begin{align*}
\Na^k {\vphi_k}^{pq}
= T^{kr}{\psi_{rk}}^{pq}
=2 T^{pq}.
    \end{align*}
Therefore, using $T_{jq}T^{pq}=\frac{1}{2}{A_j}^p$, the sixth term is
        \begin{align*}
2\Na^k(T_{ip}T_{jq}{\vphi_k}^{pq})
&=
2\Na^kT_{ip}T_{jq}{\vphi_k}^{pq}
+2T_{ip}\Na^kT_{jq}{\vphi_k}^{pq}
+4T_{ip}T_{jq}  T^{pq}\\
&=
2\Na^kT_{ip}T_{jq}{\vphi_k}^{pq}
+2T_{ip}\Na^kT_{jq}{\vphi_k}^{pq}
+2T_{ip}{A^p}_{j}.
    \end{align*}
Adding these six contributions gives the claimed identity.
\end{proof}

\begin{prop}\label{tri T soliton}
        Let $(M,\vphi,g,\nabla f)$ be a gradient $G_2$-soliton. Then
    \begin{align*}
    4\tri T_{ij}
=&-{R_{qpj}}^rf_r{\vphi_i}^{pq}-\frac{4}{3}T^{kl}\nabla_qT_{kl} \cdot {\vphi_{ij}}^{q}-4\nabla_q{T_j}^{k}T_{kp}{\vphi_i}^{pq}-4{T_j}^{k}\nabla_qT_{kp}{\vphi_i}^{pq}\\
&+{R_{qpi}}^rf_r{\vphi_j}^{pq}+\frac{4}{3}T^{kl}\nabla_qT_{kl} \cdot {\vphi_{ji}}^{q}+4\nabla_q{T_i}^{k}T_{kp}{\vphi_j}^{pq}+4{T_i}^{k}\nabla_qT_{kp}{\vphi_j}^{pq}\\
&+R_{jkpq}{T^{kr}} \varphi_{ari}{\varphi}^{apq}
-R_{ikpq}{T^{kr}} \varphi_{arj}{\varphi}^{apq} \\
&+4R_{jkir}{T^{kr}}+2R_{jipq}T^{pq}\\
&-2\Na^kT_{jp}T_{kq}{\vphi_i}^{pq}-T_{jp}{A_{q}}^r{\psi_{ri}}^{pq}\\
&+2\Na^kT_{ip}T_{kq}{\vphi_j}^{pq}+T_{ip}{A_{q}}^r{\psi_{rj}}^{pq}\\
&+2\Na^kT_{ip}T_{jq}{\vphi_k}^{pq}
+2T_{ip}\Na^kT_{jq}{\vphi_k}^{pq}
+2T_{ip}{A^p}_{j}.
    \end{align*}
In a more compact schematic form, this can be written as
\begin{align*}
            \tri T=Rm\ast \nabla f \ast\vphi+Rm\ast T\ast\vphi\ast\vphi+Rm\ast T+\nabla T\ast T\ast \vphi+T\ast T\ast T\ast\psi.
        \end{align*}
\end{prop}
\begin{proof}
Substituting the soliton equation \eqref{soliton Ricci gradient}
\[
        Ric+\nabla^2 f= \lambda_S \cdot g+A
\]
into the first term of \lemref{tri vphi}, we obtain
\begin{align*}
R_{jp,q}{\vphi_i}^{pq}=(-f_{jp,q}+\tfrac{1}{3}S_q \cdot g_{jp}+A_{jp,q}){\vphi_i}^{pq}.
\end{align*}
For the first contribution in this expression, swapping the indices $p$ and $q$ gives
\[
f_{jp,q}{\vphi_i}^{pq}=f_{jq,p}{\vphi_i}^{qp}=-f_{jq,p}{\vphi_i}^{pq},
\]
and therefore
\begin{align*}
f_{jp,q}{\vphi_i}^{pq}=\frac{1}{2}(f_{jp,q}-f_{jq,p}){\vphi_i}^{pq}.
\end{align*}
Using the Ricci commutation formula
\[
f_{jp,q}-f_{jq,p}={R_{qpj}}^rf_r,
\]
we deduce that
\begin{align*}
f_{jp,q}{\vphi_i}^{pq}=\frac{1}{2}{R_{qpj}}^rf_r{\vphi_i}^{pq}.
\end{align*}
For the remaining two contributions, differentiating \lemref{ABS} directly gives
\begin{align*}
&S_q=-2T^{kl}\nabla_q T_{kl},\\
&A_{jp,q}=(-2{T_j}^{k}T_{kp})_q=-2\nabla_q{T_j}^{k}T_{kp}-2{T_j}^{k}\nabla_qT_{kp}.
\end{align*}
Putting these pieces together, the first term can be rewritten as
\begin{align*}
R_{jp,q}{\vphi_i}^{pq}
=&-\frac{1}{2}{R_{qpj}}^rf_r{\vphi_i}^{pq}\\
&-\frac{2}{3}T^{kl}\nabla_qT_{kl} \cdot {\vphi_{ij}}^{q}-2\nabla_q{T_j}^{k}T_{kp}{\vphi_i}^{pq}-2{T_j}^{k}\nabla_qT_{kp}{\vphi_i}^{pq}.
\end{align*}
Carrying out the same computation for $R_{ip,q}{\vphi_j}^{pq}$ and
inserting these identities back into \lemref{tri vphi}, we arrive at the claimed formula for $\tri T$ in the setting of a gradient $G_2$-soliton.
\end{proof}

\begin{rem}
This expression for $\tri T$ in the $G_2$-soliton setting is the analogue of $(\p_t-\tri)T$ along the $G_2$-Laplacian flow; see \cite{MR3613456}*{Equation (3.13)}.
\end{rem}
\subsubsection{Bounds for the $G_2$-structure in terms of the Riemann curvature}
Before stating the relation between higher-order derivatives of the $G_2$-structure $\vphi$ and the curvature, we introduce the notions of domination and equivalence.
\begin{defn}[Domination and equivalence]
We say that a quantity $a$ dominates $b$, written $b\lesssim a$, if there exists a constant $C\geq 1$ such that $$b\leq C\cdot a.$$
We say that $a$ is equivalent to $b$, denoted $a\thickapprox b$, if there exists a constant $C\geq 1$ with
\[
C^{-1}\cdot a\leq b\leq C\cdot a.
\]
The constants $C$ are independent of the particular choice of metric and depend only on the dimension and derivative order.
\end{defn}

The inequalities we establish show that all derivatives of $$\vphi,\quad \psi,\quad T,\quad A$$ are bounded above by a constant multiple of products involving the curvature $Rm$ and its derivatives $\Na^k Rm$.

\begin{defn}\label{defkRm1}
For any $k\geq 0$, we set a vector $$v_k:=(n_0,\cdots,n_k)$$ such that $2n_0$ and other components are non-negative integers, that is $$2n_0\in \mathbb{N}_0, \quad n_i\in \mathbb{N}_0, \quad \forall1\leq i\leq k.$$ For each vector $v_k$, we associate it with a non-negative integer $$\|v_k\|:=\sum_{0\leq i\leq k}(i+2)n_i=2n_0+3n_1+\cdots+(k+2)n_k.$$
\end{defn}
\begin{defn}\label{defkRm2}
    We also define a quantity associated with the curvature $Rm$ and the vector $v_k$
 $$\|Rm\|_{v_k}:=\prod_{0\leq i\leq k}|\nabla^i Rm|^{n_{i}}:= |Rm|^{n_0}|\nabla Rm|^{n_1}\cdots|\nabla^k Rm|^{n_{k}}.$$
\end{defn}
\begin{prop}\label{lem:tau Rm k}
We obtain that
        \begin{align}
\label{scalevphi}
            & |\nabla^{k} \vphi|\lesssim \sum_{\|v_k\|=k } \|Rm\|_{v_k} ,\\
\label{scalepsi}
            & |\nabla^{k} \psi|\lesssim \sum_{\|v_k\|=k }\|Rm\|_{v_k},\\
\label{scaletau}
            & |\nabla^{k} T|\lesssim \sum_{\|v_k\|=k +1}\|Rm\|_{v_k},\\
\label{scaleA}
            &|\nabla^{k} A|\lesssim \sum_{\|v_k\|=k +2 } \|Rm\|_{v_k},
        \end{align}
        Moreover, $\Na^k S$ has the same estimate as $\Na^k A$.
    \end{prop}

\begin{proof}
The proof is based on \eqref{T and vphi 2} and \eqref{nabla psi},
        \begin{align}\label{gradient vphi psi}
        \nabla\vphi=T\ast\psi,\quad \nabla\psi=T\ast\vphi,
        \end{align}
together with the expressions for $\nabla T$ in \eqref{eq:nabla_T} and for $\nabla A$ in \lemref{ABS},
        \begin{align}\label{gradient T A}
\Na T=Rm\ast\vphi+T\ast T\ast \vphi,\quad \nabla A=\nabla T\ast T.
        \end{align}

For $k=0,1,2$, the required estimates will be established in \lemref{lem:tau Rm 2} below.
We now use induction to obtain the higher-order estimates. Assume that the desired bounds hold up to order $k$; we will verify them for order $k+1$.
Differentiate \eqref{gradient vphi psi} and \eqref{gradient T A} a total of $k$ times to get
\begin{equation}\label{gradient vphi psi k}
  \begin{aligned}
&\nabla^{k+1}\vphi=\sum_{p+q=k}\nabla^pT\ast\nabla^q\psi,\quad \nabla^{k+1}\psi=\sum_{p+q=k}\nabla^pT\ast\nabla^q\vphi,\\
        &\Na^{k+1} T=\sum_{p+q=k}\nabla^p Rm\ast\nabla^q\vphi+\sum_{p+q+r=k}\nabla^pT\ast \nabla^qT\ast \nabla^r\vphi,\\
        & \nabla^{k+1} A=\sum_{p+q=k+1}\nabla^{p} T\ast \nabla^qT.
  \end{aligned}
\end{equation}
First, by the induction hypothesis we obtain
\begin{align*}
|\Na^{k+1}\vphi|
&\lesssim \sum_{p+q=k} |\Na^p T|\, |\Na^{q} \psi|\\
&\lesssim\sum_{p+q=k} \sum_{\|v_p\|=p +1} \prod_{0\leq i\leq p} |\nabla^iRm|^{n_i} \sum_{\|v_q\|=q }   \prod_{0\leq j\leq q} |\nabla^jRm|^{\tilde n_j}.
\end{align*}
We extend $v_p$ and $v_q$ to vectors $v^p_{k+1}$ and $v^q_{k+1}$ in $\mathbb{R}^{k+1}$ by appending zero entries. We then define a new vector $v_{k+1}$ by adding the corresponding components of $v^p_{k+1}$ and $v^q_{k+1}$. A direct computation shows that
\[
\|v_{k+1}\| = p+1+q = k+1.
\]
Therefore, \eqref{scalevphi} is valid for $k+1$. An identical reasoning proves \eqref{scalepsi} for $k+1$ as well.

Next, we estimate
\begin{align}\label{nabla k+1 T}
|\Na^{k+1}T|
\lesssim \sum_{p+q=k}  |\Na^p Rm|\, |\Na^{q} \vphi|+\sum_{p+q+r=k}|\Na^p T|\, |\Na^q T|\, |\Na^{r} \vphi|.
\end{align}
In the first sum, corresponding to $|\Na^p Rm|$, we introduce a vector $v_p$ whose $(p+1)$-st component is $1$ and whose remaining components are all zero. By definition, this gives $\|v_p\|=p+2$. Using \eqref{scalevphi}, we obtain
\[
|\nabla^{q} \vphi|\lesssim \sum_{\|v_q\|=q } \|Rm\|_{v_q}.
\]
We now view these vectors as elements of $\mathbb{R}^k$ by padding them with zeros and then add them together to form a new vector $v_k$ satisfying
\[
\|v_k\|=\|v_q\|+\|v_p\|=q+p+2=k+2.
\]
For the second sum, the vectors $v_p$, $v_q$, and $v_r$ again combine to form a vector $v_k$ in the same manner, and
\[
\|v_k\| = \|v_p\| + \|v_q\| + \|v_r\|
= (p+1) + (q+1) + r
= k+2.
\]
By appending a zero entry to $v_k$, we construct a vector $v_{k+1}$, from which \eqref{scaletau} follows.

Finally, for $A$ we apply \eqref{gradient vphi psi k} and obtain
\begin{align*}
|\Na^{k+1}A|
&\lesssim
\sum_{p+q=k+1}|\nabla^{p} T|\,|\nabla^qT|\\
&\lesssim
\sum_{p+q=k+1} \sum_{\|v_p\|=p+1}\prod_{0\leq i\leq p}|\nabla^i Rm|^{n_{i}}  \sum_{\|v_q\|=q +1}\prod_{0\leq j\leq q}|\nabla^j Rm|^{n_{j}}.
\end{align*}
Here $v_p$ and $v_q$ combine to form a new vector $v_{k+1}$, as in the preceding argument, and
\[
\|v_{k+1}\|=\|v_p\|+\|v_q\|=p+q+2=k+3.
\]
This establishes \eqref{scaleA}.
\end{proof}

\begin{lem}\label{lem:tau Rm 2}
The estimates in Proposition \ref{lem:tau Rm k} are valid for $k=0,1,2$, namely
        \begin{align*}
        &|\vphi|\lesssim C,\quad |\psi|\lesssim C,\quad
        |T|\lesssim|Rm|^{\frac{1}{2}},\quad |A| \lesssim|Rm|,\\
        &|\nabla\vphi|\lesssim|Rm|^{\frac{1}{2}},\quad |\nabla\psi|\lesssim|Rm|^{\frac{1}{2}},\quad
     |\Na T|\lesssim |Rm|, \quad |\Na A|\lesssim|Rm|^{\frac{3}{2}},\\
     &|\nabla^{2}\vphi| \lesssim |Rm|,\quad
 |\nabla^{2}\psi| \lesssim |Rm|,\quad
|\Na^{2} T|\lesssim |\nabla Rm|+|Rm|^{\frac{3}{2}},\\
& |\nabla^{2} A | \lesssim |Rm|^{\frac{1}{2}}|\nabla Rm|+|Rm|^{2}.
        \end{align*}
\end{lem}
\begin{proof}
Using $|T|^2=-S$ from \lemref{ABS}, we immediately obtain \eqref{scaletau} and \eqref{scaleA} for $k=0$:
        \begin{align*}
        |T|= |S|^{\frac{1}{2}}\lesssim|Rm|^{\frac{1}{2}},\quad |A|\lesssim |T|^2= |S| \lesssim|Rm|.
        \end{align*}
For $\vphi$ and $\psi$, \eqref{scalevphi} and \eqref{scalepsi} for $k=0$ follow from the fixed norms
\begin{align*}
|\vphi|^2=7\times 3!,\qquad |\psi|^2=7\times 4!.
\end{align*}
For $k=1$, \eqref{gradient vphi psi} gives
        \begin{align*}
        |\nabla\vphi|\lesssim |T|\lesssim|Rm|^{\frac{1}{2}},\quad |\nabla\psi|\lesssim |T|\lesssim|Rm|^{\frac{1}{2}}.
        \end{align*}
From \eqref{gradient T A} we obtain
        \begin{align*}
            & |\Na T|\lesssim |Rm|+|T|^2\lesssim |Rm|, \quad |\Na A|\lesssim |\nabla T|\,|T|\lesssim|Rm|^{\frac{3}{2}},
        \end{align*}
which are precisely \eqref{scaletau} and \eqref{scaleA} for $k=1$.
For $k=2$, \eqref{gradient vphi psi k} becomes
\begin{equation}\label{gradient vphi psi 2}
  \begin{aligned}
\nabla^{2}\vphi&=\sum_{p+q=1}\nabla^pT\ast\nabla^q\psi ,\quad
\nabla^{2}\psi=\sum_{p+q=1}\nabla^pT\ast\nabla^q\vphi ,\\
\Na^{2} T&=\sum_{p+q=1}\nabla^p Rm\ast\nabla^q\vphi+\sum_{p+q+r=1}\nabla^pT\ast \nabla^qT\ast \nabla^r\vphi,\\
 \nabla^{2} A&=\sum_{p+q=2}\nabla^{p} T\ast \nabla^qT.
  \end{aligned}
\end{equation}
Substituting the bounds already established for $k\leq 1$ into \eqref{gradient vphi psi 2} yields the desired estimates for $k=2$.
\end{proof}
\subsection{Analysing examples of complete gradient \texorpdfstring{$G_2$}{G_2}-solitons}\label{examples}
Currently, no sharp a priori bounds are known for the potential function $f$ satisfying the equations in Proposition \ref{lem:Sc df}, as already noted in Remark \ref{bounds on S anf f remark}.

We briefly review the constructions in \cite{MR4555993,MR4349461,MR4728483,arXiv:2112.09095,arXiv:2501.05437} of complete gradient $G_2$-solitons satisfying the conditions of Definition \ref{3 assumptions}, thereby demonstrating that the moduli space $\mathcal M(\Lambda,F)$ contains many explicit examples.
\begin{lem}\label{distance compare}
If the metric has the form
\[g=dt^2+g_t\text{ on }M=\mathbb R\times N,\]
where $N$ is a compact manifold. Let $x=(t,a),y=(s,b)$ be two points in $M$. Then the $g$-distance between these two points has a bound
\begin{align*}
|t-s|\leq d_g(x,y)\leq |t-s|+diam(N, g_{t}).
\end{align*}
\end{lem}
\begin{proof}
The distance upper bound is obtained from first moving inside the $t$-slice and then along a normal $t$-curve.
\end{proof}
\subsubsection{An explicit complete gradient shrinking soliton $M_1$}
In \cite{MR4349461}*{Section 5.2}, Fowdar uses the Apostolov-Salamon ansatz \cite{MR2044890} to construct an explicit, complete gradient shrinking soliton, which we denote by $M_1$. The construction is as follows:
\begin{align*}
&g_\vphi=dt^2+b^{-0.5} \exp(\frac{bt}{2}) \sum_{1\leq i\leq 4}( e^i)^2+b\exp(bt)\sum_{i=5,6}( e^i)^2 ,\\
&X=\frac{15}{2}b\cdot \p_t, \quad b>0,\quad f(t)=\frac{15}{2}b\cdot t,\\
&\lambda=-\frac{9}{2}b^2,\quad S=-\frac{27}{4}b^2,\quad \Vol=O(\exp(2bt)).
\end{align*}
We denote the basepoint $p=(t_p,a_p)$. We could simply normalise the potential function
\[f_p(x)=\frac{15}{2}b\cdot (t-t_p)\]
such that $f_p(p)=0$.
\begin{lem}
From this we deduce that
\begin{align*}
 &M_1\in \mathcal M(\Lambda,F), \quad \Lambda=\frac{27}{4}b^2,\\
 &-\lambda+S=-\frac{9}{4}b^2<0,\quad 7\lambda+2S=-45b^2<0.
\end{align*}
Moreover, $|\nabla f|=15b/2=F$.
\end{lem}
\begin{proof}
It follows directly from the construction.
\end{proof}
We denote $D=diam(N, g_{t_p})$. Since $|f_p(x)|=\frac{15}{2}b\cdot |t-t_p|$, we see that
the potential function has linear growth
\begin{align*}
\frac{15}{2}b\cdot (d(p,x)-D)\leq |f_p(x)|\leq\frac{15}{2}b\cdot d(p,x),
\end{align*}
due to \lemref{distance compare}.

In \cite{MR1969782}, Cleyton-Swann studied $G_2$-structures that admit a cohomogeneity-one action by a compact Lie group $G$. Equivalently, the $G_2$-structure is invariant under the $G$-action, and a typical $G$-orbit in $M$ has dimension 6. The simple Lie groups $G$ that can occur in this setting are $G_2$, $SU(3)$, or $Sp(2)$ \cite{MR1969782}*{Theorem 3.1}.

The construction is presented in \cite{MR1969782}*{Section 5}.
We take the maximal torus \(T^2 \subset SU(3)\) to be the subgroup of diagonal matrices.
The associated flag manifold is
\[
\mathcal{F} = SU(3)/T^2.
\]
At a chosen basepoint \(p_0\), its tangent space can be identified with the quotient
\[
T_{p_0}\mathcal{F} \cong \mathfrak{su}(3)/\mathfrak{t},
\]
where \(\mathfrak{su}(3)\) is the real Lie algebra of traceless skew-Hermitian \(3\times 3\) complex matrices (of real dimension \(8\)), and \(\mathfrak{t} = \mathrm{Lie}(T^2)\) is the Cartan subalgebra consisting of diagonal traceless skew-Hermitian matrices (of real dimension \(2\)).


The $SU(3)$-invariant Riemannian metric is given by
\begin{align*}
g_\vphi = dt^2 + f_1^2 (e_1^2 + e_2^2) + f_2^2 (e_3^2 + e_4^2) + f_3^2 (e_5^2 + e_6^2),
\end{align*}
where $\{e_1,\dots,e_6\}$ is the coframe dual to the chosen basis of the tangent space $T_{p_0}\mathcal{F}$, \cite{MR4728483}*{Equation 5.2}.
In the $Sp(2)$-invariant situation, one further obtains the condition $f_2 = f_3$, see also \cite{arXiv:2112.09095}*{Proposition 4.17, Proposition 4.39}

The Gaussian soliton determined by $$f_1 = f_2 = f_3 = \frac{t}{2}$$ yields a conical Riemannian metric, and the corresponding $G_2$-structure is torsion-free. The soliton vector field is given by $X = -\frac{\lambda}{3}\,t\,\partial_t$, and the potential function is $$f = -\frac{\lambda}{6}t^2;$$ see \cite{arXiv:2112.09095}*{Example 4.33, Lemma 7.15}. Thus, the potential function exhibits quadratic growth in the distance. However, it is incomplete at the cone tip.

Cohomogeneity-one, $G$-invariant, complete gradient $G_2$-solitons were constructed by Haskins and Nordstr\"om in \cite{arXiv:2112.09095}. In this section, we analyse the scalar curvature and potential function associated with two concrete examples presented in their work \cite{arXiv:2112.09095}. In addition, in \cite{arXiv:2501.05437}, Haskins, Juneman, and Nordstr\"om construct further cohomogeneity-one, $Sp(2)$-invariant, complete expanding gradient $G_2$-solitons, although in that case the solutions are not given in an explicit form.

\subsubsection{An explicit cohomogeneity-one, $Sp(2)$-invariant, complete gradient shrinking soliton $M_2$}
In \cite{arXiv:2112.09095}*{Theorem D, Section 6.4}, it is shown that these solitons are asymptotic to non-torsion-free, $G$-invariant $G_2$-cones. Their construction is achieved by the following explicit values:
\begin{align*}
&\lambda=-\frac{9}{4b^2},\forall b>0, \quad f_1=t, \quad f_2=f_3=\frac{\sqrt{4b^2+t^2}}{2},\quad \tau_2=-\frac{1}{2}t,\\
&X=\Bigl[\frac{3t}{4b^2}+\frac{4t}{4b^2+t^2}\Bigr] \p_t, \quad f(t)=\frac{3t^{2}}{8b^{2}} + 2\ln(4b^{2} + t^{2}) + C.
\end{align*}
The domain $t\geq 0$, the singular orbit is $S^4$, and the underlying manifold is $\Lambda^2_-S^4$.
\begin{lem}
The scalar curvature of $(M_2,g_\vphi)$ is bounded, namely
\begin{align*}
&S=-\frac{12t^2}{(4b^2+t^2)^2}\in \Bigl[-\frac{3}{4b^2},0\Bigr],
\\
 &-\lambda+S\in [\frac{3}{2b^2},\frac{9}{4b^2}],\quad 7\lambda+2S\in [-\frac{69}{4b^2},-\frac{63}{4b^2}] .
\end{align*}
In addition, the lower bound of $Ric_f$ satisfies
\begin{align*}
\lambda_S\in [\frac{1}{2b^2},\frac{3}{4b^2}],\quad Ric_f\geq \frac{1}{2b^2}g.
\end{align*}
\end{lem}
\begin{proof}
By \cite{arXiv:2112.09095}*{Proposition 5.24}, we have
\begin{align*}
\tau_1=-\frac{2f_1^2}{f_2^2}\tau_2=\frac{4t^3}{4b^2+t^2}.
\end{align*}
Hence, the squared norm of the torsion form is given by
\begin{align*}
\|\tau\|^2=\frac{\tau_1^2}{f_1^4}+2\frac{\tau_2^2}{f_2^4}
=\frac{24t^2}{(4b^2+t^2)^2}.
\end{align*}
Consequently, the scalar curvature is
$$S
=-\frac{1}{2}\|\tau\|^2
=-\frac{12t^2}{(4b^2+t^2)^2}.$$
We see that the function
$$-\frac{12x}{(4b^2+x)^2}$$
attains its minimum value $-\frac{3}{4b^2}$ at $x=4b^2$.
This yields the desired lower bound for the scalar curvature.
\end{proof}
\begin{lem}
$M_2$ lies in $\mathcal M(\Lambda,F)$ with $$\Lambda=\frac{3}{4b^2}, \quad F(x)= \frac{3x}{4b^{2}} +\frac{1}{b}.$$
\end{lem}
\begin{proof}
We compute that
\begin{align*}
|\nabla f|=\frac{3t}{4b^2}+\frac{4t}{4b^2+t^2}.
\end{align*}
Using $\frac{4t}{4b^2+t^2}\leq \frac{1}{b}$ and \lemref{distance compare}, we have that
\begin{align*}
|\nabla f|\leq\frac{3d(x,p)}{4b^2}+\frac{1}{b}.
\end{align*}
This shows that we may choose $F$ exactly as claimed in the lemma.
\end{proof}

\subsubsection{An explicit cohomogeneity-one, $SU(3)$-invariant, complete gradient steady soliton $M_3$}
The complete, gradient steady $G_2$ solitons that are $SU(3)$-invariant and of cohomogeneity-one fall into two distinct classes \cite{arXiv:2112.09095}*{Theorem G, Theorem 9.7}. One class is asymptotically conical with decay rate $-1$. The other class, which we denote by $M_3$, is described by explicit formulas. This soliton has exponential volume growth and
\begin{align*}
&f_1=2\sinh\frac{t}{2}, \quad f_2=\sqrt{1+\exp(t)}, \quad f_3=\sqrt{1+\exp(-t)}, \\
& X=\tanh \frac{t}{2} \p_t,\quad f(t)=\ln(1 + \cosh t) + C,\quad
S=-3+\frac{3}{4}\sech^2\frac{t}{2}.
\end{align*}

\begin{lem}
The scalar curvature of $(M_3,g_\vphi)$ is bounded below:
\begin{align*}
-3<S\leq-\frac{9}{4}<\lambda=0 .
\end{align*}
$M_3$ lies in the class $\mathcal M(\Lambda,F)$ with
\begin{align*}
\Lambda=3,\quad |\nabla f|=|\tanh \frac{t}{2}|\leq 1,\quad F=1.
\end{align*}
\end{lem}
\begin{proof}
From the inequality $0<\sech x\leq 1$ valid for all $x$, we immediately deduce the bounds for the scalar curvature. 
Furthermore, the potential function satisfies
\begin{align*}
 f(t)
 &=\ln(1 + \cosh t) + C
 =\log \frac{2+\exp(t)+\exp(-t)}{2}+C\\
 &=\log\frac{\bigl(\exp(\frac{t}{2})+\exp(-\frac{t}{2})\bigr)^2}{2}+C.
\end{align*}
Since the coefficient of the vector field is strictly increasing and
\begin{align*}
\tanh \frac{t}{2} = \frac{\exp(t) - 1}{\exp(t) + 1}\in (-1,1),
\end{align*}
the two conditions required to lie in $\mathcal M(\Lambda,F)$ are fulfilled.\end{proof}
We choose the basepoint $p$ at the singular orbit.
When $t=0$, we have $f(0)=\log 2$. The normalisation of the potential function is 
\[f_p(t)=f(t)-\log 2=2\log\cosh\frac{t}{2}=t-2\log 2+o(1),\] which has linear growth by \lemref{distance compare}.
\subsubsection{A homogeneous complete gradient steady soliton $M_4$}
Fino and Raffero \cite{MR4728483} construct these solitons on an almost nilpotent Lie group. The associated Riemannian metric is defined by
$$
g_\vphi=\sum_{i=1}^7 e^i \otimes e^i,
$$
where $\{e^1,\cdots,e^7\}$ denotes the following basis of left-invariant $1$-forms \cite{MR4728483}*{Equation (2.1)},
\begin{align*}
&e^1=\exp(t) (x_6\, dt+dx_6),\quad
e^2=-\exp(t) (x_5\, dt+dx_5), \\
&e^3=\exp(-t)dx_2, \quad
e^4=\exp(t) dx_1, \\
&e^5=2\exp(t) (x_6 \exp(t) dx_1+dx_4), \quad
e^6=2\exp(t) (x_5 \exp(t) dx_1+dx_3), \\
&e^7=-dt.
\end{align*}
As shown in \cite{MR4728483}*{Page 2202}, the potential function $f:H\to\mathbb R$ is given by
\begin{align*}
f=4t+b, \quad t\in\mathbb R,\ \text{for some constant }b\in\mathbb R.
\end{align*}
\begin{lem}
The scalar curvature of $(M_4,g_\vphi)$ is constant.
Moreover, the steady soliton $M_4$ belongs to the class $\mathcal M(\Lambda,F)$ with
\begin{align*}
\Lambda=12,\quad |\nabla f|=4,\quad F=4.
\end{align*}
\end{lem}
\begin{proof}
Although this metric is not a warped product and includes cross terms, it is homogeneous as a Riemannian metric, and hence its scalar curvature is constant. 

Let $E_7$ denote its dual to $e^7$. From the expressions of $e^7$ and $f$,
we obtain
\[
    df=4\,dt=-4e^7.
\]
Consequently,
$
    \nabla f=(df)^\sharp=-4E_7.
$
Since $E_7$ is a unit vector field,
\[
    |\nabla f|_{g_\vphi}^2
    =16.
\]
Thus we could choose
$
    F(r)\equiv4.
$

The torsion is \[\tau=2\exp(12)+2\exp(34)-4\exp(56),\] from \cite{MR4728483}*{Page 2202}.
Therefore, the scalar curvature is $-12$ which gives the choice of $\Lambda$.
\end{proof}
The normalised potential function is $f_p(x)=4(t-t_p)$.

\subsubsection{Complete gradient steady solitons $M_5$ and $M_6$}
There exist further examples of gradient solitons.
On the product manifold $\mathbb R \times N$, explicit complete gradient steady solitons have been constructed in \cite{MR4555993}.

$M_5=\mathbb R \times N$, where $N$ is the twistor space of an anti-self-dual Ricci-flat $4$-manifold $M$ \cite{MR4555993}*{Theorem 6.2}. The vector field is $X=\partial_t f(t)\p_t$ and the potential function is given by
\begin{align*}
f(t)=-2k_1\int_1^t u(s)^{-2} \, ds, \quad
u(s)=k_1\bigl[W\!\bigl(-\frac{1}{k_1}\exp(\frac{2}{k_1} s)\bigr)+1\bigr],
\end{align*}
where $t,s$ denote the coordinate along the $\mathbb R$-factor, $k_1$ is a non-zero constant
\[k_1\neq 0,\] and $W$ is the Lambert function characterised by
$$
W \exp(W) = Z.
$$
The corresponding volume growth satisfies
$$
\Vol = O(1), \quad t\rightarrow+\infty; \qquad
\Vol = O(t^3), \quad t\rightarrow-\infty.
$$

We examine whether Ball's examples belong to our definition of the moduli space $\mathcal M (\Lambda,F)$.
\begin{lem}For $M_5$, we have
$
F=2|k_1|^{-1}.
$
\end{lem}
\begin{proof}
From the construction, we get
\begin{align*}
|\nabla f|=2|k_1| u^{-2}\leq 2|k_1|^{-1},
\end{align*}
which gives us $F$.
\end{proof}
Both examples are steady, so $\lambda=0$. Hence
\eqref{soliton scalar} gives
\begin{equation}
\label{eq:Ball-steady-trace}
S=-\frac{3}{2}\tri f.
\end{equation}
\begin{lem}
For $M_5$, $S\geq 
-\frac{12}{k_1^2}$ and $\Lambda=\frac{12}{k_1^2}$.
\end{lem}
\begin{proof}
The $G_2$-structure is given by \cite{MR4555993}*{Equation (6.6)}
\[
\varphi
=dt\wedge
\left(
u(t)^2\operatorname{vol}_{F}
+w(t)^2\Omega_X
\right)
+
u(t)w(t)^2\Gamma,
\]
where $\operatorname{vol}_{F}
=
\zeta_2\wedge\zeta_3,$
$\Omega_X
=
\eta_1\wedge\eta_2+\eta_3\wedge\eta_4$
\begin{align*}
&\Gamma
=
\zeta_3\wedge
\left(
\eta_3\wedge\eta_1+\eta_2\wedge\eta_4
\right) 
+
\zeta_2\wedge
\left(
\eta_4\wedge\eta_1+\eta_3\wedge\eta_2
\right).
\end{align*}
We define the local orthonormal coframe
\begin{align*}
&e^1 = dt,\quad
e^2 = -u(t)\zeta_3,\quad
e^3 = u(t)\zeta_2,\quad
e^4 = w(t)\eta_1,\\
&
e^5 = w(t)\eta_2,\quad
e^6 = w(t)\eta_3,\quad
e^7 = w(t)\eta_4.
\end{align*}
Then the $G_2$-structure becomes the standard form and the metric induced by the $G_2$-structure is
\begin{align*}
g_\varphi
=
\sum_{i=1}^{7}e^i\otimes e^i
=
dt^2
+
u(t)^2\bigl(\zeta_2^2+\zeta_3^2\bigr)
+
w(t)^2\bigl(
\eta_1^2+\eta_2^2+\eta_3^2+\eta_4^2
\bigr).
\end{align*}
The radial volume density is $J=u^2w^4$ and the radial Laplacian formula becomes
\begin{align}\label{eq:Ball-radial-Laplacian}
\tri f=\frac{1}{J}(J f')'=f''+(\log J)'f'.
\end{align}

Then we derive the formula of $\tri f$.
The formulas in
\cite{MR4555993}*{Equations (6.13), and (6.16)}
give
\begin{equation}
\label{eq:Ball-first-family}
-\frac{u}{2}
-\frac{k_1}{2}\log|u-k_1|
=
t+C,
\qquad
w^2=\frac{u-k_1}{k_1k_2u},
\end{equation}
and the solution satisfies
$
|u(t)|>|k_1|>0.
$
Differentiating the $u$ identity in
\eqref{eq:Ball-first-family} and computing the formula of $J$ and $(\log J)'$ gives
\[
u'=-\frac{2(u-k_1)}{u},\quad J=u^2w^4=
\frac{(u-k_1)^2}{k_1^2k_2^2},\quad (\log J)'
=
\frac{2u'}{u-k_1}
=
-\frac{4}{u}.
\]
Since $f'=-2k_1u^{-2}$, we also have
\[
f''
=
\frac{4k_1u'}{u^3}
=
-\frac{8k_1(u-k_1)}{u^4}.
\]
Substituting these identities into
\eqref{eq:Ball-radial-Laplacian}, we obtain
\begin{align*}
\tri f
&=
-\frac{8k_1(u-k_1)}{u^4}
+
\left(-\frac{4}{u}\right)
\left(-\frac{2k_1}{u^2}\right) 
=
\frac{8k_1^2}{u^4}.
\end{align*}
It follows from \eqref{eq:Ball-steady-trace} that
\[S=
-\frac{12k_1^2}{u^4}
\geq
-\frac{12}{k_1^2}.
\] That completes the proof of this lemma.
\end{proof}

In conclusion, we have
\begin{lem}
The steady soliton $M_5$ belongs to the class $\mathcal M(\Lambda,F)$ with
\begin{align*}
\Lambda=12 k_1^{-2},\quad F=2|k_1|^{-1}.
\end{align*}
\end{lem}

$M_6=\mathbb R \times N$, where $N$ is the total space of certain $T^2$-bundle over a suitable hyperk\"ahler $4$-manifold \cite{MR4555993}*{Theorem 6.4}. The potential function is
\begin{align*}
f(t)=-\frac{2k_1}{k_2^2}\int_1^t u(s)\,[k_1-u(s)]\, ds,
\quad s\in\mathbb R,\quad 0<u(s)<k_1, \quad k_2\neq 0,
\end{align*}
where
$$
u(s)\rightarrow k_1-\frac{k_2^2}{k_1^2}\, s^{-1},\ \text{as }s\rightarrow+\infty;
\qquad
u(s)\rightarrow -\frac{k_2^2}{k_1^2}\, s^{-1},\ \text{as }s\rightarrow-\infty.
$$
In this case, the volume growth is
$$
\Vol = O(|t|^3), \quad t\rightarrow+\infty; \qquad
\Vol = O(|t|), \quad t\rightarrow-\infty.
$$
\begin{lem}For $M_6$, we get
$
F=\frac{k_1^3}{2k_2^2}.
$
\end{lem}
\begin{proof}
We compute that
\begin{align*}
f'=-\frac{2k_1}{k_2^2} u(t) [k_1-u(t)]
\end{align*}
Then by the elementary inequality
\[
-\frac{k_2^2}{2k_1}f'= u(t) [k_1-u(t)]\leq \frac{k_1^2}{4}.
\]
Since $u(t) [k_1-u(t)]>0$, we obtain the bound
\begin{align}\label{M5 gradient f}
|\nabla f|\leq \frac{2k_1}{k_2^2} \frac{k_1^2}{4}=\frac{k_1^3}{2k_2^2}
\end{align} and the choice of $F$.
\end{proof}

\begin{lem}
For $M_6$, we choose
$
\Lambda= \frac{3k_1^6}{16k_2^4}.
$
\end{lem}
\begin{proof}
The \(G_2\)-structure is given in \cite{MR4555993}*{Equations (6.22)}
\[
\varphi
=
dt\wedge
\left(
u(t)^2\,\pi_1\wedge\pi_2
+
w(t)^2\,\Omega_I
\right)
-
u(t)w(t)^2\,\pi_2\wedge\Omega_J
+
u(t)w(t)^2\,\pi_1\wedge\Omega_K,
\]
where \(\{\eta_1,\eta_2,\eta_3,\eta_4\}\). Let \(N\) be the local oriented orthonormal
coframe on a hyperk\"ahler four-manifold such that
\[
\Omega_I=\eta_1\wedge\eta_2+\eta_3\wedge\eta_4,\quad
\Omega_J=\eta_1\wedge\eta_3+\eta_4\wedge\eta_2,\quad
\Omega_K=\eta_4\wedge\eta_1+\eta_3\wedge\eta_2.
\]
We introduce the local coframe
\begin{align*}
&e^1 = dt,\quad
e^2 = -u(t)\pi_2,\quad
e^3 = u(t)\pi_1,\quad
e^4 = w(t)\eta_1,\\
&e^5 = w(t)\eta_2,\quad
e^6 = w(t)\eta_3,\quad
e^7 = w(t)\eta_4.
\end{align*}
Then, the metric has the same form as given in $M_5$
\[
g_\varphi
=
dt^2
+
u(t)^2
\left(
\pi_1^2+\pi_2^2
\right)
+
w(t)^2
\left(
\eta_1^2+\eta_2^2+\eta_3^2+\eta_4^2
\right).
\] We also have the same formulas for the radial volume density and the radial Laplacian.

\cite{MR4555993}*{Equations (6.27)(6.30)}
give
\begin{equation}
\label{eq:Ball-second-family}
w=\frac{k_2}{\sqrt{u(k_1-u)}},
\qquad
f'=-\frac{2k_1}{w^2}
  =-\frac{2k_1}{k_2^2}u(k_1-u),
\end{equation}
and \cite{MR4555993}*{Equations (6.26abc)} give the system
\begin{align}
\label{eq:Ball-second-ODE}
u'=
\frac{u(w^2f'+2u)}{2w^2},
\quad
w'=
-\frac{w^2f'+4u}{4w},
\quad
f''=
\frac{f'(w^2f'+4u)}{2w^2}.
\end{align}
Since $J=u^2w^4$, we compute with the help of
\eqref{eq:Ball-second-ODE} to get
\begin{align*}
(\log J)'=
2\frac{u'}{u}+
4\frac{w'}{w} =
\frac{w^2f'+2u}{w^2}-
\frac{w^2f'+4u}{w^2} 
=-\frac{2u}{w^2}.
\end{align*}
Consequently, by \eqref{eq:Ball-radial-Laplacian}, we have
\begin{align*}
\tri f
&=
\frac{f'(w^2f'+4u)}{2w^2}
-
\frac{2u}{w^2}f' 
=
\frac{1}{2}f'^2.
\end{align*}
The steady trace identity \eqref{eq:Ball-steady-trace} and the bound of $\nabla f$ from \eqref{M5 gradient f} therefore yield
\[
S
=
-\frac{3}{4}f'^2
\geq
-\frac{3k_1^6}{16k_2^4}.
\]
Hence we get the choice of $\Lambda$. \end{proof}

Finally, in both examples, for a chosen basepoint $p$, replace the potential $f$ by
$
f_p=f-f(p).
$
Then $f_p(p)=0$, while neither the gradient nor the soliton
equation changes. 

Consequently, we have
\begin{lem}
The steady soliton $M_6$ belongs to the class $\mathcal M(\Lambda,F)$ with
\begin{align*}
\Lambda= \frac{3k_1^6}{16k_2^4},
\qquad
F=\frac{k_1^3}{2k_2^2}.
\end{align*}
\end{lem}

\begin{rem}
These examples demonstrate that the shrinkers $M_1$ and $M_2$ can have scalar curvature made arbitrarily small by adjusting the parameter $b$
\begin{align*}
&|S_{M_1}|=\frac{27b^2}{4}\rightarrow0, \text{ as }b\rightarrow0 ,\quad
|S_{M_2}|=\frac{3}{4b^2}\rightarrow0, \text{ as }b\rightarrow\infty .
\end{align*} 
The two steady solitons $M_3$ and $M_4$ exhibit uniformly bounded scalar curvature
\[-3<S_{M_3}\leq-\frac{9}{4},\quad S_{M_4}=-12.\]
But for the steady solitons $M_5$ and $M_6$, the scalar curvature bounds depend on the parameters
\[S_{M_5}\geq 
-\frac{12}{k_1^2},\quad S_{M_6}
\geq
-\frac{3k_1^6}{16k_2^4}.\]

We could compare these growth behaviours to those for Ricci solitons, cf. \cite{MR3409114}*{Theorem 27.4, Corollary 27.10}.
\end{rem}

\section{Compactness of \texorpdfstring{$G_2$}{G_2}-solitons}
\label{Weak compactness of G2-solitons}
We use the following notation:
\begin{itemize}
\item $(M,g)$ denotes a complete Riemannian manifold of dimension $n$,
\item $\mathbb{R}^n$ is the $n$-dimensional Euclidean space,
\item $p$ is a chosen basepoint of $M$,
\item a domain $\Om$ is a connected open subset of $M$ with smooth boundary,
\item a sequence of domains $\Om_i$ is called an exhaustion of a bounded domain $\Om$ if  $\Om_1\Subset\Om_2\Subset\cdots\Subset\Om$ and $\Om=\cup_{i=1}^\infty\Om_i$,
\item $a$ is a smooth function defined on the closure of $\Om$ with each $\Om_i$ having smooth boundary,
\item $B_r(x)$ denotes the geodesic ball of radius $r$ centred at $x$ with respect to $g$,
\item we denote the volume of a domain $\Omega$ with respect to the Riemannian measure $dv$ by $|\Omega| := \int_\Omega dv$.
\item similarly, the volume of a domain $\Omega$ with respect to the weighted measure $dv_f$ is written as $|\Omega|_f := \int_\Omega dv_f$.
\end{itemize}

\subsection{Pointed Gromov-Hausdorff convergence}\label{pointed Gromov-Hausdorff convergence}
We recall the Gromov-Hausdorff convergence theory.
\begin{defn}\label{metric space GH}
We say that a sequence of pointed metric spaces $(M_i, d_i, p_i)$ converges
to $(M_\infty,d_\infty, p_\infty)$ in the pointed Gromov-Hausdorff sense,
written
\[
(M_i,d_i, p_i) \xrightarrow{pGH} (M_\infty,d_\infty, p_\infty),
\]
if for every $R>\epsilon>0$, there exists $N\in\mathbb{N}$
such that for all $i\geq N$, there exists a map
$
\Phi_i \colon B_R(p_i) \to M_\infty
$ (depending on $R,\eps$)
satisfying:
\begin{itemize}
    \item $d_\infty(\Phi_i(p_i), p_\infty) < \epsilon$;
    \item $\forall x,y \in B_R(p_i)$,
    $
    \bigl| d_\infty(\Phi_i(x),\Phi_i(y)) - d_i(x,y) \bigr| < \epsilon;
    $
    \item the image $\Phi_i(B_R(p_i))$ is $\epsilon$-dense in $B_{R-\epsilon}(p_\infty)$,
    i.e., for every $u \in B_{R-\epsilon}(p_\infty)$ there exists
    $x \in B_R(p_i)$ such that
$
    d_\infty(\Phi_i(x), u) < \epsilon.
$
\end{itemize}

We denote
$$(M_i,d_i,f_i,p_i)\xrightarrow{pGH} (M_\infty,d_\infty,f_\infty,p_\infty)$$
if $(M_i,d_i,p_i) \xrightarrow{pGH} (M_\infty,d_\infty,p_\infty)$, and in addition, the following convergence property for the potential functions holds:
\[
\sup_{B_R(p_i)}|f_i-f_\infty \circ \Phi_i|\rightarrow0,\quad \text{ for every }R>0.
\]
\end{defn}
The limit function $f_\infty$ is locally Lipschitz.
It follows from $f_i(p_i)=0$ and the local equi-Lipschitz bounds \ref{f condition}.
\begin{thm}\label{GH convergence main body}
Let $\left( M_i, \vphi_i, g_i, f_i,p_i\right)\in \mathcal M(\Lambda,F)$ be a sequence.
Then, after possibly extracting a subsequence, the corresponding complete pointed length spaces converge in the pointed Gromov-Hausdorff sense to a complete pointed length space $(M_\infty,d_\infty,f_\infty,p_\infty)$, that is,
$$(M_i,d_i,f_i,p_i)\xrightarrow{pGH} (M_\infty,d_\infty,f_\infty,p_\infty).$$
\end{thm}
The proof will be presented at the end of this section.
\subsubsection{Comparison geometry of $G_2$-solitons}
\label{Comparison geometry}
In this subsection, we gather the necessary results for use later on.
\begin{defn}
For any ball $B_{r}(q)\subset B_{R}(p)$,
we set
\begin{align*}
K_R:=-\frac{L_\lambda+\Lambda(R)}{3(n-1)}\leq 0.
\end{align*}
With this notation, from \lemref{Ricci bound} we obtain
\begin{align}\label{Ricci bound lower}
Ric_f\geq \frac{-\lambda+S}{3}g\geq (n-1)\cdot K_R \cdot g \quad \text{ on } B_{r}(q).
\end{align}
\end{defn}
\begin{rem}
When the scalar curvature has a uniform lower bound, the quantity $K_R$ becomes independent of $R$. Under this hypothesis, we note that $\mathcal M(\Lambda,F)$ fulfills the assumptions required of the space $\mathcal N_m(F,K)$ as defined in \cite{MR4234100}*{Definition 1.3 }.
\end{rem}

We recall the $K_R$-sine function $\mathrm{sn}_{K_R}(t)$, defined by
\[
\mathrm{sn}_{K_R}(t) =
\begin{cases}
K_R^{-\frac{1}{2}} \sin \bigl(K_R^{\frac{1}{2}}\, t\bigr),\quad t\leq \dfrac{\pi}{2}K_R^{-\frac{1}{2}} & \text{if } K_R > 0,\\[0.3em]
t & \text{if } K_R = 0,\\[0.3em]
(-K_R)^{-\frac{1}{2}} \sinh\!\bigl[(-K_R)^{\frac{1}{2}}\, t\bigr] & \text{if } K_R < 0.
\end{cases}
\]
This function is non-decreasing.

Let $m_{K_R}$ denote the mean curvature of a geodesic sphere in the $7$-dimensional model space $M_{K_R}$ of constant sectional curvature $K_R$, with the normal vector chosen to point outward:
\[
m_{K_R}(t) =
\begin{cases}
(n-1)K_R^{\frac{1}{2}}\,\cot\!\bigl(K_R^{\frac{1}{2}}\,t\bigr), \quad t\leq \dfrac{\pi}{2}K_R^{-\frac{1}{2}} & \text{if } K_R > 0,\\[0.3em]
(n-1){t}^{-1} & \text{if } K_R = 0,\\[0.3em]
(n-1)(-K_R)^{\frac{1}{2}}\,\coth\!\bigl[(-K_R)^{\frac{1}{2}}\,t\bigr] & \text{if } K_R < 0.
\end{cases}
\]

We denote by $\mathcal A_{K_R}$ the radial area distortion in the model space $M_{K_R}$.
We also write $\Vol_{K_R}(r)$ for the volume of the ball of radius $r$ in the model space $M_{K_R}$:
\[
\Vol_{K_R}(r)=\omega_6 \int_0^r \mathrm{sn}_{K_R}^6(t)\, dt.
\]

The mean curvature comparison was established in \cite{MR2577473}*{Theorem 1.1} under a lower bound condition on the Bakry-\'{E}mery Ricci curvature. We now apply that result in our framework.
\begin{lem}[Laplacian comparison]\label{Laplacian Comparison}
Let $B_{r}(q)\subset B_{R}(p)$ be any ball. Suppose that $Ric_f$ satisfies the lower bound given in \eqref{Ricci bound lower} and that $|\nabla f(x)|\leq F(d(x,p))$. Then, throughout $B_{r}(q)$,
\begin{align*}
\tri_f d=\tri d-\p_r f\leq m_{K_R}+F(R)
\end{align*}
in the barrier sense globally.
\end{lem}

In geodesic polar coordinates $(r,\theta)$, the weighted volume element can be expressed in terms of the area distortion function $\mathcal A(r,\theta)$ and the standard volume element on the unit sphere $dS^{n-1}$ as
\[ dv_f = \mathcal A_f(r,\theta)\,dr \wedge dS^{n-1},\quad \mathcal A_f(r,\theta)=\exp(-f)\mathcal  A(r,\theta). \]

The volume comparison result is an application of the Laplacian comparison, see \cite{MR2577473}*{Theorem 1.2}.
\begin{lem}[Volume comparison]\label{volume comparison theorem}
Under the same hypotheses as in \lemref{Laplacian Comparison}, for any $0<s<r$ within the domain $\mathcal D$ we have
\begin{align*}
    &\frac{\mathcal A_f(r)}{\mathcal A_f(s)}
    \leq \exp(F(R)\cdot  (r-s))\frac{\mathcal A_{K_R}(r)}{\mathcal A_{K_R}(s)},\\
    &\frac{|B_r(q)|_f}{|B_s(q)|_f}\leq \frac{\int_0^r \exp(F(R)   t)\mathcal A_{K_R}(t) dt}{\int_0^s \exp(F(R)  t) \mathcal A_{K_R}(t) dt}
    \leq \exp(F(R) r)\,\frac{\Vol_{K_R}(r)}{\Vol_{K_R}(s)}.
\end{align*}
\end{lem}
\begin{proof}
Starting from the identities
\[
\partial_r \log \mathcal A_f = \tri_f d,\quad\partial_r \log \mathcal A_{K_R} = m_{K_R},
\]
the Laplacian comparison yields
\begin{align*}
\partial_r \log \mathcal A_f(r)\leq \partial_r \log \bigl[\exp(F(R) r )\mathcal A_{K_R}(r)\bigr],
\end{align*}
so that the ratio $\frac{\mathcal A_f}{\exp(F(R) r )\mathcal A_{K_R}}$ is not increasing in $r$. The desired volume comparison then follows by integrating this inequality, exactly as in the argument of \cite{MR1452876}*{Lemma 3.2}.
\end{proof}

\begin{lem}[Volume doubling]\label{Volume doubling}
Under the same assumptions as in \lemref{Laplacian Comparison} suppose that $B_{2r}(q)\subset B_R(p)$ and $r,R\in  \mathcal D$, we obtain
\begin{align*}
   |B_{2r}(q)|_f \leq C\, |B_{r}(q)|_f
\end{align*}
where $C=2^7  \exp(2F(R) R)$ when $K_R\geq 0$ and $C=2^7 \cosh\bigl(\sqrt{-K_R}\, R\bigr)^6 \exp(2F(R) R)$ when $K_R<0$.

Consequently, every $(M,\vphi,g,f,p)\in\mathcal M(\Lambda,F)$ satisfies the volume doubling property, i.e. locally uniformly doubling on every fixed $B_R(p)$.
\end{lem}

\begin{proof}
This is an immediate consequence of the volume comparison theorem \lemref{volume comparison theorem}, by substituting $r$ and $2r$ respectively for the variables $s$ and $r$ appearing there. We then estimate the constant via
\begin{align*}
    \frac{|B_{2r}(q)|_f}{|B_r(q)|_f}\leq \exp(2F(R) r)\,\frac{\Vol_{K_R}(2r)}{\Vol_{K_R}(r)}.
\end{align*}
When $K_R=0$, the volume ratio $\frac{\Vol_{0}(2r)}{\Vol_{0}(r)}$ in the Euclidean space is $2^7$.
In the case $K_R>0$ (i.e., on a spherical space), we use the classical Laplacian comparison with Euclidean space to obtain
\[
\frac{\Vol_{K_R}(2r)}{\Vol_{K_R}(r)}
\leq
2^7.
\]

For the case $K_R<0$ (hyperbolic space), we compute
\[
\frac{\Vol_{K_R}(2r)}{\Vol_{K_R}(r)}
=
\frac{\int_0^{2r} \sinh(a t)^6dt}
{\int_0^{r} \sinh(a t)^6dt},\quad a=(-K_R)^{\frac{1}{2}}.
\]
Change variables $t = 2s$ in the numerator. Then $dt = 2\,ds$, and as $t$ varies from $0$ to $2r$, $s$ varies from $0$ to $r$, so
\[
\int_0^{2r} \sinh(at)^6 \, dt = 2\int_0^{r} \sinh(2as)^6 ds.
\]
Using the double-angle formula $\sinh(2x) = 2\sinh x \cosh x$, we obtain
\[
\sinh(2as)^6 = \bigl(2\sinh(as)\cosh(as)\bigr)^6 = 2^6 \sinh(as)^6 \cosh(as)^6.
\]
Since $\cosh(as)$ is increasing on $[0,\infty)$, for $0 \le s \le r$ we have $\cosh(as) \le \cosh(ar)$. Hence
\[
\sinh(2as)^6 \leq 2^6 \sinh(as)^6 \cosh(ar)^6.
\]
Substituting this estimate back, we find
\[
\int_0^{2r} \sinh(at)^6 \, dt
\le 2^7 \cosh(ar)^6 \int_0^{r} \sinh(as)^6 \, ds,
\]
which yields
\[
\frac{\Vol_{K_R}(2r)}{\Vol_{K_R}(r)} \leq  2^7 \cosh(ar)^6=2^7 \cosh\bigl(\sqrt{-K_R}\, r\bigr)^6 .
\]
Finally, we obtain the stated constant $C$ by replacing $r$ with its upper bound $R$.
\end{proof}
\subsubsection{Proof of \thmref{GH convergence main body}}
Each Riemannian manifold in $\mathcal M(\Lambda,F)$ determines a complete pointed length space
$(M_i,d_i,p_i)$.
By \lemref{Ricci bound}, the Bakry-\'{E}mery Ricci curvature is uniformly bounded on every fixed basepoint ball. As a result, the volume doubling condition is satisfied. Then the ball-packing argument implies the pointed Gromov-Hausdorff convergence to a complete pointed length space $(M_\infty,d_\infty,p_\infty)$ described in \thmref{thm:weak convergence}; see also \cite{MR2307192}. 

The diagonal pGH extraction and an Arzela-Ascoli extraction are applied for $f_i$ on the approximation maps.
\subsubsection{Non-inflating}
Another use of the volume comparison result is to obtain an upper bound for the volume of an arbitrary geodesic ball.
\begin{lem}\label{BEVCT}
Let \(B_r(q)\subset B_R(p)\).
Under the hypotheses of \lemref{Laplacian Comparison}, we have
\[
    |B_r(q)|_f
    \le
    \exp(F(R)r-f(q))\Vol_{K_R}(r).
\]
Consequently,
\[
    |B_r(q)|
    \le
    \exp(f(q)+F(R)r)|B_r(q)|_f
    \le
    \exp(2F(R)r)\Vol_{K_R}(r).
\]
In particular, \((M,g)\) is non-inflating.
\end{lem}
\begin{proof}
By our assumptions, we can apply the volume comparison in \lemref{volume comparison theorem} to obtain
\begin{align*}
\frac{|B_r(q)|_f}{\Vol_{K_R}(r)}
\leq \exp(F(R)r)\lim_{s\rightarrow 0}\frac{|B_s(q)|_f}{\Vol_{K_R}(s)}.
\end{align*}
Moreover, since
\[
\lim_{s\rightarrow 0}\frac{\Vol_0(s)}{\Vol_{K_R}(s)}=1,
\]
it follows that
\begin{align*}
\frac{|B_r(q)|_f}{\Vol_{K_R}(r)}
\leq \exp(F(R)r-f(q)).
\end{align*}

Now let \(x\in B_r(q)\), and let \(\gamma\) be a minimizing geodesic
from \(q\) to \(x\). Since \(\gamma\subset B_r(q)\subset B_R(p)\),
\[
\begin{aligned}
    f(x)-f(q)
    &=
    \int_0^{d(q,x)}
    \langle\nabla f,\dot\gamma\rangle\,dt  
    \le
    \int_0^{d(q,x)}|\nabla f|\,dt\\
&    \le
    F(R)d(q,x)
    \le
    F(R)r.
\end{aligned}
\]
Thus
\begin{align}\label{estimate of fx}
    \sup_{B_r(q)}f
    \le
    f(q)+F(R)r.
\end{align}

Recalling the definition of the weighted volume and substituting the above estimate, we obtain
\begin{equation}\label{weighted volume volume comparison}
\begin{aligned}
|B_r(q)|&=
    \int_{B_r(q)}\exp(f)\,dv_f
\leq \exp(\sup_{B_r(q)} f)\, |B_r(q)|_f\\
&\leq
    \exp(f(q)+F(R)r)|B_r(q)|_f 
\leq \exp(2F(R)r) \Vol_{K_R}(r).
\end{aligned}
\end{equation}
This yields the required upper bound.
\end{proof}
\subsection{Pointed measured Gromov-Hausdorff convergence}
\label{Pointed measured Gromov-Hausdorff convergence}
\begin{defn}\label{weighted volume}
We define the $f$-weighted volume element by
$$dv_f := \exp(-f)dv,$$
where $dv$ is the volume form determined by $g$.
The associated weighted volume measure $m_f$ and the normalised measure are given by
\begin{align*}
m_f(\Om)=|\Om|_f=\int_{\Om} dv_f,\quad \mu_f = \frac{m_f}{m_f(B_1(p))}.
\end{align*}

Each manifold in $\mathcal M(\Lambda,F)$ naturally defines a metric measure space
$$(M, d, \mu_f),$$
where $d$ denotes the distance structure induced by the Riemannian metric $g$.
\end{defn}
\begin{defn}\label{pmGH convergence}
Using the notion introduced in Definition~\ref{metric space GH}, 
we say a sequence of pointed metric measure spaces
$(M_i,d_i,\mu_i,p_i)$
converges in the pointed measured Gromov-Hausdorff sense
to a limit pointed metric measure space
$(M_\infty,d_\infty,\mu_\infty,p_\infty)$, denoted by
\[
(M_i,d_i,\mu_i,p_i) \xrightarrow{pmGH} (M_\infty,d_\infty,\mu_\infty,p_\infty),
\]
if $(M_i,d_i, p_i) \xrightarrow{pGH} (M_\infty,d_\infty, p_\infty)$ and the pushforward of the restricted measure converges weakly:
    for every bounded continuous function $\varphi$ with support contained in
    $B_{R-\epsilon}(p_\infty)$,
    \[
    \lim_{i\to\infty}
    \int_{B_R(p_i)} \varphi(\Phi_i(x)) \, d\mu_i(x)
    =
    \int_{B_{R-\epsilon}(p_\infty)} \varphi(y) \, d\mu_\infty(y).
    \]

We write $\left( M_i, d_i, f_i,\mu_{f_i},p_i\right)\xrightarrow{pmGH}\left( M_{\infty}, d_{\infty}, f_{\infty},\mu_\infty,p_\infty\right)$ if the potential functions converge in the sense of Definition~\ref{metric space GH}.
\end{defn}
\begin{defn}
Let $$(M,d,f,p)$$ be a pointed Gromov-Hausdorff limit space.
For any integer $k = 1,2,\ldots,n,$ we say that a point $x \in M$ is $k$-regular if every tangent cone at $x$ is isometric to $\mathbb{R}^k$.

We denote the set of all such points by $\mathcal{R}_k$, and define the regular set to be the union of these:
$$\mathcal{R} = \bigcup_{1 \le k \le n} \mathcal{R}_k.$$
The complement of the regular set is called the singular set
$$\mathcal{S} = M \setminus \mathcal{R}.$$
\end{defn}

\begin{thm}[Weak regular-singular decomposition]\label{weak regular-singular decomposition}
Let $\left( M_i, \vphi_i, g_i, f_i,p_i\right)\in \mathcal M(\Lambda,F)$ be a sequence. We let \(p_i\) be a global minimum of \(f_i\).
Then, after possibly extracting a subsequence, the corresponding metric measure spaces
$$\left( M_i, d_i, f_i,\mu_{f_i},p_i\right)\xrightarrow{pmGH}\left( M, d_{\infty}, f_{\infty},\mu_\infty,p_\infty\right).$$
Moreover, the limit space admits a weak regular-singular decomposition in the following sense
\begin{itemize}
\item $M$ admits a Radon measure $\mu_\infty$, see \cite{MR4234100}*{Proposition 2.12},
\item there exists a unique $1\leq k\leq n$ such that the $k$-stratum of the regular part has positive measure $\mu_\infty(\mathcal R_k)>0$, see \cite{MR4234100}*{Theorem 1.5},
\item $\mathcal{R}$ is weakly convex, that is, for any two points $x,y$ in $\mathcal{R}$, the distance $d(x,y)$ coincides with the infimum of the lengths of all curves joining $x$ to $y$ and lying entirely in $\mathcal{R}$; see \cite{MR4234100}*{Section 4},
\item $\mathcal{R}$ is $\mu_\infty$-almost everywhere convex, meaning that for $\mu_\infty\times\mu_\infty$ almost every pair of points $(x,y)$ in $\mathcal{R}\times\mathcal{R}$ there exists a minimising geodesic connecting $x$ and $y$, that is fully contained in $\mathcal{R}$; see \cite{MR4234100}*{Section 4},
\item $\mathcal{S}$ is $\mu_\infty$-negligible, $\mu_\infty(\mathcal S)=0$, see \cite{MR4234100}*{Proposition 2.17},
\item the tangent cones obtained from the same scaling sequence are H\"older continuous as the point varies along the interior of a limit minimising geodesic with respect to the Gromov-Hausdorff topology, see \cite{MR4234100}*{Proposition 4.8}.
\end{itemize}
\end{thm}
\begin{proof}
By \lemref{Ricci bound} and the definition of $\mathcal M(\Lambda,F)$, the assumptions in \cite{MR4234100}*{Definition 1.3} are completely fulfilled. Hence, their results are applicable in our framework.
For a comprehensive account of the convergence of the measures and of the regular-singular decomposition of the limit space, we direct the reader to the references listed above.
\end{proof}

We record the following local Sobolev inequality, which will be used in Section \ref{Smooth convergence on uniform curvature bounds}.
\begin{lem}[Local Sobolev inequality]\label{lem:Sob}
        Suppose $\left( M,g,f,p\right)\in \mathcal M(\Lambda,F)$ and suppose $B_{r}(q)\subset B_{R}(p)$.
        Then there exists a constant $C_s$ such that the following local Sobolev inequality holds:
        \begin{equation*}
            \left( \int_{B_r(q)}h^{\frac{2n}{n-2}} dv_f\right)^{\frac{n-2}{n}} \leq \frac{C_s}{|B_r(q)|_f^{\frac{2}{n}}}\left(\int_{B_r(q)}\big[r^2 \left| \nabla h\right|^2+h^2 \big]dv_f \right)
        \end{equation*}
        for every geodesic ball $B_r(q)\subset B_R(p)$ and every $h\in C_0^1(B_r(q))$. The constant $C_s$ depends on $n,\Lambda, F,L_\lambda,R$.
    \end{lem}
\begin{rem}
For proofs and related discussions, see \cite{MR4234100}*{Proposition 2.9}, \cite{MR3498912}*{Lemma 2.6}, and \cite{MR2846384}*{Lemma 3.2}.
\end{rem}

\subsection{Local Perelman entropy and non-collapsing}
\label{Local Perelman entropy}
We now recall the localised version of Perelman's entropy and its main properties, following \cite{MR3855081}*{Section 2}.
\begin{defn}\label{Perelman}
Consider the function space
    \begin{align*}
        \mathcal H(\Om,g):=\Big\{u \,\Big\vert\, u\in W^{1,2}_0(\Om),\; u\geq 0,\; \int_\Om u^2\,dv=1\Big\}.
    \end{align*}
    The localised Perelman $W$-functional on a domain $\Omega$ is defined by
\begin{align*}
W^{b}(\Om,g,u,\tau):= \int_\Om [\,\tau(b+|\nabla h|^2)+h-n\,]\;u^2\,dv,\quad \text{where } u^2=(4\pi\tau)^{-\frac{n}{2}}\exp(-h),
\end{align*}
where $\tau$ is a fixed positive constant and $b$ is a parameter.

The localised $\mu$-functional is defined by
    \begin{align*}
        \mu^{b}(\Om,g,\tau):=\inf_{u\in\mathcal H(\Om,g)} W^{b}(\Om,g,u,\tau).
    \end{align*}
    The localised $\nu^b$-functional is obtained by taking the infimum of $\mu^b$ with respect to the time parameter $\tau$ over a finite or infinite interval:
    \begin{align*}
        &\nu^{b}(\Om,g,\tau):=\inf_{s\in (0,\tau]} \mu^{b}(\Om,g,s),\quad \nu^{b}(\Om,g):=\nu^{b}(\Om,g,\infty).
    \end{align*}
    We also write $\nu=\nu^S.$
\end{defn}
\begin{rem}
If $\Omega$ is a compact manifold $M$, then $\nu^{S}(M,g)$ coincides with Perelman's $\nu^{S}$-functional \cite{arXiv:math/0211159}. This quantity is monotone along the Ricci flow and is a key tool in Perelman's proof of the Poincar\'e conjecture \cite{arXiv:math/0303109,arXiv:math/0307245}.
\end{rem}
\begin{lem}\label{nu properties}
    The local entropy $\nu^b$ has the following properties.
    \begin{enumerate}[label=(\Alph*)]
        \item \label{Non-positivity} (Non-positivity) For every choice of domain and parameters,
        \[
            \nu^b(\Omega,g,\tau)\leq 0,
        \]
see \cite{MR3855081}*{Proposition 2.2}.
        \item (Monotonicity)\label{Monotonicity}  For all $\Omega_1\subset\Omega_2$:
       $$\mu^b(\Omega_1,g,\tau)\geq \mu^b(\Omega_2,g,\tau),\quad \nu^b(\Omega_1,g,\tau)\geq \nu^b(\Omega_2,g,\tau),$$
see \cite{MR3855081}*{Proposition 2.1}.
        \item (Scale invariance)\label{Scale invariance} Let $Q$ be a positive constant. Then
        \begin{align*}
            &\mu^S(\Omega,Q g,Q\tau)=\mu^S(\Omega,g,\tau),\quad
            \mu^0(\Omega,Q g,Q\tau)=\mu^0(\Omega,g,\tau),\\
            &\nu^S(\Omega,Q g,Q\tau)=\nu^S(\Omega,g,\tau),\quad
            \nu^0(\Omega,Q g,Q\tau)=\nu^0(\Omega,g,\tau).
        \end{align*}

        \item (Rigidity)\label{Rigidity} 
        Assume $(M,g)$ is a complete
Riemannian manifold, $g$ has uniformly bounded Riemann curvature and that
        \[
            \nu^S(M,g,\tau_0)\geq 0 \quad\text{for some }\tau_0>0.
        \]
        Then $(M,g)$ is isometric to $(\mathbb R^n,g_E)$. It follows directly from \cite{MR3855081}*{Proposition 4.9}. See also \cite{arXiv:2010.09981}*{Corollary 3.3}.


        \item (Isoperimetric constant)\label{Isoperimetric constant}
       Denote by $\mathbf I$ the isoperimetric constant
        \[
        \mathbf I(\Om)=\inf_{D\Subset\Omega}\frac{|\p D|}{|D|^{\frac{n-1}{n}}},\quad \mathbf I(\mathbb R^n)=n\om_n^{\frac{1}{n}}\] and define $\lambda_{iso}$ to be
        \[
            \mathbf I(\Omega)=\lambda_{iso}\cdot \mathbf I(\mathbb R^n).
        \] Here, $\om_n$ is the unit ball volume in Euclidean space $\mathbb R^n$.
        Then $\nu^S$ admits the lower bound
        \begin{align}\label{nu functional, I and S}
            \nu^S(\Omega,g,\tau)\geq n\log\lambda_{iso}+\tau\cdot\inf_{\Omega} S(g),
        \end{align}
        where the right-hand side depends only on $\lambda_{iso}$ and the lower bound for the scalar curvature of $g$ on $\Omega$,
see \cite{MR3855081}*{Lemma 3.5}.

        \item (Continuity)\label{Continuity} Let $\{\Omega_i\}$ be an exhaustion of a bounded domain $\Omega$. Then
        \[
\mu^b(\Omega,g,\tau)=\lim_{i\to\infty}\mu^b(\Omega_i,g,\tau).
        \]
        In addition, suppose that the pointed manifolds $$(M_i,g_i,p_i)\xrightarrow{pC^{\infty}}(M_\infty,g_\infty,p_\infty),$$ and that the bounded domains $\Omega_i\subset M_i$ converge to a bounded domain $\Omega_\infty\subset M_\infty.$ Then
        \[
\mu^S(\Omega_\infty,g_\infty,\tau)=\lim_{i\to\infty}\mu^S(\Omega_i,g_i,\tau),
        \]
see \cite{MR3855081}*{Proposition 2.3, Corollary 2.5}.
    \end{enumerate}
\end{lem}

We now state these properties in the setting of $G_2$-structures.
\begin{lem}\label{G2 W comparison}
Let $g$ be the Riemannian metric induced by a closed $G_2$-structure $\vphi$.
Then
\begin{align*}
&W^0(\Om,g,u,\tau)+\tau \cdot \inf S \leq W^S(\Om,g,u,\tau)\leq W^0(\Om,g,u,\tau),\\
& \mu^0(\Om,g,\tau)+\tau \cdot \inf S \leq \mu^S(\Om,g,\tau)\leq \mu^0(\Om,g,\tau),\\
& \nu^0(\Om,g,\tau)+\tau \cdot \inf S \leq \nu^S(\Om,g,\tau)\leq \nu^0(\Om,g,\tau).
\end{align*}
\end{lem}
\begin{proof}
By substituting the non-positivity of the scalar curvature from \lemref{scalar curvature general} into Definition \ref{Perelman}, we obtain
\begin{align*}
W^S-W^0=\tau\int_\Om S u^2\,dv,
\end{align*}
which directly implies the statement of the lemma.
\end{proof}

\begin{lem}[Non-collapsing]\label{noncollapse}
Let $g$ be a $G_2$-metric that satisfies the entropy lower bound \ref{noncollapsing} in Definition \ref{3 assumptions}. Suppose $$B_r(q)\subset B_R(p).$$
Then $(M,g)$ is non-collapsed at scale $R$, i.e. there exists a constant $\kappa$ such that
$$|B_r(q)|\geq \kappa \cdot \om_7 r^7, \quad \forall 0<r<\min\{R,1\}.$$
The constant $\kappa$ depends only on $\underline{\nu}$.
\end{lem}
\begin{proof}
By \cite{MR3855081}*{Theorem 3.3} and Lemma \ref{G2 W comparison}, the lower bound of the volume ratio for a geodesic ball $B_r(q)$ can be controlled in terms of the lower bound $\underline{\nu}$:
\begin{align*}
|B_r(q)| &\geq \exp(\nu^0(B_r(q),g,r^2)-2^{14}) \omega_7 r^7 \\
&\geq \exp(\nu^S(B_r(q),g,r^2)-2^{14}) \omega_7 r^7.
\end{align*}
The desired statement then follows directly from the assumed bounds.
\end{proof}

\subsection{Pointed \texorpdfstring{$C^{1,\alpha}$}{C^{1,\alpha}} convergence}\label{C1alpha compactness}
\begin{defn}[Pointed $C^{k,\alpha}$ Convergence]\label{c1alpha convergence}
A sequence of complete Riemannian manifolds $(M_i,g_i,p_i)$ converges in the pointed $C^{k,\alpha}$ topology for some $k\geq 0$ and $0<\alpha<1$,
\[
(M_i,g_i,p_i) \xrightarrow{pC^{k,\alpha}} (M_\infty,g_\infty,p_\infty),
\]
if: $(M_i,d_i,p_i) \xrightarrow{pGH} (M_\infty,d_\infty,p_\infty)$ and
for each $R>0$, there exist smooth embeddings
\[
\sigma_{i,R}: B_R(p_\infty) \cap \mathcal R \to M_i
\]
such that for every compact $K \subset B_R(p_\infty) \cap \mathcal R$,
\[
\sigma_{i,R}^\ast(g_i) \xrightarrow{C^{k,\alpha}(g_\infty)} g_\infty \quad \text{on } K,
\]
where the norm is taken with respect to $g_\infty$.

Furthermore, when $k\geq 1$, we write 
\[
(M_i,g_i,f_i,p_i) \xrightarrow{pC^{k,\alpha}} (M_\infty,g_\infty,f_\infty,p_\infty)
\]
if additionally, 
for each $R>0$ and compact $K \subset B_R(p_\infty) \cap \mathcal R$,
\[
\sigma_{i,R}^\ast(f_i) \xrightarrow{C^{k-1,\alpha}(g_\infty)} f_\infty \quad \text{on } K.
\]

We say the convergence is smooth, if $C^{k,\alpha}$ is replaced with $C^\infty$. The embeddings are called compatible if
\[\sigma_{i,j+1}\vert_{\Omega_j}=\sigma_{i,j},\quad\text{for all sufficiently large i},\]
where $\Omega_j$ exhausts $\mathcal R$ and the smooth embeddings
$
\sigma_{i,j}: \Om_j \cap \mathcal R \to M_i
$.
\end{defn}
\begin{rem}
Restricted to a compact set $K \subset B_R(p_\infty) \cap \mathcal{R}$, the map $\sigma_{i,R}$ is actually an $\epsilon_i$-isometry from $(K, d_\infty)$ onto $(\sigma_{i,R}(K), d_i)$ for some $\epsilon_i \to 0$. In particular, convergence in the pointed $C^{1,\alpha}$ sense implies convergence in the pointed Gromov-Hausdorff sense, and the two notions are consistent.
\end{rem}
\begin{rem}
Convergence in the pointed $C^{1,\alpha}$ sense is a variant of what is often called Cheeger-Gromov convergence (which is typically $C^\infty$). We explicitly specify $C^{1,\alpha}$ or $C^\infty$ to indicate the regularity of the limit metric and the convergence of the pullback tensors.
\end{rem}
\begin{defn}[Regular-singular decomposition]\label{regular-singular decomposition}
A pointed metric space $(M,d,g,f,p)$ arising as a pointed $C^{1,\alpha}$ (respectively $C^\infty$) limit is said to admit a regular-singular decomposition $M=\mathcal R \cup \mathcal S$ if
\begin{itemize}
\item the regular part $\mathcal R$ is open and dense, is a smooth manifold, and its manifold topology agrees with the topology induced by $d$,
\item $\mathcal R$ carries a $C^{1,\alpha}$ (respectively $C^\infty$) Riemannian metric $g$, and on $\mathcal R$ the length structure induced by $g$ coincides with $d$,
\item $\mathcal R$ is strongly convex, meaning that whenever a minimal geodesic meets $\mathcal R$ in a nontrivial way, its whole interior remains contained in $\mathcal R$,
\item the singular set $\mathcal S$ is closed with respect to the topology induced by $d$,
\item $\mathcal S$ has Hausdorff codimension at least $4$,
\item $f$ is a locally Lipschitz function on $M$.
\end{itemize}
In this situation, we call $(M,d,g,f,p)$ a $C^{1,\alpha}$ (respectively, $C^\infty$) singular space.
\end{defn}
In this subsection, we prove pointed $C^{1,\alpha}$ convergence.
\begin{thm}\label{main convergence 3 main body}
Let $\left( M_i, \vphi_i, g_i, f_i,p_i\right)$ be a sequence in
$\mathcal M(\Lambda,F,\underline{\nu}).$
Then after passing to a subsequence, the associated sequence $$(M_i,g_i,f_i,p_i)\xrightarrow{pC^{1,\alpha}}(M_\infty,g_\infty,f_\infty,p_\infty).$$
Furthermore, the limiting space admits a regular-singular decomposition
in the sense of Definition \ref{regular-singular decomposition}.
\end{thm}
We start by reformulating the following definition from Li-Li-Wang \cite{MR4220743}*{Definition 10.1} so that it matches our previous terminology.
\begin{defn}\label{defn01} Denote by $\cM_n(F, K)$ the collection of pointed, complete, smooth $n$-dimensional Riemannian manifolds $(M, g, p)$ that satisfy the conditions below.
\begin{enumerate}[label=(\alph*)]
\item The Bakry-\'{E}mery Ricci tensor obeys
     \begin{align*}
          -K g\leq Ric_f\leq Kg,
        \end{align*}
for some constant $K$.\label{defn01 a}
  \item The function $f$ is of class $C^2$ on $M$ and fulfills
   \begin{align*}
           \max\{|f|(x), |S|(x)+|\Na f|^2(x)\}\leq F^2(d(p, x)),
        \end{align*}
where $F(r)$ is a positive, smooth, non-decreasing function on $[0,\infty)$. \label{defn01 b}
\end{enumerate}
Moreover, the subclass $\cM_n(F, K; V_0)\subset \cM_n(F, K)$ is defined to consist of all spaces that in addition satisfy
\begin{enumerate}[resume, label=(\alph*)]
   \item the non-collapsing condition
       \begin{align*}
          |B_1(p)|_{f}\geq V_0,
        \end{align*}
where $|\cdot|_{f}$ denotes the volume with respect to the weighted measure $dv_f=\exp(-f)dv$. \label{defn01 c}
 \end{enumerate}
\end{defn}

Li-Li-Wang \cite{MR4220743}*{Theorem 10.2} proved the $C^{1,\alpha}$-compactness theorem of $\cM_n(F, K; V_0)$.

\begin{thm}\label{Theo:LLW}
Let $(M_i^n, g_i, p_i)$ be a sequence of complete Riemannian manifolds in $\mathcal M_n(F,K;V_0)$.
By passing to a subsequence if necessary, the associated metric measure spaces (shown in Definition \ref{weighted volume}) converge in the pointed measured Gromov-Hausdorff sense
\begin{align*}
(M_i^n, p_i, d_i, f_i) \xrightarrow{pmGH} \left(M_{\infty}, p_{\infty}, d_{\infty},f_{\infty} \right).
\end{align*}
Furthermore, the convergence can be improved to
\begin{align*}
(M_i, p_i, g_i,f_i) \xrightarrow{pC^{1,\alpha}} \left(M_{\infty}, p_{\infty}, g_{\infty}, f_{\infty} \right).
\end{align*}
The space $M_{\infty}$ has a regular-singular decomposition $M_{\infty}=\mathcal{R} \cup \mathcal{S}$ in the sense of Definition \ref{regular-singular decomposition}.
\end{thm}
\begin{rem}\label{Rem1}In Definition \ref{defn01} \ref{defn01 a}, the constant $K$ may be replaced by a positive, non-decreasing function of the distance $d(p, x)$. With this change, one can verify that the conclusion of Theorem \ref{Theo:LLW} still holds.
\end{rem}

\subsubsection{Proof of \thmref{main convergence 3 main body} (=Theorem \ref{thm:C1alpha-compactness})}
We now apply Theorem \ref{Theo:LLW} to prove \thmref{main convergence 3 main body}. For this purpose, it suffices to check that the class $\mathcal M(\Lambda,F,\underline{\nu})$ introduced in Definition \ref{3 assumptions} satisfies the hypotheses characterizing $\mathcal M_n(F,K; V_0)$ in Definition \ref{defn01}.

To begin with, substituting the lower bound $-\Lambda(d)$ of the non-positive scalar curvature $S$ from condition \ref{Scalar lower bound} in Definition \ref{3 assumptions} and the uniform bound of the soliton constant $\lambda$ into \lemref{Ricci bound} yields the required estimate for the Bakry-\'{E}mery Ricci curvature that appears in condition \ref{defn01 a} of Definition \ref{defn01}, for a suitable function $K$ depending on the point of $M$.

Next, incorporating the bounds $F$ on $|\nabla f|\leq F(d)$ provided by condition \ref{f condition} in Definition \ref{3 assumptions}, we obtain the bound on $f$:
\begin{align}\label{Rem2}
|f(x)|=|f(x)-f(p)|\leq\int_0^{d(p,x)}F(s)ds\leq d(p,x)F(d(p,x)),
\end{align}
which is needed in condition \ref{defn01 b} of Definition \ref{defn01}.

Thus, the growth function $F^2$ in \ref{defn01 b} in Definition \ref{defn01} could be chosen to be
\begin{align*}
1+r^2F(r)^2+\Lambda(r)+F(r)^2.
\end{align*}

Lastly, the oscillation of $f$ is controlled by the gradient assumption and the normalisation $f(p)=0$, as shown in \eqref{Rem2}. 
Then we see that the weighted volume is bounded by the standard volume
\begin{align*}
\exp(-\sup_{B_1(p)}f)|B_1(p)|\leq |B_1(p)|_f\leq \exp(-\inf_{B_1(p)}f)|B_1(p)|.
\end{align*}
Therefore, the entropy lower bound in condition \ref{noncollapsing} of Definition \ref{3 assumptions}, together with Lemma \ref{noncollapse}, yields the required lower bound on the weighted volume of the unit geodesic ball in condition \ref{defn01 c} of Definition \ref{defn01}.

Combining these observations, we conclude that Theorem \ref{main convergence 3 main body} follows from Theorem \ref{Theo:LLW}, together with Remark \ref{Rem1} and  \eqref{Rem2}.

The strong convexity of the regular part follows from \cite{MR4234100}*{Theorem 4.10}.  
\subsection{Pointed \texorpdfstring{$C^{\infty}$}{C^{\infty}} convergence}\label{Pointed smooth convergence}
\begin{defn}[Harmonic coordinates]
Given an $n$-dimensional Riemannian manifold $(M, g)$, a local coordinate chart $(x^1, \dots, x^n)$ is called harmonic if each coordinate function is harmonic with respect to the Laplace-Beltrami operator $\tri_g$, meaning that
\[
\tri_g x^i = 0 \quad \forall i = 1, \dots, n.
\]
\end{defn}
In harmonic coordinates, the Ricci curvature can be expressed as a nonlinear elliptic operator:
\begin{equation}\label{Ricci equation general}
\frac{1}{2} g^{ab} \frac{\partial^2 g_{ij}}{\partial x^a \partial x^b} =-R_{ij} +Q_{ij}(g, \partial g),
\end{equation}
where $Q_{ij}(g, \partial g)$ is quadratic in the first derivatives $\partial g$ of the metric and has coefficients that depend on $g$ and its inverse $g^{-1}$.
\begin{defn}[$C^{1,\alpha}$ harmonic radius]
Fix a H\"older exponent $\alpha \in (0,1)$, a constant $C>1$ and a small universal constant $\epsilon_0 = \epsilon_0(n, \alpha) > 0$. For a point $p \in M$, the harmonic radius $hr_g(p)$ is defined as the supremum of all $r > 0$ such that there exists a harmonic coordinate chart $\Phi = (x^1, \dots, x^n)$ on a geodesic ball $B_r(p)$ such that $\Phi(p) = 0$ and
the metric satisfies 
    the scale-invariant estimates:
\[C^{-1}\delta\leq g\leq C \delta,\quad
r |\partial g|_\infty +
r^{1+\alpha} [\partial g]_\alpha \le \epsilon_0,
\]
where $[\cdot]_{C^{0,\alpha}}$ denotes the H\"older seminorm with respect to the Euclidean metric in the harmonic coordinate chart.
\end{defn}

\subsubsection{Estimates of the $G_2$-soliton assuming bounded Riemann curvature}
We recall \thmref{main convergence 3 main body}.
For each $R>0$, we fix a compact set $K\subset B_R(p_\infty)\cap \mathcal R$. Then there is a sequence of diffeomorphisms 
$$\sigma_{i,R}: B_R(p_\infty)\cap \mathcal R \to M_i$$
such that $\sigma_{i,R}^\ast (g_i )$ has uniform $C^{1,\alpha}(g_\infty)$ bound on $K$ and 
\begin{align*}
\sigma_{i,R}^\ast (g_i )\xrightarrow{C^{1,\alpha}} g_\infty \quad \text{on } K.
\end{align*}

We take a harmonic coordinate chart
$$\Omega := B(x,r),$$ on the regular part
where $r$ denotes the harmonic radius at $x$, and this chart is one of the harmonic coordinate neighbourhoods that cover $K$.
The limiting metric is
$$g_\infty\in C^{1,\alpha}(\Omega).$$ 
The sequence of $G_2$-solitons $\left( M_i, \vphi_i, g_i, f_i,p_i\right)$ is pulled back by $\sigma_{i,R}$ to $\Omega$.
\begin{rem}\label{soliton regularity ques}
In order to answer Question \ref{Regularity main question}, 
it is sufficient to strengthen the $C^{1,\alpha}$ bound for the Riemannian metric $\sigma_{i,R}^\ast(g_i)$ obtained in \thmref{main convergence 3 main body} to yield higher-order estimates, and also to improve the estimates for the $G_2$-structure $\sigma_{i,R}^\ast (\varphi_i)$ and the potential function $\sigma_{i,R}^\ast (f_i)$.
\end{rem}
In this section, we derive higher-order estimates for $G_2$-solitons, which will be applied later in the proof of \thmref{thm:smooth-convergence-regular-set}.

\begin{proposition}[Quantitative interior estimates in harmonic coordinates]
\label{lift}
Let $(M,\varphi,g,f)$ be a smooth gradient closed $G_2$-soliton
with soliton constant $\lambda$ and basepoint $p$. Let
\[
x=(x^1,\ldots,x^7):\Omega\longrightarrow V\subset\mathbb R^7
\]
be a harmonic coordinate chart, and let $\Omega'\Subset \Omega$.
Fix $0<\alpha<1$.

Assume that, in these harmonic coordinates,
\begin{equation}
\label{eq:harmonic-initial-bounds}
\kappa^{-1}\delta_{ij}\leq g_{ij}\leq \kappa\delta_{ij},
\qquad
\|g\|_{C^{1,\alpha}\cap W^{2,q}(\Omega)}
+\|Rm\|_{L^\infty(\Omega)}
\leq K,
\end{equation}
and that
\begin{equation}
\label{eq:potential-normalization-local}
|\lambda|\leq L_\lambda,
\qquad
\|f\|_{L^\infty(\Omega)}\leq RF(R).
\end{equation}

Then, for every integer $k\geq0$, there exists a constant
\[
C_k
=
C_k\bigl(
\alpha,\kappa,K,L_\lambda,R,F(R),
\Omega',\Omega
\bigr)
\]
such that
\begin{align}
\|g\|_{C^{k+2,\alpha}(\Omega')}
&+
\|\varphi\|_{C^{k+2,\alpha}(\Omega')}
+
\|f-c\|_{C^{k+2,\alpha}(\Omega')}
\nonumber\\
&+
\|T\|_{C^{k+1,\alpha}(\Omega')}
+
\|Rm\|_{C^{k,\alpha}(\Omega')}
\leq C_k.
\label{eq:full-tuple-local-estimate}
\end{align}
All norms in \eqref{eq:full-tuple-local-estimate} are coordinate norms
in the given harmonic coordinate chart.

Equivalently, the dependence on $\Omega'\Subset \Omega$ may be expressed in
terms of $d_{\mathbb R^7}(\Omega',\partial \Omega)$ and the size of the
coordinate domain.
\end{proposition}

We could replace the condition $\|Rm\|_{L^\infty(\Omega)}\leq C_0$ with $|T|_{C^{0,\alpha}(\Omega)}\leq C_0$. 
\begin{prop}\label{lift2}
	Let $(M,\varphi,g,f,p)$ be a complete gradient $G_2$-soliton with basepoint $p$. 
	and let $\Omega$ be a harmonic coordinate chart on the regular part of $M$. 
	Suppose that the Riemannian metric $g$ has $C^{1,\alpha}(\Omega)$ bound and $T$ has  $C^{0,\alpha}(\Omega)$ bound, 
	i.e. there exists a constant $C_0$ such that
	\[|g|_{C^{1,\alpha}(\Omega)}+|T|_{C^{0,\alpha}(\Omega)}\leq C_0.\]
	Then the same conclusion as in Proposition \ref{lift} holds.
\end{prop}

We will make use of the following equations for a gradient $G_2$-soliton, as stated in Proposition \ref{lem:Sc df} and Proposition \ref{tri T soliton}:
\begin{align}
\label{btsr1}
Ric + \nabla^2 f &= -\frac{\lambda}{3} g - \frac{|T|^2}{3} g - 2 {T_i}^{k} T_{kj},\\
\label{btsr2}
3 \tri f &= -7\lambda - 2S,\\
\label{btsr4}
\tri \varphi &= Rm * \varphi - \lambda \varphi - d i_{\nabla f}\varphi,\\
\label{btsr3}
\tri T &= Rm \ast \nabla f \ast \vphi + Rm \ast T \ast \vphi \ast \vphi + Rm \ast T\\
&\quad + \nabla T \ast T \ast \vphi + T \ast T \ast T \ast \psi.\notag
\end{align}

We begin to prove Proposition \ref{lift} by deriving estimates for the potential function $f$.

\begin{lem}\label{higher estimate nabla 2 f}
Suppose the Ricci curvature of a gradient $G_2$-soliton is uniformly bounded. Then there exists a constant $C$ depending on the Ricci curvature bound such that the Hessian $\nabla^2 f$ is uniformly bounded
\[
\|\nabla^2 f\|_{L^\infty} \le C(\|Ric\|_{L^\infty}, \lambda).
\]
\end{lem}
\begin{proof}
If the Ricci curvature is bounded, then the scalar curvature is bounded. On the other hand, for a closed $G_2$-structure, the scalar curvature is $-|T|^2$. Hence the right-hand side of \eqref{btsr1} is controlled by the absolute value of the scalar curvature. Therefore, \eqref{btsr1} implies a bound on $\nabla^2 f$.
\end{proof}
\begin{rem}
The assumption of $\|Ric\|_{L^\infty}$ in the lemma could be replaced by $\|T\|_{L^\infty}$. 
\end{rem}
We then repeatedly apply the interior $L^q$-estimates and the interior Schauder estimates to derive higher order bounds for the $G_2$-structure $\varphi$, the potential function $f$, and the metric $g$. All the following estimates are taken on a smaller subdomain and depend on its distance from the boundary of $\Omega$.
\begin{lemma}\label{higher estimate f}
Assume that $g\in C^{1,\alpha}(\Omega)$ and that the scalar curvature admits a uniform lower bound. Then the potential function satisfies
\[
f\in W^{2,q}\hookrightarrow C^{1,\alpha},\quad 0<\alpha<1-\frac{n}{q},
\]
where the dependence of these norms is the same as in Proposition \ref{lift}.
\end{lemma}

\begin{proof}
By our hypotheses, the quantity $$-7\lambda-2S$$ is bounded, and the Levi-Civita connection of $g$ is of class $C^\alpha$. The potential function $f$ has $L^q$ bound by using the normalisation $f(p)=0$ and the gradient bound from condition \ref{f condition}.  

Applying the $L^q$-estimate \cite{MR1814364}*{Theorem 9.11} to \eqref{btsr2} yields the desired estimate.
\end{proof}

We now use the bound on the Riemann curvature tensor.

\begin{lemma}\label{Lpphi}
Under the assumptions of Proposition \ref{lift}, one has
\[
\varphi\in W^{2,q}\hookrightarrow C^{1,\alpha},\quad T\in C^{0,\alpha}.
\]
The norms depend on the same quantities as in Proposition \ref{lift}.
\end{lemma}
\begin{proof}
In \eqref{btsr4}, the leading coefficient is $C^{1,\alpha}$, since $g\in C^{1,\alpha}$.
\lemref{higher estimate f} and \lemref{higher estimate nabla 2 f} yield bounds for $\nabla f$ and $\nabla^2 f$. From assumption, the Riemann curvature is bounded. As a result, all coefficients in front of $\varphi$ and $\nabla\varphi$ are bounded. Therefore, the $W^{2,q}$-estimate of $\varphi$ follows from the classical $L^q$-estimate. 

By \eqref{psi}, $\psi$ is $C^{1,\alpha}$. The H\"older continuity of $T$ is then obtained from \eqref{T and vphi 1}.
\end{proof}

\begin{lemma}
Under the hypotheses of Proposition \ref{lift}, we have
        $$T\in W^{2,q}\hookrightarrow C^{1,\alpha}.$$
The norms depend on the same quantities as in Proposition \ref{lift}.
\end{lemma}
\begin{proof}
In \eqref{btsr3}, the coefficient in the Laplacian operator is $C^{1,\alpha}$, since $g\in C^{1,\alpha}$.
Lemma \ref{Lpphi} ensures that the coefficients in front of $\nabla T$ are bounded. The remaining terms on the right-hand side are likewise bounded, thanks to the assumption on $Rm$ and the estimate for $\nabla f$ given by Lemma \ref{higher estimate f}. Consequently, the standard $L^q$-estimate yields the required bound for $T$.
\end{proof}

\subsubsection{Induction argument}
We now proceed by induction. Note that, for each $k$, we need to use a domain $\Omega_k$ compactly contained in the previous one $\Omega_{k-1}$.
\begin{lemma}\label{bootstr1}
Assume that for some integer $k\geq 0$,
$$g\in C^{k+1,\alpha},\quad S\in C^{k,\alpha},\quad f\in C^{k,\alpha}.$$
Then
$$f\in C^{k+2,\alpha}.$$
The norms depend on the same quantities as in Proposition \ref{lift}.
\end{lemma}

\begin{proof}
By the inductive assumption, both the right-hand side and the coefficients in \eqref{btsr2} lie in $C^{k,\alpha}$. Schauder estimates then give a $C^{k+2,\alpha}$ bound for $f$, completing the proof.
\end{proof}

\begin{rem}\label{bootstr1 remark}
Since $|T|^2=-S$, the assumption $S\in C^{k,\alpha}(\Omega)$ in \lemref{bootstr1} could be replaced by
$$T\in C^{k,\alpha}(\Omega).$$
\end{rem}

\begin{lemma}\label{bootstr2}
Assume that, for any integer $k\geq 0$,
$$g\in C^{k+1,\alpha},\quad f\in C^{k+2,\alpha},\quad T\in C^{k,\alpha}.$$
Then we obtain
$$g\in C^{k+2,\alpha},\quad Rm\in C^{k,\alpha}.$$
The norms depend on the same quantities as in Proposition \ref{lift}.
\end{lemma}
\begin{proof}
From \eqref{btsr1}, the Ricci curvature is of class $C^{k,\alpha}$. Applying the Schauder estimates to the local elliptic equation \eqref{Ricci equation general} under the harmonic coordinate, we obtain the $ C^{k+2,\alpha}$ bound of $g$.
\end{proof}

\begin{lemma}\label{bootstr3}
Suppose that, for any integer $k\geq 0$,
$$g\in C^{k+2,\alpha},\quad f\in C^{k+2,\alpha},\quad \varphi\in C^{k,\alpha}.$$
Then
$$\varphi\in C^{k+2,\alpha},\quad T\in C^{k+1,\alpha}.$$
The norms depend on the same quantities as in Proposition \ref{lift}.
\end{lemma}
\begin{proof}
In \eqref{btsr4}, all coefficients in front of $\varphi$ and $\nabla\varphi$ belong to $C^{k,\alpha}$. Applying the Schauder estimates, we obtain a $C^{k+2,\alpha}$ norm of $\vphi$. Thus, $\psi\in C^{k+2,\alpha}$ by \eqref{psi}
and the claimed estimate of $T$ is then obtained from \eqref{T and vphi 1}.
\end{proof}

Finally, we prove the smoothness of the $G_2$-soliton.
\subsubsection{Proof of Proposition \ref{lift}}
Under the hypotheses of Proposition \ref{lift}, we aim to establish, for every integer $k\geq 0$, the estimates
\begin{align}\label{higher estimate induction}
g\in C^{k+1,\alpha},\quad f\in C^{k+1,\alpha},\quad T\in C^{k,\alpha},\quad \varphi\in C^{k+1,\alpha}.
\end{align}

For $k=0$, \eqref{higher estimate induction} follows from \lemref{higher estimate f} and \lemref{Lpphi}.
Assume now that \eqref{higher estimate induction} holds for some $k$. Then, by Lemma \ref{bootstr1} together with Remark \ref{bootstr1 remark}, we obtain $$f\in C^{k+2,\alpha}.$$ Next, by \lemref{bootstr2}, we deduce that
\begin{align*}
g\in C^{k+2,\alpha},\quad Rm\in C^{k,\alpha}.
\end{align*}Moreover, applying \lemref{bootstr3} yields
\begin{align*}
 \varphi\in C^{k+2,\alpha},\quad T\in C^{k+1,\alpha}
\end{align*} Hence the estimates \eqref{higher estimate induction} also hold for $k+1$.
The specific dependence of each norm is detailed in the respective lemmas.

\subsubsection{Proof of Proposition \ref{lift2}}

From $S=-|T|^2$, the condition $T\in C^{0,\alpha}(\Omega)$ implies
$$S\in C^{0,\alpha}(\Omega).$$ 
So the quantity $-7\lambda-2S$ is of class $C^{0,\alpha}$. 
Then Lemma \ref{higher estimate f} yields that 
$$f\in C^{1,\alpha}.$$
Following this, Lemma \ref{bootstr1} improves the estimate of $f$ to 
$$f\in C^{2,\alpha}.$$
Furthermore, Lemma \ref{bootstr2} gives 
$$Rm\in C^{0,\alpha}.$$
So we get the H\"older estimate of the Riemann curvature and Proposition \ref{lift2} can be proved by using Proposition \ref{lift}.

\begin{rem}
In \cite{MR4220743}*{Section 6}, Li-Li-Wang proved that the $C^{1,\alpha}$ Ricci shrinkers on the regular set can be enhanced to be $C^{\infty}$. The key component of their proof is a local regularity improvement theorem for Ricci shrinkers, which is derived via a conformal change of the metric combined with a bootstrapping argument.

The equations satisfied by the $G_2$-soliton in \eqref{btsr4} and \eqref{btsr3} involve the Riemann curvature, which is why a bound on the Riemann curvature is required (see Proposition \ref{lift}).

In $G_2$-solitons, the equations in \eqref{btsr1} and \eqref{btsr2} involve the torsion tensor $T$, so we can replace the curvature bound to the condition $|T|_{C^{0,\alpha}(\Omega)}\leq C_0$ (see Proposition \ref{lift2}).
\end{rem}

\begin{rem}
Formally, $g$ has the same estimate as $\varphi$ and the estimate of $T$ is one order lower than that of $\varphi$. But a given metric may be compatible with more than one $G_2$-structure so we cannot obtain $T\in C^{0,\alpha}$ from $g\in C^{1,\alpha}$.
\end{rem}

\subsubsection{Smooth convergence on regular part}
\begin{defn}\label{Cinfty convergence}
Using the notion introduced in Definition \ref{c1alpha convergence}, we say that
a sequence of complete $G_2$-structures
\[(M_i,\vphi_i,g_i,f_i,p_i)\]
converges in the pointed $C^{\infty}$ topology, written
\[(M_i,\vphi_i,g_i,f_i,p_i)\xrightarrow{pC^{\infty}}(M_\infty,\vphi_\infty,g_\infty,f_\infty,p_\infty),\]
if the following conditions are satisfied:
\begin{itemize}
\item $(M_i,g_i,f_i,p_i) \xrightarrow{pC^{\infty}} (M_\infty,g_\infty,f_\infty,p_\infty)$,
\item the $G_2$-structures converge smoothly on $K$:
\begin{align*}
\sigma_{i,R}^\ast (\vphi_i )\xrightarrow{C^{\infty}(g_\infty)} \vphi_\infty,
\end{align*}
\item the limit metric $g_\infty$ and the limiting $G_2$-structure $\vphi_\infty$ are compatible, in the sense that they satisfy \eqref{G2 metric}.
\end{itemize}
\end{defn}
\begin{defn}\label{smooth soliton}
We say that a gradient $G_2$-soliton $(M,\varphi,g,f,p)$ is smooth if the Riemannian metric, the $G_2$-structure, and the potential function are all smooth, that is,
$$\varphi,\quad g,\quad f\in C^{\infty}(M).$$
\end{defn}
\begin{thm}[Smooth convergence on the regular set under local Riemann curvature bounds]
\label{G2 smooth convergence}
Let
$(M_i,\varphi_i,g_i,f_i,p_i)$
be a sequence of complete $G_2$-solitons belonging to the class $\mathcal M(\Lambda,F,\underline{\nu})$ with fixed soliton constant $\lambda$. 

For every compact $K \subset B_R(p_\infty)\cap \mathcal R$,
assume that the Riemann curvature tensors of the metrics $g_i$ are uniformly bounded on every compact $\sigma_{i,R}(K)$:
\begin{align*}
\sup_i\sup_{\sigma_{i,R}(K)}|Rm(g_i)|\leq C_K,\quad \forall K\Subset B_R(p_\infty)\cap\mathcal R,
\end{align*}
where $\sigma_{i,R}$ are the diffeomorphisms specified in \thmref{main convergence 3 main body} and Definition \ref{c1alpha convergence}.

Then, after extracting a subsequence, this sequence converges smoothly
\[(M_i,\vphi_i,g_i,f_i,p_i)\xrightarrow{pC^{\infty}}(M_\infty,\vphi_\infty,g_\infty,f_\infty,p_\infty)\]
and the limit is a conifold $G_2$-soliton.  Moreover, it converges smoothly to a smooth $G_2$-soliton on the regular part $\mathcal R$. 
\end{thm}
\begin{proof}
By pointed Gromov-Hausdorff convergence, as in Definition \ref{pointed Gromov-Hausdorff convergence}, we obtain a family of $\Phi_{i,R}$ defined on $B_R(p_i)$. To promote this to $C^\infty$ convergence in the sense of Definition \ref{Cinfty convergence}, we need, for each compact subset $K$ of $\mathcal R \cap B_R(p_\infty)\subset M_\infty$, a sequence of diffeomorphisms $\sigma_{i,R}$ that approximate the maps $\Phi_{i,R}$.

On subsets where the Riemann curvature is uniformly bounded, Proposition \ref{lift} provides uniform higher-order estimates of $\sigma_{i,R}^\ast(g_i)$ in harmonic coordinates. With these in hand, the construction of smooth diffeomorphisms proceeds exactly as in the proof of \thmref{main convergence 3 main body} (see \cite{MR3469435}*{Section 11}), and in particular the $C^{1,\alpha}$ convergence of the metrics is improved to smooth convergence. 

Proposition \ref{lift} also provides uniform higher-order estimates of $\sigma_{i,R}^\ast(\vphi_i)$ and $\sigma_{i,R}^\ast(f_i)$.
The same reasoning applies to the $G_2$-structures and to the potential functions, which consequently also converge smoothly to $\vphi_\infty$ and $f_\infty$ respectively with respect to $g_\infty$.

 Then, after passing to a further subsequence, there exists a smooth \(3\)-form \(\varphi_\infty\), a smooth metric $g_\infty$ and a smooth function $f_\infty$ on a neighbourhood of \(K\) such
that
\begin{align*}
\sigma_{i,R}^\ast (g_i )\xrightarrow{C^{\infty}(g_\infty)} g_\infty,\quad
\sigma_{i,R}^\ast (\vphi_i )\xrightarrow{C^{\infty}(g_\infty)} \vphi_\infty,\quad
\sigma_{i,R}^\ast (f_i )\xrightarrow{C^{\infty}(g_\infty)} f_\infty \text{ on }K.
\end{align*}

Taking limit of \eqref{G2 metric}, we have the relation between $g_\infty$ and $\vphi_\infty$:
\begin{align*}
g_{\infty}(X,Y) \cdot dv_{g_{\infty}}=\frac{1}{6}i_X\vphi_\infty\wedge i_Y\vphi_\infty\wedge\vphi_\infty, \quad \forall X,Y\in TM.
\end{align*}
That implies the positivity of $\vphi_\infty$ and $g_{\infty}$ is the metric associated to the $G_2$-structure $\vphi_\infty$.

Since $\vphi_i$ is closed, the limit $\vphi_\infty$ is also closed. Taking limit of \eqref{$G_2$-soliton}, we have $(M,\varphi_\infty,g_\infty,\nabla^\infty f_\infty)$ is a $G_2$-soliton
\begin{align*}
\Delta_{\varphi_\infty} \varphi_\infty = \lambda \varphi_\infty + L_{\nabla^\infty f_\infty} \varphi_\infty \quad \text{on } K.
\end{align*}
At last, we let $R\rightarrow \infty$ and perform a diagonal argument with explicitly compatible embeddings $\sigma_{i,R}$ to conclude the smooth convergence on the regular set of the whole manifold $M$.
\end{proof}
\begin{ques}
Motivated by recent progress on Ricci solitons \cites{MR4711837,MR5055602}, one may ask whether, for $G_2$-solitons, a bound on the scalar curvature implies a bound on the full Riemann curvature tensor.
\end{ques}

\subsubsection{Conformal transformation}
Under the conformal transformation
		\begin{equation}
			\bar{g}=\exp(2\phi)g,\quad \phi=\frac{f(q)-f}{5},
		\end{equation}
we have
\begin{align}\label{cfmRim}
\overline{Rm}& =\exp(2\phi)\left[ Rm-g\KN\left(\nabla^2\phi-d\phi\otimes d\phi+\frac{1}{2}|\nabla\phi|^2g \right)\right] ,\\
\label{cfmric}
\overline{Ric}& =Ric-(n-2)\left(\nabla^2\phi-d\phi\otimes d\phi \right)-\left[\tri\phi+(n-2)|\nabla\phi|^2 \right]g,\\
\label{cfmS}
\overline{S}& =\exp(-2\phi)\left[ S-2(n-1)\tri\phi-(n-1)(n-2)|\nabla\phi|^2 \right],\\
\label{btsr6}
\overline{\tri}f&=\exp(-2\phi)\tri f -|\nabla f|^2_{\overline{g}}.
\end{align}
\begin{rem}
In the study of gradient Ricci solitons
\begin{equation}\label{eq2Ricshrinker} 
\mathrm{Ric} + \nabla^2 f = \lambda g,
\end{equation}
we have a trace equation and a conservation law
\begin{equation}\label{eqRicshrinker Hamilton}
S+\tri f=\lambda \cdot n,\quad S + |\nabla f|^2 - 2\lambda f = \text{constant},
\end{equation}
see Hamilton's work on the Ricci flow \cite[Section 20, p.~84]{MR1375255}. These identities are crucial tools, for instance, in controlling the growth of the potential function $f$, as well as in studying volume growth, see \cite{MR2732975}.

In Section \ref{G2 soliton}, we obtain analogous identities for $G_2$-solitons. In this context, the corresponding conservation law is more intricate, as indicated in \eqref{gradient gradient f}:
\begin{align*}
\nabla [S + |\nabla f|^2 + \tfrac{2\lambda}{3} f]
  = \tfrac{2}{3}S\nabla f -2\div A,
\end{align*} where $A$ is the symmetric product of the torsion form $\tau$, $
A_{ij}=-\frac{1}{2}{\tau_i}^{k}\tau_{kj}$, see Definition \ref{ABS definition}. \end{rem}
\begin{rem}\label{elliptic difficulty}
Combining the identities in \eqref{eqRicshrinker Hamilton} in gives
\begin{equation}\label{eq1Ricshrinker}
	\triangle f - |\nabla f|^2 +2\lambda f = \text{const}.
\end{equation}

Starting from $g\in C^{1,\alpha}$ and $f\in C^{1}$, the equations \eqref{eq2Ricshrinker} and the first equation in \eqref{eqRicshrinker Hamilton} alone do not suffice to improve estimates; see \cite{MR4220743}*{Remark 4.13}. In contrast, the combination of \eqref{eq2Ricshrinker} and \eqref{eq1Ricshrinker} is sufficient.

In Li-Li-Wang \cite{MR4220743}*{Page 28}, they apply the bootstrapping argument 
for the conformal change of metric
$\tilde{g}=\exp(-\frac{2f}{n-2})g,$
where $n$ denotes the dimension of the Ricci shrinker. This conformal transformation removes the $Hess f$ term from the equations. 

Then they perform the rescaling
$$\bar{g}=\rho^{-2}\tilde{g}\quad\text{and}\quad\bar{f}=f-f(q),$$
where $\rho$ denotes the harmonic radius of $\tilde{g}$ at the point $x$.
The equations \eqref{cfmric} and \eqref{eq1Ricshrinker} can be rewritten as
\begin{align}
\overline{Ric}=&\frac{1}{n-2}\left\lbrace d\bar{f}\otimes d\bar{f}+\rho^2\left[n-1-\bar{f}-f(q) \right]\exp(\frac{2\bar{f}}{n-2})\bar{g} \right\rbrace, \label{eq3Ricshrinker} \\
	\overline{\tri}\bar{f}=&\exp(\frac{2\bar{f}}{n-2})\rho^2\left[\frac{n}{2}-\bar{f}-f(q) \right].\label{eq4Ricshrinker}
\end{align}

Assuming $f\in C^{1}$ and $g\in C^{1,\alpha}$, the Schauder estimate applied to \eqref{eq4Ricshrinker} improves the estimate of $f$ to
$\bar f\in C^{1,\alpha}.$
Using this and the Schauder estimate for \eqref{eq3Ricshrinker}, we obtain
$\bar g\in C^{2,\alpha}.$
Repeating this bootstrapping argument, one can upgrade the estimate of both $g$ and $f$ to $C^{\infty}$.
\end{rem}
Comparing \eqref{eq4Ricshrinker} with \eqref{btsr6}, we observe that, in order to deduce $f\in C^{2,\alpha}$, one must impose an extra condition, for instance
	$$\|Rm\|_{L^\infty(\Omega)}\leq C\quad\text{or}\quad |T|_{C^{0,\alpha}(\Omega)}\leq C,$$
because, after performing the conformal change, the term $\overline{\tri} f$ still depends on the scalar curvature $S$ via \eqref{btsr2}, as detailed in Section \ref{Pointed smooth convergence}.

\section{Conifold compactness}
\label{Conifold compactness}
In this section, we prove the conifold compactness \thmref{thm:smooth-convergence-regular-set} in the following two steps. The first step is to prove it under small entropy as shown in \thmref{thm:conifold-compactness main}. The second step is to verify the small entropy condition in Section \ref{Verify the small entropy condition}.
\subsection{Conifold compactness under small entropy}
\begin{thm}\label{thm:conifold-compactness main}
Let $\left( M_i, \vphi_i, g_i, f_i,p_i\right)$ be a sequence in
\[\mathcal M(\Lambda,F,\underline{\nu})\]
with fixed soliton constant $\lambda$.
According to \thmref{thm:C1alpha-compactness}, after passing to a subsequence,
\[
(M_i,g_i,f_i,p_i) \xrightarrow{pC^{1,\alpha}} (M_\infty,g_\infty,f_\infty,p_\infty)
\]
in the sense of Definition \ref{c1alpha convergence} and the limit admits a regular-singular decomposition $M_\infty=\mathcal R \cup \mathcal S.$

Fix \(R>0\), a compact set
$
K\Subset B_R(p_\infty)\cap\mathcal R,
$
and an open set \(U\) satisfying
\[
K\Subset U\Subset B_R(p_\infty)\cap\mathcal R.
\]
Let
$
\sigma_{i,R}:U\longrightarrow M_i
$
be the restrictions of the convergence embeddings from
Definition~\ref{c1alpha convergence}.

Suppose that there exist \(r_K>0\) and \(i_K\in\mathbb N\) such that,
for every \(i\ge i_K\) and every \(x\in K\),
$
B_{g_i}\bigl(\sigma_{i,R}(x),2r_K\bigr)
\Subset \sigma_{i,R}(U)
$
and
\begin{align}\label{small local entropy condition}
\nu
\left(
 B_{g_i}\bigl(\sigma_{i,R}(x),2r_K\bigr),
 g_i,
 (2r_K)^2
\right)
\ge -\epsilon,
\end{align}
where \(\epsilon\) is the constant in
Theorem~\ref{Epsilon regularity}.

 Then, after passing to a further subsequence, there exists a smooth
positive \(3\)-form \(\varphi_\infty\) on a neighbourhood of \(K\) such
that
\begin{align*}
\sigma_{i,R}^\ast (g_i )\xrightarrow{C^{\infty}(g_\infty)} g_\infty,\quad
\sigma_{i,R}^\ast (\vphi_i )\xrightarrow{C^{\infty}(g_\infty)} \vphi_\infty,\quad
\sigma_{i,R}^\ast (f_i )\xrightarrow{C^{\infty}(g_\infty)} f_\infty \text{ on }K.
\end{align*}
Moreover, on the corresponding neighbourhood of \(K\)
\[
g_{\varphi_\infty}=g_\infty,\quad
d\varphi_\infty=0,\quad
\Delta_{\varphi_\infty}\varphi_\infty
=
\lambda\varphi_\infty
+
\mathcal L_{\nabla^{g_\infty}f_\infty}\varphi_\infty.
\]

If the preceding entropy hypothesis holds for every compact subset
\(K\Subset\mathcal R\), with a corresponding radius \(r_K>0\), then a
diagonal subsequence satisfies
\[
(M_i,\varphi_i,g_i,f_i,p_i)
\xrightarrow{pC^\infty}
(M_\infty,\varphi_\infty,g_\infty,f_\infty,p_\infty)
\]
on \(\mathcal R\) in the sense of
Definition~\ref{Cinfty convergence}. Consequently,
$
(M_\infty,d_\infty,g_\infty,f_\infty,p_\infty)
$
is a \(C^\infty\) singular space, see Definition~\ref{regular-singular decomposition} and
$
(\mathcal R,\varphi_\infty,g_\infty,f_\infty)
$
is a gradient \(G_2\)-soliton. In particular, the limit is a conifold
\(G_2\)-soliton in the sense of Definition~\ref{Conifold}.
\end{thm}
\subsection{Epsilon-regularity under small entropy}
\label{Epsilon regularity under small entropy}
\begin{thm}[Entropy epsilon-regularity theorem]\label{Epsilon regularity}
Let $(M^{7},\vphi,g,f,p)$ be a complete gradient $G_2$-soliton in $\mathcal M(\Lambda,F,\underline{\nu})$. 
Let $B_1:=B(y,1)\subset B(p,R)$ be a unit ball in $M$.

Suppose that the local $\nu$-functional of $g$ satisfies
\[
\nu(B_1,g,1)\ge -\epsilon,
\] where $\epsilon$ is the constant from \thmref{Gap theorem}.
Then for every integer $k\geq 0$, there exists a constant $C$ such that
\begin{align*}
    \sup_{B_1} |\nabla^k Rm|\, d^{2+k}(\cdot,\p B_1) \leq C.
\end{align*}
The constant $C$ depends on $7,k$ and the quantities $\Lambda(R),F(R),L_\lambda,\eps$ in Definition \ref{3 assumptions}.
\end{thm}

By using the higher-order estimates \thmref{Shi}, it is sufficient to prove it with $k=0$.

We proceed by adopting the same line of reasoning as in \cite{arXiv:2010.09981}*{Theorem 3.5, Theorem 3.6}.
\subsection{Local rescaling procedure}\label{Local rescaling procedure}

We now describe the \textbf{local rescaling procedure}, which is used to control the local curvature of a Riemannian metric.

For each point $x \in M$, we define
\[
Q_x := |Rm(g)|_g(x)
\]
to be the norm of the Riemann curvature tensor of $g$ at $x$. The distance function associated with $g$ will be denoted by $d$. Given a point $p_0\in M$, we fix a ball
\[
B := B_g(p_0,R)
\]
of radius $R$ centered at $p_0$ with respect to the metric $g$.

The local quantity of interest is the curvature scale-weighted norm
\begin{align}\label{defn Fx}
    \Psi_x^2 := Q_x\cdot d^2(x,\p B), \quad \forall \; x\in B.
\end{align}

\begin{lem}[Point-picking lemma]
\label{local rescaling procedure R r}
Assume that $Q$ is bounded in the closure $\bar B$, $\Psi_{p_0}>0$ and let $R>0$.
Then  \(\Psi_x\) attains a positive maximum at some point
\(q\in B\) and for any positive number
\begin{align}\label{r upper bound}
0<r< \Psi_q
\end{align}
and the rescaled metric
\begin{align}\label{small entropy rescaled metric}
\tilde g=Q_q\cdot g,
\end{align} one has
\begin{align*}
|\widetilde {Rm}(\tilde g)|(q)=1\quad\text{and}\quad
\sup_{B_{\tilde g}(q,r)}|\widetilde {Rm}(\tilde g)|\leq \frac{1}{\bigl(1-r \Psi^{-1}_{q}\bigr)^2}.
\end{align*}
In particular, choosing $r=\frac{\Psi_q}{2}$ we have
\begin{align*}
\sup_{B_{\tilde g}(q,r)}|\widetilde{Rm}(\tilde g)|\leq 4.
\end{align*}
\end{lem}

\begin{proof} The proof is divided into several steps.

\textbf{Step 1: choose a maximum point $q$.}
Since the Riemann curvature is bounded on $\bar B$, the function $\Psi_x$ must vanish along the boundary of $B$.
On the other hand, $\Psi_{p_0}>0$ at the point $p_0$.
Hence the maximum of $\Psi^2_x$ is achieved at an interior point of $B$,
denoted by $q$:
\begin{align}\label{local maximum}
    \Psi_q^2=Q_q  d^2(q,\p B)=\sup_{x\in B}\Psi^2_x\geq \Psi_{p_0}^2= Q_{p_0} R^2>0.
\end{align}
here we take $x=p_0$ in \eqref{defn Fx} and substitute in the last identity. Note that a priori $q$ need not coincide with $p_0$. 
We will now use the weighted curvature quantity in \eqref{local maximum} to derive a bound for the Riemann curvature near this maximum point $q$.

\textbf{Step 2: distance estimate near $q$.} In this step we control the weight $d(\cdot,\p B)$ in a neighbourhood of $q$.
From the first identity in \eqref{local maximum} we obtain
\begin{align*}
    d(q,\p B)=Q_q^{-\frac{1}{2}}\Psi_q.
\end{align*}
Choose a positive constant $0<r< \Psi_q$ such that the ball $B_g(q,Q_q^{-\frac{1}{2}}r)$ centred at $q$ is contained in $B$.

Then for any $x\in B_{g}(q,Q_q^{-\frac{1}{2}}r),$ the triangle inequality yields
\begin{align}\label{the distance lower bound}
    d(x,\p B)\geq   d(q,\p B)- d(x,q)
\geq Q_q^{-\frac{1}{2}}(\Psi_q-r).
\end{align}

\textbf{Step 3: curvature estimate near $q$.} We now control the curvature norm.
From the second relation in \eqref{local maximum} we have
\begin{align*}
    Q_x&\leq \frac{ \Psi_q^2}{ d^2(x,\p B)},\quad \forall x\in B_g(q,Q_q^{-\frac{1}{2}}r).
\end{align*}
Substituting the distance lower bound \eqref{the distance lower bound} into this inequality, we get
\begin{align}\label{local curvature bound original}
   Q_x\leq\frac{\Psi^2_q}{Q_q^{-1}(\Psi_q-r)^2}=\frac{Q_q}{(1-r\Psi_q^{-1})^2}.
\end{align}

\textbf{Step 4: rescaling the metric at $q$.}
Finally, we rescale the metric to obtain a uniform curvature bound.
With this rescaling $\tilde g=Q_q \cdot  g$ from \eqref{small entropy rescaled metric}, the ball $B_g(q,Q_q^{-\frac{1}{2}}r)$ is transformed into $B_{\tilde g}(q,r)$, and the Riemann curvature tensor along with its norm become
$$\widetilde{Rm} = Q_q \cdot Rm,\quad |\widetilde{Rm}| = Q_q^{-1} \cdot |Rm|.$$
In particular, at the point $q$ the curvature norm is normalised to
\begin{align*}
    |\widetilde {Rm}|(q)=1.
\end{align*}
Moreover, the local curvature estimate \eqref{local curvature bound original} transforms into
\begin{align}\label{local curvature bound scaled}
    \sup_{B_{\tilde g}(q,r)}|\widetilde{Rm}|\leq  \frac{1}{(1-r \Psi^{-1}_{q})^2},
\end{align}
which is precisely the bound we need.
\end{proof}

\begin{rem}
Since \(\Psi_q\geq \Psi_{p_0}\), for every \(r\) satisfying
$
    0<r<\Psi_{p_0},
$
we also have
\[
    \sup_{B_g(q,Q_q^{-1/2}r)}Q_x
    \leq
    \frac{Q_q}{(1-r/\Psi_{p_0})^2}.
\]
Indeed,
\[
    1-r/\Psi_q
    \geq
    1-r/\Psi_{p_0}>0.
\]
\end{rem}

\subsection{Gap theorem}
\begin{thm}[Gap theorem]\label{Gap theorem}
    There exists a small dimensional constant $\eps:=\eps(n)$ such that any complete, non-flat, $G_2$-soliton with zero scalar curvature and uniformly bounded gradient of the potential $|\nabla f|$ satisfies
    \begin{align*}
        \nu(M,g)< -\eps.
    \end{align*}
\end{thm}

\begin{proof}
    We argue by contradiction. Assume that for some sequence $\eps_i\to 0,$ there is a sequence of complete, non-flat $G_2$-solitons with vanishing scalar curvature and a potential with bounded gradient, denoted $(M_i,g_i)$, such that
    \begin{align}\label{nu lower bound}
        \nu(M_i,g_i)\geq -\eps_i.
    \end{align}

    \textbf{Step 1: construction of a rescaled sequence.}
We choose a point $(p_0)_i\in M_i$ where the curvature is non-zero and normalise the curvature to be $Q_{(p_0)_i}=1$ by using the constant scaling
\[Q_{(p_0)_i}^{3/2}\vphi_i,\quad Q_{(p_0)_i}g_i,\quad f_i-f_i((p_0)_i),\quad Q_{(p_0)_i}^{-1}\lambda_i.\]
Note that the entropy is invariant under the constant scaling.

We set a sequence of radii $r_i\rightarrow\infty$, as $i\rightarrow \infty$. We denote the ball
\[
B_i := B_{g_i}((p_0)_i,R_i),\quad R_i=2r_i.
\]
Let $q_i$ be the maximum of 
\[\Psi_i^2(x)=Q_i(x)\cdot d^2(x,\p B_i), \quad \forall \; x\in B_i.
\]
 We denote $Q_i=Q_i(q_i)$.  Then
 \begin{align*}
    \Psi_{q_i}^2=Q_{i}  d^2(q_i,\p B)\geq Q_{p_0} (2r_i)^2=(2r_i)^2.
\end{align*}

Thus, \lemref{local rescaling procedure R r} is applied and
the rescaled metrics $\tilde g_i=Q_i g_i$ satisfies
    \begin{align}\label{Scaling Riemann curvature}
        |\widetilde {Rm_i}|(q_i)=1,\quad \sup_{B_{\tilde g_i}(q_i,r_i)}|\widetilde {Rm_i}|\leq \frac{1}{\bigl(1-r_i \Psi_{q_i}^{-1}\bigr)^2}\leq 4.
    \end{align}  
 We rescale the $G_2$-structure, the soliton constant and the potential function as well
\begin{align*}
\tilde\vphi_i=Q_i^{\frac{3}{2}}\vphi_i,\quad \tilde \lambda_i=Q_i^{-1}\lambda_i,\quad \tilde f_i=f_i-f_i(q_i).
\end{align*}    

    \textbf{Step 2: njectivity radius.}
    By the monotonicity and scale invariance of $\nu$ in \lemref{nu properties}, we obtain    
    \begin{align}\label{nu lower bound tilde}
        \nu(B_{\tilde g_i}(q_i,r_i),\tilde g_i)\geq \nu(M_i,\tilde g_i)=\nu(M_i,g_i)\geq-\eps_i.
    \end{align}
    
    For $x\in B(q_i,r_i-1)$ and $B(x,1) \subset B(q_i,r_i)$, \eqref{nu lower bound tilde}  and the domain monotonicity of the local entropy imply $\nu(B(x,1),\tilde g_i,1)\geq-\epsilon_i.$
    Combining it with \lemref{noncollapse}, we get a uniform lower bound on the volume of geodesic balls.

Then a theorem of Cheeger-Gromov-Taylor \cite{MR658471}*{Page 42} yields a positive lower bound for the njectivity radius in the smaller ball $B_{\tilde g_i}(q_i,r_i-1),$ which is a dimensional constant $c(n)>0$.

   \textbf{Step 3: regularity improvement.}
Since each soliton has zero scalar curvature and
\[
    S(g_i)=-|T_i|^2,
\]
the \(G_2\)-structure is torsion-free
$
    T_i\equiv0.
$
From the expression of Ricci curvature from \lemref{scalar curvature general}, the metric is therefore Ricci-flat $
Ric(g_i)=0.
$
The soliton equation \eqref{Ricci curvature soliton} consequently reduces to
\[
    \nabla^2 f_i=-\frac{\lambda_i}{3}g_i.
\]
Since \(g_i\) is complete and \(|\nabla f_i|\) is globally bounded,
we must have \(\lambda_i=0\). Indeed, along every complete unit-speed
geodesic \(\gamma\),
\[
    \frac{d}{dt}
    \langle\nabla f_i,\dot\gamma\rangle
    =
    -\frac{\lambda_i}{3},
\]
and boundedness of \(|\nabla f_i|\) forces \(\lambda_i=0\).
Therefore,
\[
  Ric(g_i)=0,
    \qquad
    \nabla^2f_i=0.
\]
The same conclusions hold after constant rescaling:
\[
    Ric(\tilde g_i)=0.
\]

By Step 2, there exists \(i_0>0\), independent of \(i\), such that
\[
    \operatorname{inj}_{\tilde g_i}(x)\ge i_0, \text{ for every }x\in B_{\tilde g_i}(q_i,r_i-1).
\]
Together with \eqref{Scaling Riemann curvature}
\[
\sup_{B_{\tilde g_i}(q_i,r_i)}|\widetilde {Rm_i}|\leq 4.
\] 

This gives a uniform positive
harmonic radius lower bound on
\(B_{\tilde g_i}(q_i,r_i-2)\).
In harmonic coordinates, the Ricci flat equation \eqref{Ricci equation general}
 is the uniformly
elliptic system. Interior elliptic estimates and bootstrapping therefore give
\[
    |\tilde g_i|_
    {C^{k,\alpha}(B_{\tilde g_i}(x,\rho))}
    \le C(k,\alpha),\text{ for every }k\ge0, \quad 0<\alpha<1,
\]
for a fixed \(\rho>0\), uniformly in \(i\) and \(x\) in the smaller
ball. 

Following the same argument as \thmref{G2 smooth convergence}, we use the estimates on the $G_2$-structures on every fixed compact ball as shown in Proposition \ref{lift} to obtain $C^\infty$ convergence of the sequence $\tilde g_i$ to a smooth limit $G_2$-soliton (see Definition \ref{Cinfty convergence}):
    \begin{align*}
        (\tilde M_i=B_{\tilde g_i}(q_i,r_i),\tilde g_i,q_i)\longrightarrow (M_\infty,\tilde g_\infty,q_\infty),\quad S(\tilde g_i)\rightarrow 0,\quad \tilde \lambda_i\rightarrow 0.
    \end{align*}
and the smooth convergence of the $G_2$-structures 
\[
\sigma_{i,R}^\ast(\tilde\vphi_i)\rightarrow \tilde\vphi_\infty,\quad
\sigma_{i,R}^\ast(\tilde f_i)\rightarrow \tilde f_\infty.
\]

    \textbf{Step 4: contradiction.}
    Passing to the limit in \eqref{Scaling Riemann curvature}, we have
    \begin{align}
        |\widetilde {Rm_\infty}|(q_\infty)=1. \label{eq10}
    \end{align}
    
    Let $\{\Om_k\}$ be an exhaustion of $M_\infty$ such that each $\Om_k$ is compactly contained in $M_\infty$. Let $\Omega_{i,k}\subset\tilde M_i$ be the images of $\Om_k$ under the convergence embeddings.
    The $\mu$-functional is also continuous under this convergence, \ref{Continuity} in \lemref{nu properties}
\begin{align*}
\mu(\Om_k,\tilde g_\infty,\tau)=\lim_{i\rightarrow \infty}\mu(\Omega_{i,k},\tilde g_i,\tau).
\end{align*}    
By definition and monotonicity \ref{Monotonicity}, and \eqref{nu lower bound tilde}, we get
\begin{align*}
\mu(\Omega_{i,k},\tilde g_i,\tau)\geq \nu(\Omega_{i,k},\tilde g_i,\tau)\geq \nu(M_i,\tilde g_i)\geq -\eps_i.
\end{align*} 
So, taking $i\rightarrow \infty$ and then $k\rightarrow \infty$, we have
\begin{align*}
\mu(M_\infty,\tilde g_\infty,\tau)\geq 0.
\end{align*}    
Using definition and taking infimum over $\tau$, we get
    \begin{align*}
        \nu(M_\infty,\tilde g_\infty,1)\geq  \nu(M_\infty,\tilde g_\infty)=\inf_{\tau\in (0,\infty)} \mu(M_\infty,\tilde g_\infty,\tau)\geq 0.
    \end{align*}
    
    However, the rigidity statement \ref{Rigidity} in \lemref{nu properties} shows that $(M_\infty,\tilde g_\infty)$ must be isometric to the Euclidean space. Consequently,
    \begin{align*}
        |\widetilde {Rm_\infty}|(q_\infty)=0,
    \end{align*}
    which contradicts (\ref{eq10}). This completes the proof.
\end{proof}

\subsection{Proof of \thmref{Epsilon regularity}}
\begin{proof}
We begin by establishing the case $k=0$ using a contradiction argument.
Suppose there exists a sequence $(M_i, \vphi_i,g_i, f_i,p_i)$ with fixed $\lambda$ and let $B_i:=B(y_i,1)$ be the unit ball, such that
\begin{align*}
\nu(B_i, g_i,1)\geq -\eps,\quad\text{and}\quad
\sup_{B_i}| Rm_i|(\cdot)\, d_{g_i}^{2}(\cdot,\p B_i)\rightarrow \infty.
\end{align*}
Here $Rm_i$ denotes the Riemann curvature tensor of the metric $g_i$.

Following the notation introduced in Section \ref{Local rescaling procedure}, we define
\begin{align*}
Q_{i,x}:=|Rm_i|(x),\quad \Psi^2_{i,x}:=|Rm_i|(x)\, d_{g_i}^2(x,\p B_i),\quad \forall\; x\in B_i.
\end{align*}
Let $q_i$ be a point in $B(y_i,1)$ at which the function $\Psi^2_{i,x}$ attains its maximum.
Then, by our contradictory assumption, we have
\begin{align}\label{Epsilon regularity contradiction assumption}
Q_{i,q_i}\geq \Psi^2_{i,q_i}\rightarrow\infty.
\end{align}

\textbf{Step 1: construction of the rescaling sequence.}
Define the rescaled metric by $$\tilde g_i = Q_{i,q_i}\cdot g_i.$$
 We also rescale the $G_2$-structure, the soliton constant and the potential function 
\begin{align*}
\tilde\vphi_i=Q_{i,q_i}^{\frac{3}{2}}\vphi_i,\quad \tilde \lambda_i=Q_{i,q_i}^{-1}\lambda_i,\quad \tilde f_i=f_i-f_i(q_i).
\end{align*}

We can apply \lemref{local rescaling procedure R r} and obtain
\begin{align}\label{Epsilon regularity Riemannian}
    |\widetilde{Rm}(\tilde g_i)|(q_i)=1\quad\text{and}\quad
    \sup_{B_{\tilde g_i}(q_i,r_i)} |\widetilde{Rm}(\tilde g_i)|
    \le \frac{1}{\bigl(1-r_i \Psi_{i,q_i}^{-1}\bigr)^2}.
\end{align}

The radius $r_i$ is chosen as 
\begin{align*}
    r_i := \frac{\Psi_{i,q_i}}{2},
\end{align*}
so Lemma \ref{local rescaling procedure R r} applies.
This choice ensures, in view of \eqref{Epsilon regularity contradiction assumption}, that the ball
$
B_{g_i}\bigl(q_i, Q_{i,q_i}^{-\frac12} r_i\bigr)
$
is contained in $B_i$ and the rescaled Riemann curvature has a uniform bound:
\begin{align*}
    \sup_{B_{\tilde g_i}(q_i,r_i)}|\widetilde{Rm}(\tilde g_i)| \le 4.
\end{align*}

\textbf{Step 2: injectivity radius.}
We now use the entropy bound to obtain a uniform lower bound on the injectivity radius of $\tilde g_i$. The entropy is preserved under rescaling and is monotone (non-decreasing) on the domain; see \ref{Scale invariance} and \ref{Monotonicity} in \lemref{nu properties}. The entropy is also monotone with respect to the terminal time. Hence
\begin{equation}\label{Epsilon regularity entropy}
\begin{split}
    \nu\bigl(B_{\tilde g_i}(q_i,r_i),\tilde g_i,r_i^2\bigr)
        &= \nu\bigl(B_{g_i}(q_i, Q_{i,q_i}^{-\frac12}r_i), g_i,Q_{i,q_i}^{-1}r_i^2\bigr)\\
    &\geq  \nu\bigl(B_{g_i}(q_i, Q_{i,q_i}^{-\frac12}r_i), g_i,1\bigr)\\
    &\geq \nu(B_i, g_i,1) 
    \geq -\eps.
\end{split}
\end{equation}
In the first inequality, we use $Q_{i,q_i}^{-1}r_i^2=d(q_i,\partial B_i)^2/4\leq 1/4$, which follows from $r_i=\Psi_{i,q_i}/2$.

From here on, the argument is identical to \textbf{Step 2} in the proof of \thmref{Gap theorem}, which yields
\begin{align*}
   \operatorname{inj}_{\tilde g_i} (x)> c(7,\eps),\quad \forall x\in B_{\tilde g_i}(q_i,r_i-1).
\end{align*}
Together with the curvature bound, it gives the harmonic radius lower bound.

\textbf{Step 3: contradiction.}
Since
\(r_i\to\infty\), for every fixed \(A<\infty\) and all sufficiently
large \(i\),
\[
    B_{\widetilde g_i}(q_i,A+1)
    \subset
    B_{\widetilde g_i}(q_i,r_i-1).
\]
By the derivative estimates \thmref{Shi}, on these domains we obtain uniform higher order bounds for the metrics:
\begin{align*}
    |\p^k \tilde g_i| \le C(k,A), \quad \forall k \ge 0 \text{ on }B_{\widetilde g_i}(q_i,A).
\end{align*}
Consequently, by the Cheeger-Gromov theorem we have $C^\infty$ convergence of $\tilde g_i$ to a smooth metric 
\begin{align*}
    (M_i,\tilde g_i,q_i) \longrightarrow (M_\infty,\tilde g_\infty,q_\infty).
\end{align*}
Additionally, Proposition \ref{lift} provides the required estimates for $\tilde\vphi_i,\tilde f_i$ and \thmref{G2 smooth convergence} further implies 
the limit is a smooth $G_2$-soliton.

For $x$ in a fixed $B_{\tilde g_i}(q_i,A)$, we have $d_{g_i}(x,y_i)\leq2$ for large $i$. Hence, from assumptions,
\[|S_i|\leq\Lambda(2),\quad |\nabla f_i|\leq F(2),\quad |\lambda_i|\leq L_\lambda.\] 
The scalar curvature for $\tilde g_i$ and the gradient of the potential function take the form
\[
    S(\tilde g_i) = Q_{i,q_i}^{-1}\cdot S(g_i),\quad |\nabla\tilde f_i|^2_{\tilde g_i}=Q_{i,q_i}^{-1}\cdot|\nabla f_i|^2_{g_i}.
\]
Thus the limit metric $\tilde g_\infty$ has zero scalar curvature, is Ricci-flat and 
\[|\nabla\tilde f_i|^2_{\tilde g_i}\rightarrow 0.\]

By \eqref{Epsilon regularity Riemannian}, $\tilde g_\infty$ is non-flat
\begin{align*}
    |\widetilde{Rm}(\tilde g_\infty)|(q_\infty) = 1.
\end{align*}
According to the Gap \thmref{Gap theorem}, we have
\begin{align*}
    \nu(M_\infty,\tilde g_\infty) < -\eps.
\end{align*}

On the other hand, let $\{\Om_k\}$ be an exhaustion of $M_\infty$. Let $\Omega_{i,k}=\sigma_{i,k}(\Omega_k)$, where $\sigma_{i,k}$ are convergence embeddings. For fixed $k,\tau$ and sufficiently large $i$, we have
$\Omega_{i,k}\subset B_{\tilde g_i}(q_i,r_i)$.
Then we use the continuity \ref{Continuity}, the monotonicity \ref{Monotonicity} in \lemref{nu properties} and \eqref{Epsilon regularity entropy}. We fix
$\tau$ and $k$ such that $\tau\leq r_i^2$ for large $i$.
\begin{align*}
\mu(\Om_{k},\tilde g_\infty,\tau)&=\lim_{i\rightarrow \infty}\mu(\Om_{i,k},\tilde g_i,\tau)\\
&\geq \lim_{i\rightarrow \infty}\nu(\Om_{i,k},\tilde g_i,\tau)\\
&\geq  \lim_{i\rightarrow \infty}\nu\bigl(B_{\tilde g_i}(q_i,r_i),\tilde g_i,r_i^2\bigr)
\geq -\eps.
\end{align*} 
Taking $k\rightarrow \infty$ gives
\begin{align*}
\mu(M_\infty,\tilde g_\infty,\tau)\geq -\eps
\end{align*}    
and using the definition of $\nu$ leads to
    \begin{align*}
 \nu(M_\infty,\tilde g_\infty)=\inf_{\tau\in (0,\infty)} \mu(M_\infty,\tilde g_\infty,\tau)\geq -\eps.
    \end{align*}

Contradiction. Thus, \thmref{Epsilon regularity} is proved.
\end{proof}

\begin{rem}
We could use the following result to simplify the proof:
every complete nonflat Ricci-flat $n$-manifold with bounded curvature has $\nu<-\eps(n)$.
\end{rem}

\subsection{Proof of \thmref{thm:conifold-compactness main}}
Recall the notion in \thmref{thm:conifold-compactness main}. We know $x$ is a point in $K$.
We rescale the ball $B_{g_i}\bigl(\sigma_{i,R}(x),2r_K\bigr)$ to unit size and apply the epsilon-regularity result \thmref{Epsilon regularity}
\begin{align*}
    \sup_{B_{g_i}\bigl(\sigma_{i,R}(x),r_K\bigr)} |\nabla^k Rm(g_i)|\, \leq C_kr_K^{-k-2} .
\end{align*}
Here the constant $C$ depends on the quantities $\Lambda(R+Cr_K),F(R+Cr_K),L_\lambda,\eps$. 
So we take a finite subcover of $K$ and obtain a uniform curvature bound on an open neighbourhood of $K$. Then we could apply \thmref{G2 smooth convergence} to complete the proof of this theorem.

\subsection{Verify the small entropy condition}\label{Verify the small entropy condition}
\begin{lemma}
\label{lem:automatic-local-entropy}
On every compact subset of \(\mathcal R\), the small local
entropy condition \eqref{small local entropy condition} required in \thmref{thm:conifold-compactness main} is automatic.

More precisely, under the assumptions in \thmref{thm:conifold-compactness main},
we choose open sets $U_0$ and $U_1$ satisfying
\[
K\Subset U_0\Subset U_1\Subset U
\Subset B_R(p_\infty)\cap\mathcal R.
\]
 Then, for every \(\eps_0>0\), there exist
\(r_K\in(0,\tfrac14)\) and \(i_K\in\mathbb N\) such that
\begin{equation}
\label{eq:ball-containment}
B_{g_i}\bigl(\sigma_{i,R}(x),4r_K\bigr)
\Subset \sigma_{i,R}(U_1),\quad \forall i\ge i_K, \quad \forall x\in\overline{U_0}
\end{equation}
and
\begin{equation}
\label{eq:automatic-entropy}
\nu^{S_{g_i}}\!\left(
B_{g_i}\bigl(\sigma_{i,R}(x),2r_K\bigr),
 g_i,(2r_K)^2
\right)
\ge -\eps_0.
\end{equation}
\end{lemma}

\begin{proof}
By pointed \(C^{1,\alpha}\)-convergence,
\begin{equation}
\label{eq:c1a-convergence}
h_i=\sigma_{i,R}^{*}g_i\rightarrow g_\infty
\quad\text{in }C^{1,\alpha}(\overline{U}).
\end{equation}
Then, for all sufficiently large \(i\),
$
\frac{1}{2} g_\infty\le h_i\le 2 g_\infty$
on $\overline{U}$.
Moreover,
\[
d_{h_i}
\bigl(\overline{U_0},U\setminus U_1\bigr)
\ge \frac{d_0}{\sqrt 2},\quad d_0:=
d_{g_\infty}
\bigl(\overline{U_0},U\setminus U_1\bigr)>0.
\]
Thus, we choose \(r_K>0\) satisfying
$
4r_K<\frac{d_0}{\sqrt 2}
$ such that
\[
d_{g_i}
\bigl(\sigma_{i,R}(\overline{U_0}),\sigma_{i,R}(U)\setminus \sigma_{i,R}(U_1)\bigr)
\ge \frac{d_0}{\sqrt 2}>4r_K,
\] which proves
\eqref{eq:ball-containment}.

Using \eqref{eq:c1a-convergence} again, we could choose a finite collection of coordinate charts covering
\(\overline{U_0}\) such that for every
\(x\in\overline{U_0}\), the ball \(B_{h_i}(x,4r_K)\) is contained in one
of these charts, and, after normalization,
\begin{equation}
\label{eq:bilipschitz}
(1-\eta)g_E\le h_i\le(1+\eta)g_E
\end{equation}
for all sufficiently large \(i\) and a fixed constant \(\eta\in(0,\tfrac14)\).

From \eqref{eq:ball-containment}, we see
\begin{align}\label{eq:ball-containment 1}
\Omega_{i,x}:=
B_{g_i}\bigl(\sigma_{i,R}(x),2r_K\bigr)\Subset\sigma_{i,R}(U_1).
\end{align}
Let \(D\Subset\Omega_{i,x}\) be a smooth domain. In
dimension seven, \eqref{eq:bilipschitz} gives
\[
|\partial D|_{g_i}
\ge(1-\eta)^3|\partial D|_{g_E},\quad
|D|_{g_i}^{6/7}
\le(1+\eta)^3|D|_{g_E}^{6/7}.
\]
Consequently,
\[
\frac{|\partial D|_{g_i}}{|D|_{g_i}^{6/7}}
\ge
\left(\frac{1-\eta}{1+\eta}\right)^3
\frac{|\partial D|_{g_E}}{|D|_{g_E}^{6/7}}.
\]
Taking the infimum over all such \(D\), we obtain the lower bound
\begin{equation}
\label{eq:isoperimetric-ratio}
\lambda_{\mathrm{iso}}(\Omega_{i,x},g_i)
\ge\lambda_\eta,
\qquad
\lambda_\eta:=
\left(\frac{1-\eta}{1+\eta}\right)^3.
\end{equation}
Notice that \(\lambda_\eta\to1\) as \(\eta\to0\).

Since $p_i\rightarrow p_\infty$ and $
U\Subset B_R(p_\infty)\cap\mathcal R
$, after increasing
\(i_K\),
\[
d_{g_i}\bigl(p_i,\sigma_{i,R}(x)\bigr)\le R+1
\qquad\text{for all }x\in\overline{U_0}.
\]
Consequently,
$
\Omega_{i,x}\subset B_{g_i}(p_i,R+2).
$
The scalar curvature assumption for
\(\mathcal M(\Lambda,F,\underline\nu)\) in Definition \ref{3 assumptions} gives
\begin{equation}
\label{eq:scalar-lower}
\inf_{\Omega_{i,x}}S_{g_i}
\ge-\Lambda(R+2).
\end{equation}

Inserting \eqref{eq:isoperimetric-ratio} and \eqref{eq:scalar-lower} into \eqref{nu functional, I and S} with \(n=7\), we obtain
\begin{equation}
\begin{split}
\nu^{S_{g_i}}\!\left(
\Omega_{i,x},g_i,(2r_K)^2
\right)
&\ge
7\log\lambda_{\mathrm{iso}}(\Omega_{i,x},g_i)
+(2r_K)^2\inf_{\Omega_{i,x}}S_{g_i}\\
&\ge
7\log\lambda_\eta
-4r_K^2\Lambda(R+2).
\label{eq:entropy-lower-final}
\end{split}
\end{equation}
First choose \(\eta>0\) sufficiently small that
\[
7\log\lambda_\eta\ge-\frac{\eps_0}{2}.
\]
Then decrease \(r_K\) once more, preserving all the preceding requirements, so that
\[
4r_K^2\Lambda(R+2)\le\frac{\eps_0}{2}.
\]
So, \eqref{eq:entropy-lower-final} proves
\eqref{eq:automatic-entropy}, uniformly for
\(x\in\overline{U_0}\) and \(i\ge i_K\).
\end{proof}

\subsection{Proof of \thmref{thm:smooth-convergence-regular-set}}
Fix a compact set \(K\Subset\mathcal R\). Choose \(R>0\) and open sets
\[
K\Subset U_0\Subset U_1\Subset U
\Subset B_R(p_\infty)\cap\mathcal R.
\]
Let \(\eps>0\) be the small constant in the small entropy conifold compactness theorem, i.e. \thmref{thm:conifold-compactness main}. Applying
Lemma~\ref{lem:automatic-local-entropy} with
\(\eps_0=\eps\), we obtain \(r_K\in(0,\tfrac14)\) such
that, for all sufficiently large \(i\) and every
\(x\in\overline{U_0}\),
\begin{equation}
\label{eq:entropy-for-regularity}
\nu^{S_{g_i}}\!\left(
B_{g_i}\bigl(\sigma_{i,R}(x),2r_K\bigr),
 g_i,(2r_K)^2
\right)
\ge-\eps.
\end{equation}
This verifies the condition in \thmref{thm:conifold-compactness main} and proves the theorem.


\subsection{Compactness of compact solitons}
We present an application of the epsilon-regularity theorem to compact $G_2$-solitons.
Let $0<\lambda_{\mathrm{iso}} \leq 1$ be a positive constant such that
\begin{align*}
7 \log \lambda_{\mathrm{iso}} \geq -\frac{\eps}{2}.
\end{align*}
Here $\eps$ is the constant that appears in the Gap \thmref{Gap theorem}.
Let $$\mathcal M(\lambda_{\mathrm{iso}},r_0,\sigma,D,\lambda)$$ denote the moduli space of $G_2$-solitons satisfying the following properties:
\begin{itemize}
    \item $\displaystyle \inf_{x\in M} \mathbf I(B(x,r_0),g)\geq \lambda_{\mathrm{iso}} \cdot \mathbf I(\mathbb R^{7})$,
    \item $\displaystyle -S, |\nabla f|\leq \sigma$,
    \item $\displaystyle \operatorname{diam}(M,g)\leq D$,
    \item $\lambda$ is a fixed constant.
\end{itemize}

\begin{thm}
    The space $\mathcal M(\lambda_{\mathrm{iso}},r_0,\sigma,D,\lambda)$ is precompact in the pointed $C^\infty$ topology. Every pointed smooth limit is again a compact gradient closed
\(G_2\)-soliton.
\end{thm}

\begin{proof}
Define
\begin{equation}\label{eq:rho-zero-compact}
    \rho_0
    :=
    \min\left\{
        r_0,\,
        1,\,
        \sqrt{\frac{\epsilon}{2\sigma}}
    \right\},
\end{equation}
where the third entry is omitted when \(\sigma=0\). Thus
\[
    \sigma\rho_0^2
    \leq
    \frac{\epsilon}{2}.
\]

Fix \(x\in M\). Since
\[
    B_g(x,\rho_0)\subset B_g(x,r_0),
\]
monotonicity of the isoperimetric constant under restriction of the
domain gives
\[
    \mathbf I(B_g(x,\rho_0),g)
    \geq
    \mathbf I(B_g(x,r_0),g)
    \geq
    \lambda_{\mathrm{iso}}\,
    \mathbf I(\mathbb R^7).
\]
Indeed, every relatively compact test domain in \(B_g(x,\rho_0)\)
is also an admissible test domain in \(B_g(x,r_0)\).

    Using the relation between the $\nu$-functional and the isoperimetric constant from \eqref{nu functional, I and S}, we obtain
\begin{align*}
    \nu
    \bigl(B_g(x,\rho_0),g,\rho_0^2\bigr)
    &\geq
   7\log\lambda_{\mathrm{iso}}
    +
    \rho_0^2
    \inf_{B_g(x,\rho_0)}S_g
    \notag\\
    &\geq
   7\log\lambda_{\mathrm{iso}}
    -
    \sigma\rho_0^2
    \notag\\
    &\geq
    -\frac{\epsilon}{2}
    -
    \frac{\epsilon}{2}
    =
    -\epsilon.
\end{align*}

Since $M$ is compact, we can cover it by finitely many balls.  For each $x\in M$, we apply the rescaled local theorem to $B(x,\rho_0)$. If the curvature vanishes there, the conclusion is automatic.

Using a uniform gradient bound and the normalisation $f(p)=0$ and the soliton constant given in the assumptions, we apply the epsilon-regularity result \thmref{Epsilon regularity} to $(\rho_0^{-3}\vphi, \rho_0^{-2}g, f, \rho_0^2 \lambda)$ and scale the resulting derivative bounds back, to obtain
\begin{equation}\label{eq:compact-global-curvature}
    \sup_M|\nabla^k Rm|
    \leq
    C_k\rho_0^{-k-2},
    \qquad
    \text{for every }k\geq0.
\end{equation}

We next prove a uniform injectivity-radius lower bound. Applying the
 monotonicity \ref{Monotonicity}, we get
\[
    \nu(B_g(x,s),g,s^2)
    \geq-\epsilon,\quad 0<s\leq\frac{\rho_0}{4}.
\]
The local noncollapsing \lemref{noncollapse} therefore supplies
\[
  |B_s(x)|
    \geq
    \kappa(\epsilon) s^7.
\]

Combining this volume lower bound with the curvature bound \eqref{eq:compact-global-curvature},
the Cheeger-Gromov-Taylor injectivity radius estimate gives
\begin{equation}\label{eq:compact-injectivity}
   \operatorname{inj}(M,g)
    \geq
    i_0(7,C_0,\kappa)\rho_0.
\end{equation} 

The bounds
\eqref{eq:compact-global-curvature},
\eqref{eq:compact-injectivity}, and
 the diameter bound
allow us to apply the smooth Cheeger-Gromov compactness theorem.
After passing to a subsequence, there is a compact pointed smooth limit
manifold
\[
    (M_\infty,g_\infty,p_\infty).
\]

Finally, the estimates in Proposition \ref{lift}
give uniform bounds for all derivatives of
$
    \varphi_i, f_i, T_i.
$
Following the same argument as shown in \thmref{G2 smooth convergence}, passing to the limit in the closed
\(G_2\)-soliton equations shows that
\[
    (M_\infty,\varphi_\infty,g_\infty,f_\infty,p_\infty)
\]
is again a compact gradient closed \(G_2\)-soliton.
\end{proof}

\begin{rem}
The epsilon-regularity and gap theorems extend analogous results \cite{arXiv:2010.09981}*{Theorem 3.5 and Theorem 3.6} established for Einstein manifolds.
\end{rem}

    \section{Smooth convergence on uniform \texorpdfstring{$L^{7/2}$}{L^{7/2}} curvature bounds}
    \label{Smooth convergence on uniform curvature bounds}
    In this section, we prove the following theorem by following Anderson's argument \cite{MR999661}
\begin{thm}[Pointed \(C^\infty\) compactness under \(L^{7/2}\)-curvature bounds]\label{half energy convergence}
Let $\left( M_i, \vphi_i, g_i, f_i,p_i\right)$ be a sequence in
\[\mathcal M(\Lambda,F,\underline{\nu})\]
with a uniform $L^{\frac{7}{2}}$ bound on the Riemann curvature:
\begin{equation}
\int_{B_{g_i}(p_i,R)}\left|Rm_{g_i} \right|_{g_i}^{\frac{7}{2}} dv_{f_i}\leq E(R)<\infty .
\end{equation}

Then a subsequence of $\left( M_i, \vphi_i, g_i, f_i, p_i\right)$ converges, in the pointed $C^\infty$ sense, to a $G_2$-soliton without any singularities.
\end{thm}

\subsection{Epsilon-regularity under an \texorpdfstring{$L^{7/2}$}{L^{7/2}} curvature bound}
We set
\begin{align*}
	u:=|Rm|.
\end{align*}
Recall that we are assuming the potential function $f$ fulfills the gradient bound \ref{f condition} given in Definition \ref{3 assumptions}, that is,
\begin{align}\label{condition gradient f}
	|\nabla f|\leq F(R) \text{ in } B_{R}.
\end{align}
Choose $B_r(x)\subset B_{R}(p)$.
\begin{thm}[\(L^{7/2}\)-energy epsilon-regularity]\label{Laplace4}
	Let $(M,\vphi,g,f,p)$ be a gradient $G_2$-soliton obeying \eqref{condition gradient f}. Suppose that
	\begin{equation}\label{M2.11-2}
		\|u\|_{\frac{7}{2};B_{r},f}^{\frac{7}{2}}=\int_{B_{r}} |Rm|^{\frac{7}{2}}\, dv_f \leq \epsilon\leq 1
	\end{equation}
	for some sufficiently small $\epsilon$, where $dv_f = \exp(-f)dv$ and $r<1$. Then
	\begin{equation}\label{M2.12}
		\|u\|_{\infty, B_{r/2}} \leq \frac{C}{r^2}\,\delta,\quad \delta = \epsilon^{\frac{2}{7} }.
	\end{equation}
	The constant $C$ depends on the Sobolev constant $C_s$, the entropy lower bound $\underline{\nu}$, the local bound $F$ for the potential function, the soliton constant $\lambda$ and the basepoint radius $R$.
\end{thm}

We use \lemref{lem:tau Rm 1} together with Proposition \ref{lem:tau Rm k} to derive the following differential inequality.
\begin{lemma}[Differential inequality for $u$]
	Under the assumptions of \thmref{Laplace4}, we have
	\begin{align}   \label{M1.2}
		u\tri_f u \geq -C u^2(u+F_\lambda)-\frac{1}{4}\left|\nabla u \right|^2,\quad 	F_\lambda=|F(R)|^2+|\lambda|.
	\end{align}
\end{lemma}

\begin{proof}
	From the curvature equation \eqref{Rm brief} in \lemref{lem:tau Rm 1}, together with Lemma \ref{lem:tau Rm 2}, it follows that
	\begin{align}\label{M1.1}
		\left|\tri Rm \right| \leq C\left( u+F^2\right)^{\frac{1}{2}} \left|\nabla Rm \right|+Cu^2 +\frac{2}{3}|\lambda| u.
	\end{align}
	Using the identity
	\begin{align*}
		\langle\tri Rm, Rm\rangle+|\nabla Rm|^2
		=\frac{1}{2}\tri |Rm|^2
		= |Rm| \tri|Rm|+|\nabla|Rm||^2,
	\end{align*}
	we deduce
	\begin{align*}
		u \tri u & = \langle\tri Rm, Rm\rangle+|\nabla Rm|^2-|\nabla u|^2 \\
		& \geq -|\tri Rm|u+|\nabla Rm|^2-|\nabla u |^2.
	\end{align*}
	Substituting \eqref{M1.1} into this inequality gives
	\begin{align}
		u \tri u \geq -Cu(u+F^2)^{\frac{1}{2}}\left|\nabla Rm \right|-Cu^3-C|\lambda| u^2+|\nabla Rm|^2-|\nabla u|^2.  \label{M1.3}
	\end{align}
	
	By the elementary inequality, we obtain
	\begin{align*}
		-Cu(u+F^2)^{\frac{1}{2}}\left|\nabla Rm \right|\geq -Cu^2(u+F^2)-\frac{1}{8}|\nabla Rm|^2.
	\end{align*}
	Inserting this estimate into \eqref{M1.3},
	and adding
	\begin{align*}
		-u\left\langle \nabla f,\nabla u\right\rangle \geq -2u^2|\nabla f|^2-\frac{1}{8}|\nabla u|^2 \geq -2u^2F^2-\frac{1}{8}|\nabla u|^2,
	\end{align*}
	we arrive at
	\begin{align*}
		&u \tri_f u = u \tri u-u\left\langle \nabla f,\nabla u\right\rangle \\
		& \geq -Cu^2(u+F^2)-\frac{1}{8}|\nabla Rm|^2\\
		&-2u^2F^2-\frac{1}{8}|\nabla u|^2\\
		&-Cu^3-C|\lambda| u^2+|\nabla Rm|^2-|\nabla u|^2 \\ \nonumber
		& \geq -Cu^2(u+F^2+|\lambda| )+\frac{7}{8}|\nabla Rm|^2-\frac{9}{8}|\nabla u|^2.
	\end{align*}
	
	Applying the Kato inequality $|\nabla|Rm||\leq |\nabla Rm|$ to the middle term, we get
	\begin{align*}
		u \tri_f u \geq -Cu^2(u+F^2+|\lambda| )-\frac{1}{4}\left|\nabla u\right|^2.
	\end{align*}
The lemma is proved by using $u_\eps=(|Rm|^2+\eps)^{1/2}$ instead and taking $\eps\rightarrow 0$.
\end{proof}

\begin{lemma}[Integral inequality]\label{Laplace1}
	Let $p\geq \frac{7}{2}$. Under the assumptions of \thmref{Laplace4}, and for a cut-off function $\eta$ supported in a domain $\Om$ (to be specified later), the following estimate holds:
	\begin{equation*}
		\int_{\Omega}\left|\nabla( \eta u^{\frac{p}{2}})\right|^2dv_f
		\leq \frac{Cp}{2}\int_{\Omega} \eta^2 u^{p+1}dv_f
		+\frac{CF_\lambda p}{2}\int_{\Omega} \eta^2 u^{p}dv_f
		+\int_{\Omega}u^{p}\left|\nabla\eta\right|^2dv_f.
	\end{equation*}
\end{lemma}

\begin{proof}
Multiplying \eqref{M1.2} by $\eta^2 u^{p-2}$ gives
	\begin{equation}\label{M2.2}
		\begin{aligned}
			\eta^2 u^{p-1}(-\tri_f u) &\leq \eta^2 u^{p-2}\left(Cu^2(u+F_\lambda)+\frac{1}{4}\left|\nabla u \right|^2 \right) \\
			&\leq C \eta^2 u^{p+1}+CF_\lambda \eta^2 u^{p}+\frac{1}{4}\eta^2 u^{p-2}\left|\nabla u \right|^2 .
		\end{aligned}
	\end{equation}
	Integrating \eqref{M2.2}, we deduce
	\begin{equation}\label{M2.6}
		\begin{aligned}
			&\int_{\Omega} \left\langle \nabla u, \nabla (\eta^2 u^{p-1})\right\rangle \exp(-f)dv
			=\int_{\Omega} \eta^2 u^{p-1}(-\tri_f u)\exp(-f)dv \\
			&\leq C\int_{\Omega} \eta^2 u^{p+1}\exp(-f)dv
			+CF_\lambda\int_{\Omega} \eta^2 u^{p}\exp(-f)dv
			+\frac{1}{4}\int_{\Omega} \eta^2 u^{p-2}\left|\nabla u \right|^2\exp(-f)dv.
		\end{aligned}
	\end{equation}

	A direct computation yields
	\begin{align}
		\left|\nabla( \eta u^{\frac{p}{2}} ) \right|^2
		=\frac{p^2}{4}\eta^2 u^{p-2}\left|\nabla u \right|^2
		+p\eta u^{p-1}\left\langle \nabla u,\nabla \eta\right\rangle
		+u^{p}\left|\nabla\eta\right|^2   . \label{M2.3}
	\end{align}
	Furthermore,
	\begin{align}
		\left\langle \nabla u, \nabla (\eta^2 u^{p-1})\right\rangle
		=(p-1)\eta^2 u^{p-2}\left|\nabla u \right|^2
		+2\eta u^{p-1}\left\langle \nabla u,\nabla \eta\right\rangle   . \label{M2.4}
	\end{align}
	Combining \eqref{M2.3} with \eqref{M2.4}, we obtain
	\begin{equation}\label{M2.5}
		\begin{aligned}
			&\left|\nabla( \eta u^{\frac{p}{2}})\right|^2-\frac{p}{2}\left\langle \nabla u, \nabla (\eta^2 u^{p-1})\right\rangle  \\
			=&\left( \frac{p^2}{4}-\frac{p(p-1)}{2}\right) \eta^2 u^{p-2}\left|\nabla u \right|^2+u^{p}\left|\nabla\eta\right|^2 \\
			=&\left(\frac{p}{2}-\frac{p^2}{4}\right) \eta^2 u^{p-2}\left|\nabla u \right|^2+u^{p}\left|\nabla\eta\right|^2.
		\end{aligned}
	\end{equation}

	Integrating \eqref{M2.5} and inserting \eqref{M2.6}, we arrive at
	\begin{equation}\label{M2.7}
		\begin{aligned}
			&\int_{\Omega}\left|\nabla( \eta u^{\frac{p}{2}})\right|^2\exp(-F)dv
			=\frac{p}{2}\int_{\Omega}\langle\nabla u, \nabla (\eta^2 u^{p-1})\rangle \exp(-F)dv\\
			&\quad+\left(\frac{p}{2}-\frac{p^2}{4}\right)\int_{\Omega} \eta^2 u^{p-2}\left|\nabla u \right|^2\exp(-F)dv
			+\int_{\Omega}u^{p}\left|\nabla\eta\right|^2\exp(-F)dv \\
			&\leq  \frac{Cp}{2}\int_{\Omega} \eta^2 u^{p+1}\exp(-F)dv
			+\frac{CF_\lambda p}{2}\int_{\Omega} \eta^2 u^{p}\exp(-F)dv
			+\int_{\Omega}u^{p}\left|\nabla\eta\right|^2\exp(-F) dv\\
			&\quad+\left(\frac{p}{2}-\frac{p^2}{4}+\frac{p}{8}\right)\int_{\Omega} \eta^2 u^{p-2}\left|\nabla u \right|^2\exp(-F)dv.
		\end{aligned}
	\end{equation}
	
	Since $p\geq \frac{7}{2}$, it follows that
	\[
	\frac{p}{2}-\frac{p^2}{4}+\frac{p}{8}=\frac{p}{2}\left(\frac{5}{4}-\frac{p}{2} \right) <0.
	\]
	Hence the coefficient of the last integral in \eqref{M2.7} is nonpositive, and this term can be discarded. Consequently,
	\begin{align*}
		&\int_{\Omega}\left|\nabla( \eta u^{\frac{p}{2}})\right|^2\exp(-F)dv\\
		&\leq \frac{Cp}{2}\int_{\Omega} \eta^2 u^{p+1}\exp(-F)dv
		+\frac{CF_\lambda p}{2}\int_{\Omega} \eta^2 u^{p}\exp(-F)dv
		+\int_{\Omega}u^{p}\left|\nabla\eta\right|^2\exp(-F)dv,
	\end{align*}
	which concludes the proof of \lemref{Laplace1}.
\end{proof}
\subsection{Proof of \thmref{Laplace4}}
For each $k$, we introduce a decreasing sequence of radii and an increasing sequence of exponents given by
\[
r_1 = r,\quad r_k := \left(\frac{1}{2} + \frac{1}{2^k}\right) r,\quad p_1=\frac{7}{2},\quad p_k = \left(\frac{7}{5}\right)^{k-1} p_1.
\]
Obviously $r_k\rightarrow \frac{r}{2}$ as $k\rightarrow \infty$. 

Next, we choose a cut-off function $\eta_k$ satisfying $\eta_k=1$ on $B_{r_{k+1}}$, $\eta_k=0$ outside $B_{r_k}$, and
\begin{equation}\label{M3.1}
	\left|\nabla\eta_k\right|^2\leq \frac{C}{(r_k-r_{k+1})^2}=\frac{C2^{2k}}{r^2}.
\end{equation}
For later use, set
\begin{equation}\label{Ik}
	I_k=\left[ \int_{B_{r_k}}\left( \eta_k u^{\frac{p_k}{2}}\right)^{\frac{2n}{n-2}}\exp(-F)dv\right]^{\frac{n-2}{n}}.
\end{equation}
Since $n=7$, we get $\frac{n}{n-2}=\frac{7}{5}$. We denote $V_f(s)=\int_{B_s(x)}\exp(-F)dv$.

Applying the Sobolev inequality in Lemma \ref{lem:Sob} to $h=\eta_k u^{\frac{p_k}{2}}$ and then using \lemref{Laplace1}, we obtain
\begin{equation}\label{M2 iteration} 
	\begin{aligned}
		I_k
		&\leq \frac{C_s r_k^2}{V_{f}(r_k)^{2/n}}\int_{B_{r_k}}\left|\nabla( \eta_k u^{\frac{p_k}{2}})\right|^2dv_f
		+\frac{C_s}{V_{f}(r_k)^{2/n}}\int_{B_{r_k}}\eta_k^2 u^{p_k}dv_f \\
		&\leq  \frac{C_s r_k^2}{V_{f}(r_k)^{2/n}} \times 
		\bigg(Cp_k\int_{B_{r_k}} \eta_k^2 u^{p_k+1}dv_f\\
		&\qquad\qquad\qquad\qquad +Cp_kF_\lambda \int_{B_{r_k}} \eta_k^2 u^{p_k}dv_f
		+\int_{B_{r_k}}u^{p_k}\left|\nabla\eta_k\right|^2dv_f \bigg)\\
		&\quad+\frac{C_s}{V_{f}(r_k)^{2/n}}\int_{B_{r_k}}\eta_k^2 u^{p_k}dv_f . 
	\end{aligned}
\end{equation}

\textbf{Step 1: Estimation of $I_1$.}
Taking $k=1$ in the inequality \eqref{M2 iteration} and using the small energy assumption \eqref{M2.11-2}, we get $V_{f}(r_1)=V_{f}(r)$ and 
\begin{align*}
	I_1
	=&\left[ \int_{B_{r}}\left( \eta_1 u^{\frac{p_1}{2}}\right)^{\frac{2n}{n-2}}dv_f\right]^{\frac{n-2}{n}}\\
	\leq & \frac{C_sr^2}{V_{f}(r)^{2/n}}\left(C\int_{B_{r}} \eta_1^2 u^{p_1+1}dv_f+CF_\lambda\int_{B_{r}} \eta_1^2 u^{p_1}dv_f
	+\int_{B_{r}}u^{p_1}Cr^{-2}dv_f \right)\\
	&+\frac{C_s}{V_{f}(r)^{2/n}}\epsilon  \\
	\leq &\frac{CC_sr^2}{V_{f}(r)^{2/n}}\int_{B_{r}}\eta_1^2 u^{p_1+1}dv_f+\frac{CC_sr^2}{V_{f}(r)^{2/n}}\left(F_\lambda+r^{-2}\right)\epsilon . 
\end{align*}
We now apply H\"{o}lder's inequality to estimate the first integral:
\begin{align*}
	\int_{B_{r}}\eta_1^2 u^{p_1+1}dv_f
	\leq & \int_{B_{r}}\eta_1^2 u^{p_1}\exp(-\frac{n-2}{n}f)\cdot u\exp(-\frac{2}{n}f) dv\\
	\leq & \left(\int_{B_{r}}\eta_1^{\frac{2n}{n-2}} u^{\frac{p_1n}{n-2}}dv_f\right)^{\frac{n-2}{n}}
	\left( \int_{B_{r}}u^{\frac{n}{2}}dv_f\right)^{\frac{2}{n}}
	\leq I_1\epsilon^{\frac{2}{n}}.
\end{align*}
Therefore,
\begin{equation}\label{M3.5}
	I_1\leq \frac{CC_sr^2}{V_{f}(r)^{2/n}}I_1\epsilon^{\frac{2}{n}}+\frac{CC_sr^2}{V_{f}(r)^{2/n}}\left(F_\lambda+r^{-2}\right)\epsilon.
\end{equation}

\textbf{Step 2: Non-collapsing.} 
To further control the leading term on the right-hand side, we use the non-collapsing estimate from \lemref{noncollapse}
\[
V(r)\geq \kappa_1 r^7,
\]
where the constant $\kappa_1$ depends only on the lower bound $\underline{\nu}$ of Perelman's functional. 

The uniform bound on $|\nabla f|$ then implies a corresponding weighted volume.
Taking $q=x$ in \eqref{weighted volume volume comparison}, we get
\begin{align*}
|B_r(x)|\leq
    \exp(f(x)+F(R)r)|B_r(x)|_f.
\end{align*}
The bound for $f$ is obtained directly along minimizing geodesics from $p$
\begin{align*}
 f(y)\leq F(R)R,\quad \forall y
\in B_R(p).
\end{align*}
So, we have the weighted non-collapsing estimate
\begin{align*}
V_f(r)
&\geq \exp(-2F(R)R)V(r)\geq \exp(-2F(R)R) \kappa_1 r^7\\
&= \kappa_2 r^7,
\end{align*}
where $\kappa_2$ depends on $\underline{\nu}$, $F$ and $R$. Equivalently, this can be written as
\[
\frac{r^2}{V_{f}(r)^{2/7}}\leq \kappa_2^{-\frac{2}{7}}.
\]

Let $A(C_s,\underline{\nu},F)=2CC_s\kappa_2^{-\frac{2}{7}}$ which only depends on $C_s$, $\underline{\nu}$ and $F$. We can choose $\epsilon>0$ sufficiently small so that $A\epsilon^{\frac{2}{7}}\leq \frac{1}{2}$ and
\begin{align}\label{choice of epsilon}
	\frac{CC_sr^2}{V_{f}(r)^{2/7}}\epsilon^{\frac{2}{7}}\leq CC_s\kappa_2^{-\frac{2}{n}}\epsilon^{\frac{2}{7}}=\frac{1}{2} A\epsilon^{\frac{2}{7}}\leq \frac{1}{2}.
\end{align}

Plugging these estimates and $r<1$ into \eqref{M3.5}, we conclude that
\begin{equation}\label{M3.6}
	I_1\leq C(C_s,\underline{\nu},F_\lambda)r^{-2}\epsilon.
\end{equation}

\textbf{Step 3: Estimation of $I_k$.} We now return to obtain an upper bound for $I_k$ in \eqref{M2 iteration} by estimating the first term and the remaining terms separately. Apart from the first term, the other terms in \eqref{M2 iteration} can be bounded as follows:
\begin{equation}\label{M2 iteration 2} 
	\begin{aligned}
		& \frac{C_s r_k^2}{V_{f}(r_k)^{2/7}}\left(Cp_kF_\lambda\int_{B_{r_k}} \eta_k^2 u^{p_k}dv_f
		+\int_{B_{r_k}}u^{p_k}\left|\nabla\eta_k\right|^2dv_f\right)\\
		&		+\frac{C_s}{V_{f}(r_k)^{2/7}}\int_{B_{r_k}}\eta_k^2 u^{p_k}dv_f\\
		&\leq  \frac{CC_s r^2}{V_{f}(r_k)^{2/7}}\left(p_kF_\lambda+\frac{2^{2k}}{r^2}\right)\int_{B_{r_k}}u^{p_k}dv_f.
	\end{aligned}
\end{equation}

We bound the first term on the right-hand side of \eqref{M2 iteration} by applying H\"older's inequality.
Set
\[
q=\frac{n^2}{2(n-2)}=\frac{n}{n-2}p_1=p_2,\quad 
q'=\frac{n^2}{n^2-2n+4}<\frac{n}{n-2},\quad
\frac{1}{q}+\frac{1}{q'}=1.
\]
Then
\begin{equation}\label{M2 iteration k} 
	\begin{aligned}
		\int_{B_{r_k}}\eta_k^2 u^{p_{k}+1}dv_f
		&= \int_{B_{r_k}}\eta_k^2 u^{p_{k}}\exp(-\frac{1}{q'}f)\cdot u\exp(-\frac{1}{q}f)dv \\
		&\leq \left(\int_{B_{r_k}}\eta_k^{2q'} u^{p_kq'}dv_f \right)^{\frac{1}{q'}}
		\left( \int_{B_{r_k}}u^{q}dv_f\right)^{\frac{1}{q}}.
	\end{aligned}
\end{equation}

By the relation $q=\frac{n}{n-2}p_1$ and \eqref{M3.6}, the second factor in \eqref{M2 iteration k} satisfies
\begin{equation}\label{Step1}
	\left( \int_{B_{r_k}}u^{q}dv_f\right)^{\frac{1}{q}}
	\leq \left( \int_{B_{r_2}}u^{q}dv_f\right)^{\frac{1}{q}}
	\leq I_1^{\frac{1}{p_1}}
	\leq 
	\left[ C(C_s,\underline{\nu},F_\lambda)r^{-2}\epsilon\right]^{\frac{2}{n}}.
\end{equation}

We further decompose
\[
q'=\frac{n^2}{n^2-2n+4}
=\frac{n^2-2n}{n^2-2n+4}+\frac{2n}{n^2-2n+4},
\quad
1=\frac{(n-2)^2}{n^2-2n+4}+\frac{2n}{n^2-2n+4}.
\]
Applying H\"older's inequality to the first factor in \eqref{M2 iteration k} gives
\begin{align*}
	&\int_{B_{r_k}}\left( \eta_k^{2} u^{p_k}\right)^{q'} \exp(-F) dv\\
	&= \int_{B_{r_k}}\left( \eta_k^{2} u^{p_k}\right)^{\frac{n(n-2)}{n^2-2n+4}} \exp(\frac{(n-2)^2}{n^2-2n+4}(-f))
	\left( \eta_k^{2} u^{p_k}\right)^{\frac{2n}{n^2-2n+4}} \exp(\frac{2n}{n^2-2n+4}(-f))dv\\
	&\leq \left[ \int_{B_{r_k}}\left( \eta_k^{2} u^{p_k}\right)^{\frac{n}{n-2}}dv_f\right]^{\frac{(n-2)^2}{n^2-2n+4}}
	\left(\int_{B_{r_k}}\eta_k^{2} u^{p_k}dv_f\right)^{\frac{2n}{n^2-2n+4}} .
\end{align*}
Hence
\begin{align*}
	\left[ \int_{B_{r_k}}\left( \eta_k^{2} u^{p_k}\right)^{q'} dv_f\right]^{\frac{1}{q'}} 
	&\leq \left[\int_{B_{r_k}}\left( \eta_k^{2} u^{p_k}\right)^{\frac{n}{n-2}}dv_f \right]^{\frac{(n-2)^2}{n^2}}
	\left(\int_{B_{r_k}}\eta_k^{2} u^{p_k}dv_f\right)^{\frac{2}{n}} \\
	&= \left( I_{k}\right)^{\frac{n-2}{n}}\left( \int_{B_{r_k}}\eta_k^{2} u^{p_k}dv_f\right)^{\frac{2}{n}}  .
\end{align*}

By inserting these estimates into \eqref{M2 iteration k}, choosing a fixed constant $A\geq 1$, and applying Young's inequality $x^ay^{1-a}\leq ax+(1-a)y\leq x+y$, we obtain the following upper bound for the first term in \eqref{M2 iteration}:
\begin{equation}\label{M2 iteration 1} 
	\begin{aligned}
		&p_k\int_{B_{r_k}}\eta_k^2 u^{p_{k}+1}dv_f\\
		&\leq p_k \left( I_{k}\right)^{\frac{n-2}{n}}\left( \int_{B_{r_k}}\eta_k^{2} u^{p_k}dv_f\right)^{\frac{2}{n}} 
		\left[ C(C_s,\underline{\nu},F_\lambda)r^{-2} \epsilon\right]^{\frac{2}{n}}\\
		&\leq \left( \frac{1}{A}I_{k}\right)^{\frac{n-2}{n}}\left[ (Ap_k)^{\frac{n}{2}}C(C_s,\underline{\nu},F_\lambda)r^{-2}\epsilon\int_{B_{r_k}}\eta_k^{2} u^{p_k}dv_f\right]^{\frac{2}{n}}\\
		&\leq \frac{1}{A}I_{k}+(Ap_k)^{\frac{n}{2}}C(C_s,\underline{\nu},F_\lambda)r^{-2}\epsilon\int_{B_{r_k}}\eta_k^{2} u^{p_k}dv_f.
	\end{aligned}
\end{equation}

Substituting \eqref{M2 iteration 1} together with \eqref{M2 iteration 2} into \eqref{M2 iteration}, we arrive at
\begin{align*}
	I_k	
	&\leq  \frac{CC_s r^2}{V_{f}(r_k)^{2/n}}\left[ \frac{1}{A}I_{k}+(Ap_k)^{n/2}C(C_s,\underline{\nu},F_\lambda)r^{-2} \epsilon\int_{B_{r_k}}\eta_k^{2} u^{p_k}dv_f\right] \\
	&\quad+\frac{CC_s r^2}{V_{f}(r_k)^{2/n}}\left[ \left(p_kF_\lambda+\frac{2^{2k}}{r^2}\right)\int_{B_{r_k}}u^{p_k}dv_f\right] .
\end{align*}

\textbf{Step 4: Iterated inequality.}
Similarly to \eqref{choice of epsilon}, we have 
\[
V_f(r_k)\geq \kappa_2 r_k^n\geq 2^{-n}\kappa_2 r^n.
\]
We also choose $A(C_s,\underline{\nu},F)=8CC_s\kappa_2^{-\frac{2}{n}}$ so that
\[
\frac{CC_s r^2}{V_{f}(r_k)^{2/n}}\leq 4CC_s\kappa_2^{-\frac{2}{n}}=2^{-1}A.
\]
A simple rearrangement then yields
\[
I_k	\leq C(C_s,\underline{\nu},F_\lambda)\left[p_k+\left( p_k^{\frac{n}{2}}+2^{2k}\right) r^{-2}\right]\int_{B_{r_k}}u^{p_k}dv_f.
\]
We know 
\[p_k^{\frac{n}{2}}=[\frac{7}{2} (\frac{7}{5} )^{(k-1)}]^{\frac{7}{2}}\leq C 2^{2k},\quad C = \frac{1}{4} \left( \frac{7}{2} \right)^{7/2}.\]
So we can simplify this inequality as
\[
I_k	\leq C(C_s,\underline{\nu},F_\lambda)\frac{2^{2k}}{r^2}\int_{B_{r_k}}u^{p_k}dv_f.
\]

Consequently, for any integer $k \geq 2$, we obtain the following iterated inequality
\begin{equation}\label{M3.10}
	\begin{aligned}
		&\left( \int_{B_{r_{k+1}}} u^{p_{k+1}}dv_f\right)^{\frac{1}{p_{k+1}}}
		\leq \left[ \int_{B_{r_k}}\left( \eta_k u^{\frac{p_k}{2}}\right)^{\frac{2n}{n-2}}dv_f\right]^{\frac{n-2}{np_k}}\\
		&		= I_k^{\frac{1}{p_{k}}}
		\leq \left[C(C_s,\underline{\nu},F_\lambda)\frac{2^{2k}}{r^2}\right] ^{\frac{1}{p_{k}}}
		\left[ \int_{B_{r_k}}u^{p_k}dv_f\right]^{\frac{1}{p_{k}}}.
	\end{aligned}
\end{equation}  

\textbf{Step 5: Iteration.}
On the closed ball $B_r$, both the volume $V_f$ and the function $f$ are bounded from above. Therefore, by the usual Moser iteration argument, we obtain
\begin{equation}\label{M3.11}
	\left\|u\right\|_{\infty, B_{r/2}}\leq \prod_{k=2}^{\infty}\left[C(C_s,\underline{\nu},F_\lambda)\frac{2^{2k}}{r^2}\right]^{\frac{1}{p_{k}}}
	\left[ \int_{B_{r_2}}u^{p_2}dv_f\right]^{\frac{1}{p_{2}}}
\end{equation}
Substituting \eqref{Step1} into \eqref{M3.11}, we have
\begin{align*}
	\left\|u\right\|_{\infty, B_{r/2}}\leq &\prod_{k=2}^{\infty}\left[C(C_s,\underline{\nu},F_\lambda)\frac{2^{2k}}{r^2}\right]^{\frac{1}{p_{k}}}
	\left[ C(C_s,\underline{\nu},F_\lambda)r^{-2}\epsilon\right]^{\frac{2}{n}} \\
	\leq & \prod_{k=1}^{\infty}\left[C(C_s,\underline{\nu},F_\lambda)\frac{2^{2k}}{r^2}\right]^{\frac{1}{p_{k}}}\epsilon^{\frac{2}{n}}.
\end{align*}
Recall that $p_1 = 7/2$ and, for $k \ge 1$, the sequence is given by $p_k = p_1\left(\frac{7}{5}\right)^{k-1}$.  
Using the identity $\sum_{k=1}^{\infty}\frac{1}{p_k}=1$ and the geometric series formula, we obtain
\[
\prod_{k=1}^{\infty}2^{\frac{2k}{p_k}}\leq C, \qquad 
\prod_{k=1}^{\infty}r^{\frac{-2}{p_k}}=r^{-2}.
\]
Consequently, we deduce
\begin{equation}\label{M3.12}
	\left\|u\right\|_{\infty, B_{r/2}}\leq C(C_s,\underline{\nu},F_\lambda)r^{-2}\epsilon^{\frac{2}{7}}.
\end{equation}
This completes the proof of \eqref{M2.12}.
 \subsection{Local higher-order estimates}\label{Shi-type estimates}
Shi-type estimates for the Laplacian flow were obtained in \cite{MR3613456}, and these in turn imply the Shi-type estimates for $G_2$-solitons stated in Section \ref{Associated flow and higher derivatives}. For our subsequent arguments, we need an explicit constant in this estimate, so we provide the proof below.
      
     \begin{thm}[Shi-type estimates]\label{Shi}
      Let $B_r(x) \subset B_R(p)$ and $0<r<1$.
      Assume that 
          	\begin{equation*}
    		\left|Rm \right|\leq \frac{C(0)}{r^2} \text{ in }B_{r/2}.
    	\end{equation*}
      Then, for any positive integer $k$, we have
        \begin{equation}\label{M2.13}
            \left|\nabla^k Rm \right| \leq \frac{C(C_s,\underline{\nu},F,\lambda,k, C(0))}{r^{k+2}}\text{ in }B_{r/4}.
        \end{equation}
            \end{thm}

Before presenting the proof of \thmref{Shi}, we first state an application of the higher order curvature estimates. The following lemma gives the higher order derivative bounds for $\varphi,\psi, T, A$, and $f$. This Lemma also leads up to the inductive method in \thmref{Shi}.       
    \begin{lemma}\label{higher-order derivatives}
       For any integer $k\geq 0$, if the Riemann curvature satisfies
        \begin{equation}\label{M2.14}
        	\left|\nabla^i Rm \right| \leq \frac{K(i)}{r^{i+2}},\quad 0\leq i\leq k,
        \end{equation}
         in the ball $B_{r}$. Then, for constants
\[
C(j)=C\bigl(j,K(0),\ldots,K(k)\bigr),
\]
depending only on the indicated quantities and the dimension, the
following estimates hold on $B_g(x,r)$:
        \begin{align}
            \label{scalevphi4}
 &|\nabla^{j} \vphi|+ |\nabla^{j} \psi|\leq C(j)r^{-j},\quad 0\leq j\leq k+2,\\
            \label{scaletau4}
 &|\nabla^{j} T|\leq C(j) r^{-j-1},\quad 0\leq j\leq k+1,\\
            \label{scaleA4}
 &           |\nabla^{j} A|+|\Na^{j} S|\leq C(j)r^{-j-2},\quad 0\leq j\leq k+1.
        \end{align}

Furthermore, the potential function of the $G_2$-soliton satisfies that
        \begin{align}\label{scaleF4}
            &|\nabla^{2} f|\leq |\lambda|+C(0)r^{-2},\quad |\nabla^{j} f|\leq C(j)r^{-j},\quad 3\leq j\leq k+2.
        \end{align}
    \end{lemma}    
    \begin{proof}
        From \lemref{lem:tau Rm 2}, the inequality \eqref{scalevphi} and the Definition \ref{defkRm2}, we apply \eqref{M2.14} to have
  \begin{align*}
  |\vphi|\leq C,\quad|\nabla\vphi|\lesssim|Rm|^{\frac{1}{2}}
  \end{align*} and 
  \begin{equation} \label{scalevphi1}
        \begin{aligned}
            &|\nabla^{j} \vphi|\lesssim \sum_{\|v_{j}\|=j } \|Rm\|_{v_{j}} \\
&            =\sum_{\|v_{j}\|=j } \prod_{0\leq i\leq j}|\nabla^i Rm|^{n_{i}}
= \sum_{\|v_{j}\|=j } |Rm|^{n_0}|\nabla Rm|^{n_1}\cdots|\nabla^{j} Rm|^{n_{j}}\\
&            \leq Cr^{-\sum_{i=0}^{j}(i+2)n_i}.
        \end{aligned}
        \end{equation}
        Recall Definition \ref{defkRm1} that
        \[j=\|v_{j}\|=\sum_{0\leq i\leq j}(i+2)n_i=\sum_{0\leq i\leq j-2}(i+2)n_i.\]
        So, we have \eqref{scalevphi4}.
 Also, we have
 \begin{align*}
 &|T|\lesssim|Rm|^{\frac{1}{2}},\quad
            |\Na T|\lesssim |Rm|\\
 &|A| \lesssim|Rm|,\quad |\Na A|\lesssim|Rm|^{\frac{3}{2}}.
\end{align*}    
     
        Similarly, from Proposition \ref{lem:tau Rm k}, we have
        \begin{align}
            \label{scalepsi2}
            & |\nabla^{j} \psi|\leq C(j)\sum_{\|v_{j}\|=j}\|Rm\|_{v_{j}}\leq Cr^{-j},\\
            \label{scaletau2}
            & |\nabla^{j} T|\leq C(j)\sum_{\|v_{j+1}\|=j +1}\|Rm\|_{v_{j}}\leq Cr^{-j-1},\\
            \label{scaleA2}
            &|\nabla^{j} A|\leq C(j)\sum_{\|v_{j+2}\|=j +2 } \|Rm\|_{v_{j}}\leq Cr^{-j-2},
        \end{align}The constants $C$ depend on $j$.
         Moreover, $S=-|T|^2=T*T$ and $A=T*T$, so $\Na^{j} S$ has the same estimates as $\Na^{j} A$.       
        
        Since $\nabla^2 f=\lambda_S g+A-Ric$ and $\lambda_S=-\frac{\lambda}{3}+\frac{S}{3}$, we have
        \begin{align}
            \label{scaleF2}
            |\nabla^{j} f|\leq \left|\nabla^{j-2} S\right|+\left| \nabla^{j-2} A \right|+\left|\nabla^{j-2}Rm\right|\leq C(j)r^{-j},\quad j\geq 3.
        \end{align}
When $j=2$, the constant also depends on $|\lambda|.$ So, we prove \eqref{scaleF4}.
    \end{proof}

\subsection{Proof of \thmref{Shi}}
    	We argue by induction on $k$. For $k=0$, the assumption gives
    	\begin{equation*}
    		\left|Rm \right|\leq \frac{C(0)}{r^2} \text{ in }B_{r/2}.
    	\end{equation*}
	Fix a final order $N$.
    	Let $R_{N}=\frac{r}{4}$ and 
    	$$R_{j}=\frac{r}{2}-\frac{jr}{4N}, \quad \forall 0\leq j\leq N.$$ 
    	Assume now that (\ref{M2.13}) holds in a slightly smaller ball $R_{k-1}$, namely
    	\begin{equation}\label{triRm7 inductive hypothesis}
    		\left|\nabla^{j} Rm \right| \leq \frac{C(j)}{r^{j+2}}\text{ in }B_{R_{k-1}},\quad \forall j\leq k-1.
    	\end{equation}
    	We then establish the corresponding estimate 
    	\begin{equation}
    		\left|\nabla^{k} Rm \right| \leq \frac{C(k)}{r^{k+2}}\text{ in }B_{R_{k}},
    	\end{equation}
    	for order $k\geq 1$ in a smaller ball $R_{k}$ where $B_{R_{k}}<B_{R_{k-1}}$.

\textbf{Differential inequality.}    	
To derive local estimates, we introduce a cutoff function $\eta$. We then calculate
\begin{equation}\label{triRm7 cutoff}
	\begin{aligned}
    		\tri \left(\eta^2 \left| \nabla^{k} Rm\right|^2\right)
		&=\eta^2\tri  \left| \nabla^{k} Rm\right|^2
    		+\left| \nabla^{k} Rm\right|^2\tri \left(\eta^2 \right) \\ 
    		&+2\left\langle \nabla\left(\eta^2 \right),\nabla \left| \nabla^{k} Rm\right|^2\right\rangle .
	\end{aligned}
\end{equation} 

\textbf{Estimate the third term in \eqref{triRm7 cutoff}.}
Applying a basic quadratic inequality together with Kato's inequality, we can estimate the third term in \eqref{triRm7 cutoff} by
\begin{equation}\label{triRm7 cutoff mixed term}
	\begin{aligned}
    		2\left\langle \nabla\left(\eta^2 \right),\nabla \left| \nabla^{k} Rm\right|^2\right\rangle 
    		=&8\eta \left| \nabla^{k} Rm\right|\left\langle \nabla\eta,\nabla\left| \nabla^{k} Rm\right|\right\rangle \\ 
    		\geq &-\frac{1}{2}\eta^2\left| \nabla^{k+1} Rm\right|^2-32\left| \nabla\eta\right|^2\left| \nabla^{k} Rm\right|^2.
	\end{aligned}
\end{equation}

\textbf{Estimate the first term in \eqref{triRm7 cutoff}.}  
The first term in \eqref{triRm7 cutoff} can be written as
    	\begin{align}\label{triRm7}
    		\frac{1}{2}\tri \left| \nabla^{k} Rm\right|^2=\left\langle \nabla^{k} Rm,\tri \nabla^{k} Rm\right\rangle
    		+\left| \nabla^{k+1} Rm\right|^2.
    	\end{align}
Thus, it remains to estimate the term $\tri \nabla^{k} Rm$.

\subsubsection{Deal with $\triangle\nabla^k Rm $}
Recall the Ricci identity
\begin{equation}\label{triRm1-2}
\triangle\nabla^k Rm=\nabla^k\tri Rm+Rm*\nabla^k Rm+\sum_{l=1}^{k-1}\nabla^{l} Rm*\nabla^{k-l} Rm.
\end{equation}
To obtain an explicit formula, we start from \eqref{Rm brief}, which yields
\begin{equation}\label{triRm1-1}
\begin{aligned}
\tri Rm&=\nabla f*\nabla Rm+ Rm\ast Rm-\frac{2}{3}\lambda\cdot Rm\\
&+ Rm\ast T*T +T* \nabla^2T+\nabla T* \nabla T.
\end{aligned}
\end{equation}
Substituting (\ref{triRm1-1}) into (\ref{triRm1-2}), we obtain the expression of $\triangle\nabla^k Rm$: 
\begin{equation}\label{triRm1-3}
\begin{aligned}
\triangle\nabla^k Rm 
&=\nabla^k\big( \nabla f*\nabla Rm+ Rm\ast Rm-\frac{2}{3}\lambda\cdot Rm\\
&+ Rm\ast T*T +T* \nabla^2T+\nabla T* \nabla T\big) \\ 
&+Rm*\nabla^k Rm+\sum_{l=1}^{k-1}\nabla^{l} Rm*\nabla^{k-l} Rm \\ 
&=\nabla^k\left( \nabla f*\nabla Rm\right)
+\nabla^k\left( Rm*T*T\right)\\
&+T*\nabla^{k+2}T+\nabla T*\nabla^{k+1}T 
+ \sum_{l=2}^{k}\nabla^{l}T*\nabla^{k+2-l}T\\
&+Rm*\nabla^k Rm
+\sum_{l=1}^{k-1}\nabla^{l} Rm*\nabla^{k-l} Rm
-\frac{2}{3}\lambda\cdot \nabla^k Rm.
\end{aligned}
\end{equation}

\textbf{Estimate the first term in \eqref{triRm1-3}.}
For the first term in \eqref{triRm1-3}, we use \eqref{scaleF4}. This gives
\begin{align*}
\nabla^{l}f\leq \frac{C(l)}{r^{l}},\quad l>2,\quad|\nabla^{2} f|\leq |\lambda|+\frac{C(0)}{r^2}.
\end{align*}
Using these estimates and the inductive assumption \eqref{triRm7 inductive hypothesis}, we obtain
\begin{equation}\label{triRm1-4}
	\begin{aligned}
& \left| \nabla^k\left( \nabla f*\nabla Rm\right)\right|  \\ 
    		\leq &\, C(k)\left( |\nabla f|\left| \nabla^{k+1}Rm\right|+ |\nabla^2 f|\left| \nabla^{k}Rm\right|
    	+\sum_{l=3}^{k+1}\left|\nabla^{l}f \right|\left|\nabla^{k+2-l}Rm \right|\right)  \\ 
    		\leq &\, C(k)\left( F \left| \nabla^{k+1}Rm\right|+[|\lambda|+\frac{C(0)}{r^2}]\left| \nabla^{k}Rm\right|+\frac{1}{r^{k+4}}\right).
 	\end{aligned}
\end{equation}

\textbf{Estimate the second term in \eqref{triRm1-3}.}
For the second through fifth terms in \eqref{triRm1-3}, we apply the inductive assumption \eqref{triRm7 inductive hypothesis} together with inequality \eqref{scaletau4}, which yields
\begin{equation}\label{triRm1-5}
    		\left|\nabla^{l} T \right| \leq \frac{C(l)}{r^{l+1}} \quad \forall\, l\leq k \quad \text{ in } B_{R_{k-1}}(p).
    	\end{equation}

Then the second term in \eqref{triRm1-3} is bounded by
\begin{equation}\label{triRm1-4-2}
	\begin{aligned}
&\left|\nabla^k\left( Rm*T*T\right)\right|\leq\\
    		\leq &\, C(k)\left( |\nabla^k Rm| |T|^2
    	+\sum_{l=0}^{k-1}\left|\nabla^{l}Rm \right|\left|\nabla^{k-l}(T*T) \right|\right)  \\ 
    		\leq &\, C(k)\left( \frac{C(0)}{r^2} \left| \nabla^{k}Rm\right|+\frac{1}{r^{k+4}}\right).
 	\end{aligned}
\end{equation}

\textbf{Estimate the third to fifth term in \eqref{triRm1-3}.}
For the higher derivatives $\nabla^{k+1}T$ and $\nabla^{k+2}T$, we recall the estimate \eqref{scaletau} and the inequality \eqref{nabla k+1 T} gives
\begin{equation}\label{triRm1-6}
\begin{aligned}	
    	&	\left| \nabla^{k+1}T\right|\leq C(k)\left( \left| \nabla^k Rm\right| +\frac{C(k)}{r^{k+2}}\right) , \\
    	&	\left| \nabla^{k+2}T\right|\leq 
		C(k)\left( \left| \nabla^{k+1} Rm\right|
    		+|Rm|^{\frac{1}{2}}\left| \nabla^{k} Rm\right| +\frac{C(k)}{r^{k+3}}\right).  
		\end{aligned}
    	\end{equation}
Together with \eqref{triRm1-5}, we have
\begin{equation}	
\label{triRm1-7 middle three}
\begin{aligned}		
&|T*\nabla^{k+2}T|+|\nabla T*\nabla^{k+1}T|
+ |\sum_{l=2}^{k}\nabla^{l}T*\nabla^{k+2-l}T|\\
&\leq C(k)\left( \frac{1}{r}\left| \nabla^{k+1} Rm\right|
    		+\frac{1}{r^{2}}\left| \nabla^{k} Rm\right| +\frac{C(k)}{r^{k+4}}\right).  
\end{aligned}
\end{equation}	

\textbf{Estimate the sixth to seventh term in \eqref{triRm1-3}.}	
\begin{equation}	
\label{triRm1-7 last two}
\begin{aligned}	
&|Rm*\nabla^k Rm|
+|\sum_{l=1}^{k-1}\nabla^{l} Rm*\nabla^{k-l} Rm|
\leq \frac{C(0)}{r^2}  |\nabla^k Rm|
+\frac{C(k)}{r^{k+4}}.
\end{aligned}
\end{equation}	
	
By substituting \eqref{triRm1-5}, \eqref{triRm1-4-2}, \eqref{triRm1-7 middle three} and \eqref{triRm1-7 last two} into \eqref{triRm1-3}, we arrive at
\begin{equation}\label{triRm1-9}
	\begin{aligned}
\left| \triangle\nabla^k Rm \right|  
\leq C(k) & \Big( [F+\frac{1}{r}]\left| \nabla^{k+1}Rm\right|\\
&+[|\lambda|+\frac{C(0)}{r^2}]\left| \nabla^{k}Rm\right|
+\frac{C(k)}{r^{k+4}}\Big).
 	\end{aligned}
\end{equation}

\textbf{Estimate $\tri \left| \nabla^{k} Rm\right|^2$.}
Plugging \eqref{triRm1-9} into \eqref{triRm7}, we have
\begin{equation}\label{triRm9}
\begin{aligned}
\frac{1}{2}\tri \left| \nabla^{k} Rm\right|^2 
&\geq \left| \nabla^{k+1} Rm\right|^2\\
&-C(k)\left| \nabla^{k} Rm\right|[F+\frac{1}{r}]\left| \nabla^{k+1}Rm\right| \\ 
& -[|\lambda|+\frac{C(0)}{r^2}]\left| \nabla^{k}Rm\right|^2
-\left| \nabla^{k}Rm\right|\frac{C(k)}{r^{k+4}}.
\end{aligned}
\end{equation}

By Cauchy-Schwarz inequality, the 2nd term is bounded by
\begin{equation}\label{triRm8}
\begin{aligned}
&-C(k)\left| \nabla^{k} Rm\right|  [F+\frac{1}{r}]\left| \nabla^{k+1}Rm\right|  \\
\geq &-\frac{1}{2}\left| \nabla^{k+1} Rm\right|^2-C(k) [F+\frac{1}{r}]^2\left| \nabla^{k} Rm\right|^2 .
\end{aligned}
\end{equation}
Also, the last term is bounded by
\begin{equation}\label{triRm10}
\begin{aligned}
&-\left| \nabla^{k} Rm\right| \frac{C(k)}{r^{k+4}}
=- \frac{\left| \nabla^{k} Rm\right|}{r} \frac{C(k)}{r^{k+3}}  
\geq -\left| \nabla^{k} Rm\right|^2\frac{1}{r^2}-\frac{C(k)}{r^{2k+6}}.
\end{aligned}
\end{equation}

Combining the three inequalities above, we have
\begin{equation}\label{triRm11}
	\begin{aligned}
\frac{1}{2}\tri \left| \nabla^{k} Rm\right|^2
&\geq  \frac{1}{2}\left| \nabla^{k+1} Rm\right|^2\\
&-C(k)\left(F_\lambda+\frac{1}{r^2} \right)\left| \nabla^{k} Rm\right|^2 
-\frac{C(k)}{r^{2k+6}}.
 	\end{aligned}
\end{equation}
Here we use the notation $F_\lambda=|F(R)|^2+|\lambda|$.

\subsubsection{Deal with $\tri \left(\eta^2 \left| \nabla^{k} Rm\right|^2\right)$}
We first derive the estimates for the cutoff function.

\textbf{Local harmonic coordinates.}
We set
\[
d_k=R_{k-1}-R_k=\frac{r}{4N},\quad \rho_k=\frac{d_k}{4}=\frac{r}{16N}.
\]
Then, for any \(x\in B_{R_{k}}(p)\),
\[
    B_{2\rho_k}(x)\subset B_{R_{k-1}}(p),
\]
the assumption of
the curvature bound on \(B_{r}(x)\) gives that
\[
\left|Rm \right|\leq \frac{C(0)}{r^2}
    \qquad\text{on }B_{r/2}(x).
\] Moreover, the
non-collapsing estimate gives
\[
    \operatorname{Vol}_g B_s(x)\geq \kappa s^7,
    \qquad 0<s\leq r,
\]
where \(\kappa=\kappa(\underline{\nu})>0\).

   	Now we should estimate the injectivity radius. From Klingenberg's theorem, we know the injectivity radius is bounded below by
    	$\min(\frac{\pi}{\sqrt{K}},l)$ where $K$ is the upper bound of sectional curvature and $2l$ is the length of the shortest geodesic loop. 
    	Since $\left|Rm \right|\leq \frac{C(\delta)}{r^2}$ in $B_r(p)$, we know \[\frac{\pi}{\sqrt{K}}\geq \frac{\pi}{\sqrt{C}}r.\] 

    	A theorem of Cheeger-Gromov-Taylor \cite{MR658471}*{Page 42} shows that the lower bound of $l$ depends on the volume lower bound $\kappa_1$ in the non-collapsing property	\[V_r\geq \kappa_1 r^n\] where $\kappa_1$ only depends on Perelman's functional $\underline{\nu}$.
	
	We then apply the
	harmonic radius estimate \cite{MR1074481}, there exist constants
	\[
	\rho=\rho(C_s,\underline{\nu},F)>0,
	\qquad
	C=C(C_s,\underline{\nu},F)<\infty,
	\]
	and harmonic coordinates
	\[
	y=(y^1,\ldots,y^7):
	U_x\longrightarrow B_{2\rho r}(0)\text{ in }\mathbb R^7,
	\qquad y(x)=0,
	\]
	such that
	$U_x\subset B_{\rho_k}(x)
	$
	and
	\begin{equation}\label{harmonic-coordinate-ellipticity}
		C^{-1}|\xi|^2
		\leq g^{ab}(y)\xi_a\xi_b
		\leq C|\xi|^2, \quad \forall \xi\in\mathbb R^7.
	\end{equation} Since the coordinates are harmonic,
	\begin{equation}\label{harmonic-coordinate-identity}
		\tri y^a=0,
		\qquad
		g^{bc}\Gamma_{bc}^{a}=0.
	\end{equation}

	\textbf{Cutoff function.}
	In $\mathbb R^7$, we choose \(\chi\in C_c^\infty(B_{3/2}(0))\) such that
	\(0\leq\chi\leq1\) and \(\chi=1\) on
	\(B_{1}(0)\), and define
	$
	\eta(y):=\chi\left(\frac{y}{\rho r}\right).
	$
	Then \(\eta=1\) near \(x\), \(\eta=0\) near
	\(\partial U_x\), and under the $C^{1,\alpha}$ harmonic coordinate 
	\begin{equation}\label{cutoff gradient}
		|\nabla\eta|\leq \frac{C(N)}{r},
		\qquad
		|\tri\eta|
		=
		|g^{ab}\partial_a\partial_b\eta|\leq \frac{C(N)}{r^2}.
	\end{equation}

\textbf{Estimate $\tri \left(\eta^2 \left| \nabla^{k} Rm\right|^2\right)$.}    
Substituting the third term estimate \eqref{triRm7 cutoff mixed term}, the first term estimate \eqref{triRm11}, and the bounds for the cutoff function \eqref{cutoff gradient} into \eqref{triRm7 cutoff}, we obtain
\begin{equation}\label{triRm12}
	\begin{aligned}
    		&\tri \left(\eta^2 \left| \nabla^{k} Rm\right|^2\right) \\ 
    		\geq & \frac{1}{2}\eta^2\left| \nabla^{k+1} Rm\right|^2-C_1\left(F_\lambda+\frac{1}{r^2} \right)\left| \nabla^{k} Rm\right|^2
    		-\frac{C(k)}{r^{2k+6}}.
 	\end{aligned}
\end{equation}
    	All the $C$'s depend on $C_s,\underline{\nu},F,k$.

	Similarly, we obtain
\begin{equation}\label{triRm13}
	\begin{aligned}
&\tri \left( \left| \nabla^{k-1} Rm\right|^2\right) \\
&\geq \frac{1}{2} \left| \nabla^{k} Rm\right|^2-C_2\left(F_\lambda+\frac{1}{r^2}\right)\left| \nabla^{k-1} Rm\right|^2
-\frac{C(k)}{r^{2k+4}}\\
&\geq \frac{1}{2} \left| \nabla^{k} Rm\right|^2
-\frac{C(k)}{r^{2k+4}}.
 	\end{aligned}
\end{equation}
In deriving the final inequality, we have applied the induction hypothesis \eqref{triRm7 inductive hypothesis}: $\left|\nabla^{k-1} Rm \right| \leq \frac{C(k-1)}{r^{k+1}}.$

\subsubsection{Derive the estimates}
At last, we apply the maximum principle to derive the estimates of $\left| \nabla^{k} Rm\right|$.

\textbf{Maximum principle.}
Defining the auxiliary function
\[
v=\eta^2 \left| \nabla^{k} Rm\right|^2+2C_1\left(F_\lambda+\frac{1}{r^2}\right)\left| \nabla^{k-1} Rm\right|^2,
\]
we add \eqref{triRm12} and \eqref{triRm13} to derive on $U_x$:
\begin{equation}\label{triRm14}
	\begin{aligned}
\tri v 
&\geq
-\frac{C(k)}{r^{2k+6}}
- C_1\left(F_\lambda+\frac{1}{r^2}\right)\frac{C(k)}{r^{2k+4}}
\geq -\frac{C(k)}{r^{2k+6}}:=-E.
 	\end{aligned}
\end{equation}

Since each coordinate function \(y^a\) is harmonic,
\[
\begin{aligned}
    \tri |y|^2
    &=
    \sum_{a=1}^7\tri\bigl((y^a)^2\bigr)   =
    2\sum_{a=1}^7|\nabla y^a|^2
    =
    2\sum_{a=1}^7g^{aa}\geq\frac{14}{C}=:c>0
\end{aligned}
\]
by \eqref{harmonic-coordinate-ellipticity}.
Define
\[
     w:=v+\beta,\quad \beta(y):=\frac{E}{c}|y|^2.
\]
Then
$
    \tri\beta\geq E,
$
and hence \eqref{triRm14} gives
\[
    \tri w\geq0
    \qquad\text{on }U_x.
\]
The maximum principle therefore yields
\[
    w(x)\leq\sup_{\partial U_x}w.
\]

Since \(y(x)=0\) and \(\eta(x)=1\), we have
$
    w(x)=v(x)
    \geq|\nabla^k Rm|^2(x).
$
On the other hand, \(\eta_x=0\) near \(\partial U_x\), and therefore
the inductive hypothesis gives
\[
\begin{aligned}
    \sup_{\partial U_x}v
    &=
    C_1\left(F_\lambda+\frac1{r^2}\right)
    \sup_{\partial U_x}
    |\nabla^{k-1} Rm|^2                       \\
    &\leq
    C\left(F_\lambda+\frac1{r^2}\right)
    \frac{C(k-1)}{r^{2k+2}}                              
    \leq
    \frac{C(k)}
         {r^{2k+4}},
\end{aligned}
\]
where we used \(r<1\). Furthermore, since
\(|y|\leq C\rho_k\) on \(U_x\),
\[
    \sup_{\partial U_x}\beta
    \leq
    CE(\rho_k)^2
    \leq
    \frac{C(k,N)}
         {r^{2k+4}}.
\]
It follows that
\[
    |\nabla^k\operatorname{Rm}|^2(x)
    \leq
    \frac{C(k,N)}
         {r^{2k+4}},\quad \forall x\in B_{r/4}(p).
\]

Assume the bounds through order $k-1$ on $B_{R_{k-1}}(p)$. From the argument above we
prove the $k$-th estimate on $B_{R_{k}}(p)$.
Thus, we prove inductively that
\[
    \sup_{B_{R_k}(p)}
    |\nabla^k Rm|
    \leq
    \frac{C(k,N)}{r^{k+2}},
\]
This completes the proof of Theorem \ref{Shi}.
    
 \subsection{Convergence}
In this section, we will prove the compactness \thmref{half energy convergence}.
Using the higher-order curvature estimates together with the volume non-collapsing condition, we can extract a subsequence that converges in the pointed smooth sense.
We adopt the approach used by Anderson \cite{MR999661}*{Section 5}; see also page 16-18 in \cite{MR2846384}, or page 12-14 in \cite{MR2323545}. 
Next, we briefly introduce the main steps and required conditions of the proof.
    
\textbf{Step 1: An upper bound on the number of singular components.}

For Ricci shrinkers, combining the uniform \(L^{\frac{n}{2}}\)-bound on the Riemann curvature tensor with the volume-doubling and non-collapsing properties, Haslhofer and M\"{u}ller \cite{MR2846384} established the existence of finitely many pairwise disjoint balls of radius \(r\), each of which carries curvature energy exceeding a fixed threshold \(\epsilon\). They then derived an upper bound on the maximal number of such disjoint balls that can be contained in any ball of radius \(R\). Consequently, the number of singular components admits a local upper bound. For a more detailed argument, see in particular p.~12 of \cite{MR2323545}.

Within the framework of $G_2$-solitons, the volume doubling property established in \lemref{Volume doubling} together with the non-collapsing property obtained in \lemref{noncollapse} are available. These two ingredients permit us to derive a local upper bound on the number of singular components by employing an argument entirely analogous to the one used in the preceding setting.

\textbf{Step 2: Convergence of regular parts.}

For Ricci shrinkers, Haslhofer and M\"{u}ller \cite{MR2846384} proved that such spaces converge in the pointed Gromov-Hausdorff sense, and that the resulting limit can be completed to a metric space by adjoining only locally finitely many points. In other words, all singularities in the limit are isolated point singularities.

In the setting of $G_2$-solitons, the epsilon regularity estimate \thmref{Laplace4} together with the Shi-type estimates \thmref{Shi} implies uniform control of the curvature tensor away from these singularities:
\[
\left|Rm\right|\leq \frac{C\delta}{r^2}, \qquad \left| \nabla^k Rm\right|\leq \frac{C}{r^{2+k}}.
\]
By \lemref{higher-order derivatives}, these curvature bounds induce corresponding bounds for the $G_2$-structures, and hence the metric bounds in Proposition \ref{lift} become applicable. A uniform lower bound on the injectivity radius is then obtained by the same argument used in \textbf{Step 2} of the proof of \thmref{Gap theorem}. Consequently, by the same reasoning, the sequence of $G_2$-solitons
\[
\left( M_i, \varphi_i, g_i, f_i,p_i\right)
\]
converges smoothly to a limit
\[
\left( D, \varphi_{\infty}, g_{\infty}, f_{\infty},p_\infty\right),
\]
and this limit satisfies the $G_2$-soliton equation, c.f. \thmref{G2 smooth convergence}.

Let $M_{\infty}$ denote the metric completion of $D$. Following 
Anderson \cite{MR999661}*{p.~479} or Haslhofer and M\"{u}ller \cite{MR2846384}, one obtains
\[
M_{\infty}=D\cup \left\{q_l\right\},
\]
where the set $\left\{q_l\right\}$ consists of locally finitely many points. These points $\left\{q_l\right\}$ are referred to as the curvature singularities of $M_{\infty}$.

\textbf{Step 3: Multifold.}

For Ricci shrinkers, it has been established in \cite{MR999661} and \cite{MR2846384} that, after performing a blow-up procedure around a singular point, the resulting tangent cone can be expressed as a finite union of flat metric cones whose links are spherical space forms.

In the setting of $G_2$-solitons, an analogous conclusion holds. Indeed, by combining the epsilon regularity estimate \thmref{Laplace4} and the Shi-type derivative estimates \thmref{Shi} with the volume doubling property \lemref{Volume doubling} and the non-collapsing property \lemref{noncollapse}, one may apply a similar blow-up argument to show that the tangent cone at any singular point is likewise a finite union of flat cones over spherical space forms.

\textbf{Step 4: Orbifold.}

We now establish the uniqueness of the tangent cone at an arbitrary singular point.

In \cite{MR1118730}*{Lemma 1.2}, the proof relies on the Ricci curvature volume comparison theorem in conjunction with a non-collapsing hypothesis. In the setting of a $G_2$-soliton, the Bakry-\'{E}mery volume comparison result \lemref{BEVCT} can be applied in place of the classical Ricci volume comparison theorem, by virtue of the bounds on
\[
f \text{ and } |\nabla f|
\]
specified in condition \ref{f condition} of Definition \ref{3 assumptions}. 

Consequently, by replicating the argument of \cite{MR1118730}*{Lemma 1.2} verbatim with this substitution, we obtain a neck analysis which implies that a punctured neighbourhood of \(q\) cannot decompose into more than one connected component. Therefore, \(U_q\setminus\{q\}\) has exactly one connected component. Equivalently, the tangent cone at any singular point is a single flat cone whose link is a spherical space form
\[
S^{n-1}/\Gamma.
\]

\textbf{Step 5: Absence of singular points.}


Let \(q\in\mathcal S\), and set
\[
U_q^\ast:=U_q\setminus\{q\}\subset D.
\]
By the smooth convergence on the regular region, the pulled-back positive
forms \(\varphi_i\) converge smoothly on compact subsets of \(U_q^\ast\)
to a positive \(3\)-form \(\varphi_\infty\), with
\[
g_{\varphi_\infty}=g_\infty.
\]
A positive \(3\)-form canonically determines an orientation, since the
transition maps between \(\varphi_\infty\)-adapted frames lie in
\(G_2\subset SO(7)\). Hence \(U_q^\ast\) is oriented. 

Assume that the orbifold analysis provides a covering
\[
\pi:B_\rho(0)\setminus\{0\}\longrightarrow U_q^\ast
\]
with deck group \(\Gamma\subset O(7)\), acting freely on \(S^6\).
The pullback of the orientation of \(U_q^\ast\) is invariant under every
deck transformation, and therefore \(\Gamma\subset SO(7)\).
However, every element of \(SO(7)\) has a \(+1\)-eigenvector and hence
fixes a point of \(S^6\). Since the action of \(\Gamma\) on \(S^6\) is
free, it follows that \(\Gamma=\{e\}\). In this case, the tangent cone at a putative singular point is Euclidean space \(\mathbb{R}^{7}\).

Assuming that \(q\) is a curvature singularity, we apply a standard blow-up argument to show that \(q\) is in fact a regular point. Choose a sequence of points \(x_i \to q\) and rescale the metrics about \(x_i\) so that the curvature norm at the corresponding base point in the limit is normalised to one. By assumption, the scalar curvature of the \(G_2\)-soliton is uniformly bounded on a sufficiently large ball. After rescaling and passing to a limit, the limiting metric has vanishing scalar curvature and is Ricci-flat. By the Bishop-Gromov volume comparison theorem, the asymptotic volume ratio satisfies
\[
\operatorname{AVR} \leq 1.
\]
On the other hand, since the tangent cone is Euclidean, the asymptotic volume ratio at \(q\) equals \(1\). Hence \(\operatorname{AVR} = 1\), and the limit is isometric to \(\mathbb{R}^{7}\).

This contradicts the assumption that \(q\) is a curvature singularity. Thus no singular point can arise, and the curvature is uniformly bounded in a neighborhood of \(q\). The local higher-order derivative estimates in \thmref{Shi} then yield bounds for all covariant derivatives of the curvature. Combined with the noncollapsing property and the triviality of the local group, this ensures smooth Cheeger-Gromov convergence across \(q\). The corresponding estimates for \(\varphi_i\) and \(f_i\) imply their smooth convergence as well; see \thmref{G2 smooth convergence}. This completes the proof of \thmref{half energy convergence}.


\section*{Funding}
K. Zheng is partially funded by NSFC grants No. 12571091.
He acknowledges support from the ICTP through the Associates Programme (2020-2025).
H. Z. Li is partially funded by NSFC grant  No. 12471058 and  the CAS Project for Young Scientists     in Basic Research (YSBR-001).

\section*{Declaration of competing interest}

The authors declare that they have no known competing financial
interests or personal relationships that could have appeared to
influence the work reported in this article.


\end{document}